\documentclass[12pt]{amsart}
\usepackage[margin=1in]{geometry}
\usepackage[mathscr]{eucal}
\usepackage[new]{old-arrows}
\usepackage{graphicx}
\usepackage{subcaption}
\usepackage{amssymb}
\usepackage{setspace}
\usepackage{hyperref}
\newtheorem{theorem}{Theorem}[section]
\newtheorem*{theorem*}{Theorem}
\newtheorem{proposition}[theorem]{Proposition}
\newtheorem{lemma}[theorem]{Lemma}
\theoremstyle{definition}
\newtheorem{definition}[theorem]{Definition}
\newtheorem{remark}[theorem]{Remark}

\newtheorem{warning}[theorem]{Caution}
\newtheorem{conceptualremark}[theorem]{Conceptual Remark}
\renewcommand\subset{\subseteq}

\renewcommand{\geq}{\geqslant}
\renewcommand{\leq}{\leqslant}
\renewcommand{\vec}[1]{\mathbf{#1}}
\newcommand{\surf}{\mathfrak{S}}
\newcommand{\Complex}{\mathbb{C}}
\newcommand{\idealtriang}{\lambda}

\newcommand{\Zinteger}{\mathbb{Z}}
\newcommand{\Qrational}{\mathbb{Q}}

\title[Quantum snakes]{Points of quantum $\mathrm{SL}_n$ coming from quantum snakes}
\author[D. C. Douglas]{Daniel C. Douglas}
\address{Department of Mathematics, Virginia Tech, 225 Stanger St, Blacksburg, VA 24061, USA}
\email{dcdouglas@vt.edu}
\date{\today}
\begin{document}
\begin{abstract}
We show that the quantized Fock-Goncharov monodromy matrices satisfy the relations of the quantum special linear group $\mathrm{SL}_n^q$.  The proof employs a quantum version of the technology of Fock-Goncharov called snakes.  This relationship between higher Teichm{\"u}ller theory and quantum group theory is integral to the construction of a $\mathrm{SL}_n$-quantum trace map for knots in thickened surfaces, partially developed in \cite{DouglasArxiv21}.  
\end{abstract}
\maketitle
For a finitely generated group $\Gamma$ and a suitable Lie group $G$, a primary object of study in low-dimensional geometry and topology is the $G$-character variety
\begin{equation*}
	\mathscr{R}_G(\Gamma) = \left\{  \rho : \Gamma \longrightarrow G  \right\} /\!\!/ \hspace{2.5pt} G
\end{equation*}
consisting of group homomorphisms $\rho$ from $\Gamma$ to $G$, considered up to conjugation.  Here, the quotient is taken in the algebraic geometric sense of Geometric Invariant Theory \cite{Mumford94}.  Character varieties can be explored using a wide variety of mathematical skill sets.  Some examples include the Higgs bundle approach of Hitchin \cite{HitchinTopology92}, the dynamics approach of Labourie \cite{LabourieInvent06}, and the representation theory approach of Fock-Goncharov~\cite{FockIHES06}.

In the case where the group $\Gamma = \pi_1(\surf)$ is the fundamental group of a punctured surface $\surf$ of finite topological type, and where the Lie group $G = \mathrm{SL}_n(\mathbb{C})$ is the special linear group, we are interested in studying a relationship between two competing deformation quantizations of the character variety $\mathscr{R}_{\mathrm{SL}_n(\mathbb{C})}(\surf) := \mathscr{R}_{\mathrm{SL}_n(\mathbb{C})}(\pi_1(\surf))$.  Here, a deformation quantization $\{ \mathscr{R}^q \}_q$ of a Poisson space $\mathscr{R}$ is a family of non-commutative algebras $\mathscr{R}^q$ parametrized by a nonzero complex parameter $q = e^{2 \pi i \hbar}$, such that the lack of commutativity in $\mathscr{R}^q$ is infinitesimally measured in the classical limit $\hbar \to 0$ by the Poisson bracket of the space $\mathscr{R}$.  In the case where  $\mathscr{R} = \mathscr{R}_{\mathrm{SL}_n(\mathbb{C})}(\surf)$ is the character variety, the bracket is provided by the Goldman Poisson structure on  $\mathscr{R}_{\mathrm{SL}_n(\mathbb{C})}(\surf)$ \cite{GoldmanAdvMath84, GoldmanInvent86}.  

The first quantization of the character variety is the $\mathrm{SL}_n(\mathbb{C})$-skein algebra $\mathscr{S}^q_n(\surf)$ of the surface $\surf$; see \cite{Turaev89, WittenCommMathPhys89, PrzytyckiBullPolishAcad91, BullockKnotTheory99, KuperbergCommMathPhys96, SikoraAlgGeomTop05}.  The skein algebra is motivated by the classical algebraic geometric approach to studying the character variety $\mathscr{R}_{\mathrm{SL}_n(\mathbb{C})}(\surf)$ via its algebra of regular functions $\mathbb{C}[\mathscr{R}_{\mathrm{SL}_n(\mathbb{C})}(\surf)]$.  An example of a regular function is the trace function $\mathrm{Tr}_\gamma : \mathscr{R}_{\mathrm{SL}_n(\mathbb{C})}(\surf) \to \mathbb{C}$ associated to a closed curve $\gamma \in \pi_1(\surf)$ sending a representation $\rho : \pi_1(\surf) \to \mathrm{SL}_n(\Complex)$ to the trace $\mathrm{Tr}(\rho(\gamma)) \in \mathbb{C}$ of the matrix $\rho(\gamma) \in \mathrm{SL}_n(\mathbb{C})$.  A theorem of Classical Invariant Theory, due to Procesi \cite{ProcesiAdvMath76}, implies that the trace functions $\mathrm{Tr}_\gamma$ generate the algebra of functions $\mathbb{C}[\mathscr{R}_{\mathrm{SL}_n(\mathbb{C})}(\surf)]$ as an algebra.  According to the philosophy of Turaev and Witten, quantizations of the character variety should be of a 3-dimensional nature.  Indeed,  knots (or links) $K$ in the thickened surface $\surf \times (0, 1)$ represent elements of the skein algebra $\mathscr{S}^q_n(\surf)$.  The skein algebra $\mathscr{S}^q_n(\surf)$ has the advantage of being natural, but can be difficult to study~directly.  

The second quantization of the $\mathrm{SL}_n(\mathbb{C})$-character variety is the Fock-Goncharov quantum space $\mathscr{T}_n^q(\surf)$; see \cite{Fock99TheorMathPhys, Kashaev98LettMathPhys, FockENS09}.  At the classical level, Fock-Goncharov \cite{FockIHES06}     introduced a framed version $\mathscr{R}_{\mathrm{PSL}_n(\mathbb{C})}(\surf)_{\mathrm{fr}}$ (called the $\mathcal{X}$-space) of the $\mathrm{PSL}_n(\mathbb{C})$-character variety, which, roughly speaking, consists of representations $\rho : \pi_1(\surf) \to \mathrm{PSL}_n(\mathbb{C})$ equipped  with additional linear algebraic data  attached to the punctures of $\surf$.  Associated to each ideal triangulation $\idealtriang$ of the punctured surface $\surf$ is a $\idealtriang$-coordinate chart $U_\idealtriang$ for $\mathscr{R}_{\mathrm{PSL}_n(\mathbb{C})}(\surf)_{\mathrm{fr}}$ parametrized by $N$ nonzero complex coordinates $X_1, X_2, \dots, X_N$ where the integer $N$ depends only on the topology of the surface $\surf$ and the rank of the Lie group $\mathrm{SL}_n(\mathbb{C})$.  These coordinates $X_i$ are computed by taking various generalized cross-ratios of configurations of $n$-dimensional flags attached to the punctures of $\surf$.  When written in terms of these coordinates $X_i$ the trace functions $\mathrm{Tr}_\gamma = \mathrm{Tr}_\gamma(X_i^{\pm 1/n})$ associated to closed curves $\gamma$ take the form of Laurent polynomials in $n$-roots of the variables $X_i$.  At the quantum level, there are $q$-deformed versions $X_i^q$ of these coordinates, which no longer commute but $q$-commute with each other.  The quantized character variety $\mathscr{T}_n^q(\surf)$ is obtained by gluing together quantum tori $\mathscr{T}_n^q(\sigma)$, including one for each triangulation $\sigma=\idealtriang$ consisting of Laurent polynomials in the quantized Fock-Goncharov coordinates  $X_i^q$.  The quantum character variety $\mathscr{T}_n^q(\surf)$ has the advantage of being easier to work with than the skein algebra $\mathscr{S}^q_n(\surf)$, however it is less intrinsic.  

We are interested in studying $q$-deformed versions $\mathrm{Tr}_\gamma^q$ of the trace functions $\mathrm{Tr}_\gamma$, associating to a closed curve $\gamma$ a Laurent polynomial in the quantized Fock-Goncharov coordinates $X_i^q$.  Turaev and Witten's philosophy leads us from the 2-dimensional setting of curves $\gamma$ on the surface $\surf$ to the $3$-dimensional setting of knots $K$ in the thickened surface $\surf \times (0, 1)$.  In the case of $\mathrm{SL}_2(\Complex)$, such a \textit{quantum trace map} was developed in \cite{BonahonGT11} as an injective algebra homomorphism 
 \begin{equation*}
 	\mathrm{Tr}^q(\idealtriang) :
	\mathscr{S}_2^q(\surf)
	\longhookrightarrow
	\mathscr{T}_2^q(\idealtriang)
\end{equation*}
from the $\mathrm{SL}_2(\Complex)$-skein algebra to (the $\lambda$-quantum torus of) the quantized $\mathrm{SL}_2(\Complex)$-character variety.  Their construction is ``by hand'', but  is implicitly related to the theory of the quantum group $\mathrm{U}_q(\mathfrak{sl}_2)$ or, more precisely, of its Hopf dual $\mathrm{SL}_2^q$; see \cite{Kassel95}.  Developing a quantum trace map for $\mathrm{SL}_n(\mathbb{C})$ requires a more conceptual approach, making  explicit this connection between higher Teichm\"{u}ller theory and quantum group theory.  In a companion paper \cite{DouglasArxiv21}, we make significant progress in this direction.  The goal of the present work is to establish a local building block result that is essential to understanding the quantum trace map more conceptually.  

Whereas the classical trace $\mathrm{Tr}_\gamma(\rho) \in \Complex$ is a number obtained by evaluating the trace of a $\mathrm{SL}_n(\Complex)$-monodromy $\rho(\gamma)$ taken along a curve $\gamma$ in the surface $\surf$, the quantum trace $\mathrm{Tr}_K(X_i^q) \in \mathscr{T}_n^q(\idealtriang)$ is a Laurent polynomial obtained from a quantum monodromy associated to a knot $K$ in the thickened surface $\surf \times (0, 1)$.  This quantum monodromy is essentially constructed by chopping the knot $K$ into little pieces, namely the components $C$ of $K \cap (\idealtriang_k \times (0,1))$ where the $\idealtriang_k$'s are the triangles   of the ideal triangulation $\idealtriang$, 
and associating to each piece $C$ a local quantum monodromy matrix $\vec{M}_C^q \in \mathrm{M}_n(\mathscr{T}_n^q(\idealtriang_k))$.  Here, the coefficients of the matrix $\vec{M}_C^q$ lie in   a local quantum torus $\mathscr{T}_n^q(\idealtriang_k)$ associated to the triangle $\idealtriang_k$, closely associated to the quantum torus $\mathscr{T}_n^q(\idealtriang)$.  

\begin{theorem*}
\label{thm:intro1}
	When $C$ is an arc on the corner of a triangle $\lambda_k$, the Fock-Goncharov quantum matrix $\vec{M}^q_C \in \mathrm{M}_n(\mathscr{T}_n^q(\idealtriang_k))$ is a $\mathscr{T}_n^q(\idealtriang_k)$-point of the quantum special linear group $\mathrm{SL}_n^q$.  In other words, each such matrix defines an algebra homomorphism
\begin{equation*}
	\varphi(\vec{M}^q_C) : \mathrm{SL}_n^q \longrightarrow \mathscr{T}_n^q(\idealtriang_k)
\end{equation*}
by the property that the $n^2$-many generators of the algebra $\mathrm{SL}_n^q$ are sent to the corresponding $n^2$-many entries of the matrix $\vec{M}^q_C$ (see  {\upshape\S\ref{sec:quantum-SLn}}).  
\end{theorem*}

See Theorem \ref{thm:first-theorem} (and \cite[Theorem 3.10]{DouglasThesis20}). 
Our proof uses a quantum version of the technology of Fock-Goncharov called snakes.   

 The main property of the quantum trace $\mathrm{Tr}_K(X_i^q) \in \mathscr{T}_n^q(\idealtriang)$ is its invariance under isotopy of the knot $K$.   This is equivalent to invariance under a handful of local, Reidemeister-like, moves in the thickened triangulated surface.  These topological moves are independent of $n$, and can be seen as the oriented versions of the moves depicted in  \cite[Figures $15$-$19$]{BonahonGT11}.  In particular, due to their local nature, these  moves have a purely algebraic formulation, as equalities involving $n \times n$ matrices with coefficients in the quantum torus.  Our main result is essentially equivalent to the algebraic formulation of one of these moves, specifically that depicted in \cite[Figure $17$]{BonahonGT11}; see also \cite[\S $6$]{DouglasArxiv21}.  

For an independent study of these same algebraic identities underlying the isotopy invariance of the quantum trace map, in the context of integrable systems, see \cite[Theorems 2.12 and 2.14]{ChekhovArxiv20} (which, in particular, reproduces our main result), motived in part by   \cite{SchraderInvent19,SchraderArxiv17}; see also \cite{Fock06, GekhtmanSelMathNewSer09,GoncharovArxiv19}.  Our work complements that of \cite{ChekhovArxiv20} by focusing attention on a single isotopy move, and conceptualizing the associated quantum phenomenon as arising naturally  from the underlying geometry.

	\section*{Acknowledgements}
		This work would not have been possible without the constant guidance and support of my Ph.D. advisor Francis Bonahon.  We thank Sasha Shapiro for informing us about related research and for enjoyable conversations.  We are also grateful to the referee for their helpful comments.    This work was partially supported by the U.S. National Science Foundation grants DMS-1406559 and DMS-1711297.

	\section{Fock-Goncharov snakes}	
	\label{sec:fock-goncharov-snakes}

We recall some of the classical (as opposed to the quantum) geometric theory of Fock-Goncharov \cite{FockIHES06}, underlying the quantum theory discussed later on; see also \cite{Fock07, FockAdvMath07}. This section is a condensed version of \cite[Chapter 2]{DouglasThesis20}.  For other references on Fock-Goncharov coordinates and snakes, see \cite{Hollands16LettMathPhys, GaiottoAnnHenriPoincare14,MartoneMathZ19}.  When $n=2$, these coordinates date back to Thurston's shearing coordinates for Teichm{\"u}ller space \cite{Thurston97}.  

Let $ n \in \Zinteger$, $n \geq 2$, and $V=\mathbb{C}^n$ be the standard $n$-dimensional complex vector space.

	\subsection{Generic configurations of flags and Fock-Goncharov invariants}

A \textit{(complete) flag} $E$ in $V$ is a collection of linear subspaces $E^{(a)} \subset V$ indexed by $0 \leq a \leq n$, satisfying the property that each subspace $E^{(a)}$ is properly contained in the subspace $E^{(a+1)}$.  In particular, $E^{(a)}$ is $a$-dimensional, $E^{(0)} = \left\{ 0 \right\}$, and $E^{(n)} = V$.  Denote the space of flags by $\mathrm{Flag}(V)$.

	\subsubsection{Generic triples and quadruples of flags}
	\label{subsec:generic-triples-of-flags}

There are at least two notions of genericity for a configuration of flags.  We will use just one of them, the Maximum Span Property; for a complementary notion, the Minimum Intersection Property, see \cite[\S 2.10]{DouglasThesis20}.  

\begin{definition}
\label{def:maximum-span-property}
	A flag tuple $(E_1, E_2, \dots, E_k) \in \mathrm{Flag}(V)^k$ satisfies the \textit{Maximum Span Property} if either of the following equivalent conditions are satisfied:  for all $ 0 \leq a_1, a_2, \dots, a_k \leq n$,
\begin{enumerate}
	\item  the sum $E_1^{(a_1)} + E_2^{(a_2)} + \cdots + E_k^{(a_k)} = E_1^{(a_1)} \oplus E_2^{(a_2)} \oplus \cdots \oplus E_k^{(a_k)}$ is direct for all $a_1 + a_2 + \cdots + a_k = n$, thus the sum is $V$;
	\item  the dimension formula $\mathrm{dim}(E_1^{(a_1)} + E_2^{(a_2)} + \cdots + E_k^{(a_k)}) = \mathrm{min}( a_1 + a_2 + \cdots + a_k, n )$.
\end{enumerate}
In the case $n=3$, such a flag triple $(E, F, G) \in \mathrm{Flag}(V)^3$ is called a \textit{maximum span flag triple}, and in the case $n=4$, such a  flag quadruple $(E, F, G, H)\in\mathrm{Flag}(V)^4$ is called a \textit{maximum span flag quadruple}.  
\end{definition}

	\subsubsection{Discrete triangle}
	\label{subsec:the-discrete-triangle}

The \emph{discrete $n$-triangle} $\Theta_n \subset \mathbb{Z}_{\geq 0}^3$ is defined by
\begin{equation*}
\Theta_n = \left\{ (a, b, c) \in \Zinteger^3_{\geq 0} ; \quad a+b+c=n \right\}.
\end{equation*}
See Figure \ref{fig:n-discrete-triangle}.  The \textit{interior} $\mathrm{int}(\Theta_n) \subset \Theta_n$ of the discrete triangle is defined by
\begin{equation*}
\mathrm{int}(\Theta_n) = \left\{ (a, b, c) \in \Theta_n ; \quad a, b, c > 0 \right\}.
\end{equation*}

An element $ \nu \in \Theta_n $ is called a \textit{vertex} of $\Theta_n$.  Put $\Gamma(\Theta_n) = \left\{ (n, 0, 0), ( 0, n, 0 ), ( 0, 0, n ) \right\} \subset \Theta_n$.  An element $\nu \in \Gamma(\Theta_n)$ is called a \textit{corner vertex} of $ \Theta_n $.

	\subsubsection{Fock-Goncharov triangle and edge invariants}
	\label{subsec:fock-goncharov-triangle-and-edge-invariants}

For a maximum span triple of flags $(E, F, G) \in \mathrm{Flag}(V)^3$, Fock and Goncharov assigned to each interior point $(a, b, c) \in \mathrm{int}(\Theta_n)$ a \textit{triangle invariant} $\tau_{abc}(E, F, G) \in \Complex - \left\{ 0 \right\}$, defined by the formula
\begin{equation*}
	\tau_{abc}(E, F, G) = 
	\frac{e^{(a-1)} \wedge f^{ (b+1)} \wedge g^{(c)}}
	{e^{(a+1)} \wedge f^{ (b-1)} \wedge g^{(c)}}
	\frac{ e^{(a)} \wedge f^{ (b-1)} \wedge g^{(c+1)} }
	{ e^{(a)} \wedge f^{( b+1)} \wedge g^{(c-1)} }
	\frac{ e^{(a+1)} \wedge f^{ (b)} \wedge g^{(c-1)} }
	{ e^{(a-1)} \wedge f^{ (b)} \wedge g^{(c+1)} }
	\quad \in \Complex - \left\{ 0 \right\}.
\end{equation*}
Here, $e^{(a^\prime)}$, $f^{(b^\prime)}$, and $g^{(c^\prime)}$ are choices of generators for the exterior powers $\Lambda^{a^\prime}(E^{(a^\prime)}) \subseteq \Lambda^{a^\prime}( V )$,  $\Lambda^{b^\prime}(F^{(b^\prime)}) \subseteq \Lambda^{b^\prime}( V )$, and  $\Lambda^{c^\prime}(G^{(c^\prime)}) \subseteq \Lambda^{c^\prime}( V )$, respectively.  The Maximum Span Property ensures that each wedge product $ e^{ ( a^\prime ) } \wedge f^{ ( b^\prime ) } \wedge g^{ ( c^\prime ) }$ is nonzero in $ \Lambda^{a^\prime + b^\prime + c^\prime}(V) = \Lambda^n ( V ) \cong \mathbb{C}$.  Since there are the same number of terms in the numerator as the denominator, $\tau_{abc}(E,F,G)$ is independent of this choice of isomorphism $\Lambda^n(V)\cong\mathbb{C}$.  Since each generator $e^{(a^\prime)}$, $f^{(b^\prime)}$, $g^{(c^\prime)}$ appears exactly once in the numerator and denominator, $ \tau_{ abc }( E, F, G ) $ is independent of the choices of these generators.  

The six numerators and denominators appearing in the expression defining $\tau_{abc}(E, F, G)$ can be visualized as the vertices of a hexagon in $\Theta_n$ centered at $(a, b, c)$; see Figure \ref{fig:n-discrete-triangle}.

Similarly, for a maximum span quadruple of flags $(E, G, F, F^\prime) \in \mathrm{Flag}(V)^4$, Fock and Goncharov assigned to each integer $1 \leq j \leq n-1$ an \textit{edge invariant} $\epsilon_j(E, G, F, F^\prime)$ by
\begin{equation*}
\label{eq:def-of-edge-invariant}
	\epsilon_j(E, G, F, F^\prime) = 
	-
	\frac{ e^{(j)} \wedge g^{(n-j-1)} \wedge f^{ (1)} }
	{ e^{(j)} \wedge g^{(n-j-1)} \wedge f^{ \prime(1)} }
\,
	\frac{ e^{(j-1)} \wedge g^{(n-j)} \wedge f^{\prime(1)} }
	{ e^{(j-1)} \wedge g^{(n-j)} \wedge f^{ (1)} }
	\quad
	\in \Complex - \left\{ 0 \right\}.
\end{equation*}

The four numerators and denominators appearing in the expression defining $\epsilon_j(E, G, F, F^\prime)$ can be visualized as the vertices of a square, which crosses the ``common edge'' between two ``adjacent'' discrete triangles $\Theta_n(G, F, E)$ and $\Theta_n(E, F^\prime, G)$; see Figure \ref{fig:edge-invariants}.  

\begin{figure}[htb]
\centering
\includegraphics[scale=.60]{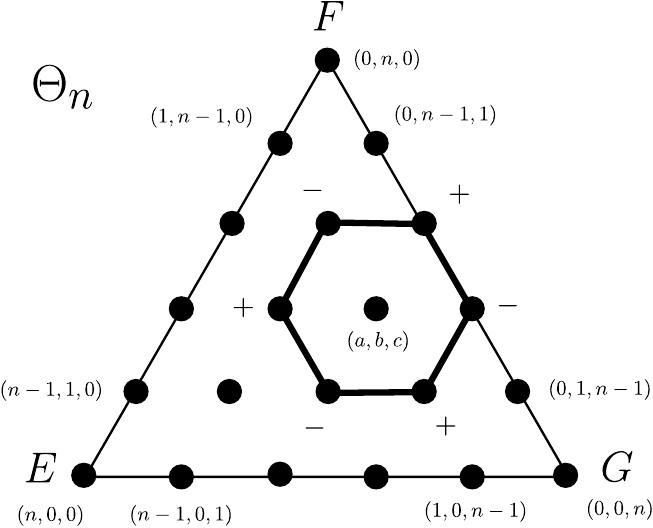}
\caption{\small Discrete triangle, and triangle invariants for a generic flag triple}
\label{fig:n-discrete-triangle}    
\end{figure} 

\begin{figure}[htb]
	\centering
	\includegraphics[scale=.55]{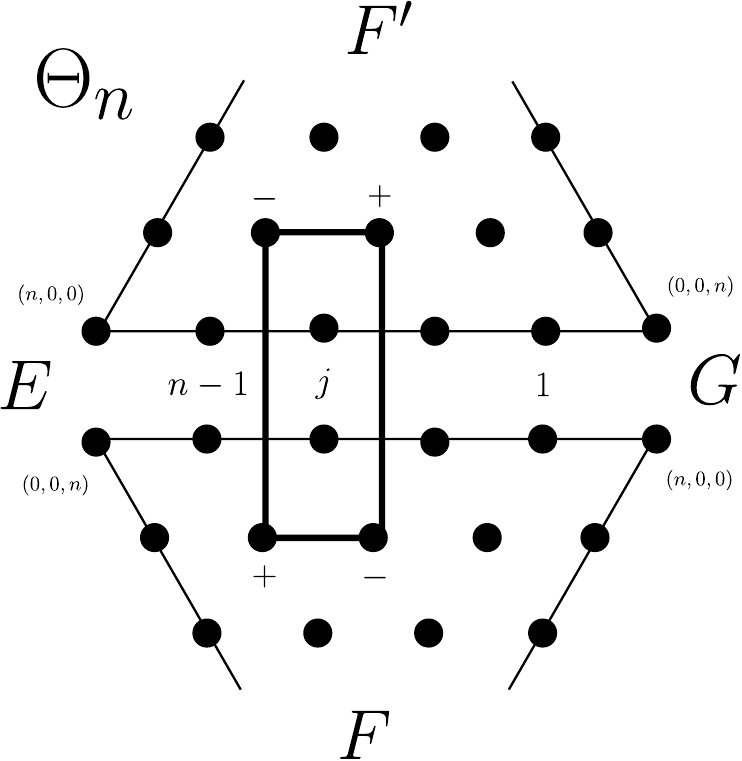}
	\caption{\small Edge invariants for a generic flag quadruple}
	\label{fig:edge-invariants}
\end{figure}

	\subsubsection{Action of $ \mathrm{PGL}(V) $ on generic flag triples}
	\label{subsec:the-action-of-PGLnC-on-the-space-of-maximum-span-flag-triples}

The action of the general linear group $\mathrm{GL}(V)$ on the vector space $V$ induces an action of the projective linear group $\mathrm{PGL}(V)$ on the space $\mathrm{Flag}(V)$ of flags.  The corresponding diagonal action of $\mathrm{PGL}(V)$ on $\mathrm{Flag}(V)^n$ restricts to generic configurations of flags.  By an elementary argument, for $n=2$ this diagonal action on generic flag pairs $(E,F)$ has a single orbit in $\mathrm{Flag}(V)^2$.  

\begin{theorem}[Fock-Goncharov]
\label{thm:Fock-and-Goncharov's-theorem}
Two maximum span flag triples $(E, F, G)$ and $(E^\prime, F^\prime, G^\prime)$ have the same triangle invariants, namely $\tau_{abc}(E, F, G) = \tau_{abc}(E^\prime, F^\prime, G^\prime) \in \mathbb{C}-\left\{0\right\}$ for every $(a, b, c) \in \mathrm{int}(\Theta_n)$, if and only if there exists $\varphi \in \mathrm{PGL}(V)$ such that $(\varphi E, \varphi F, \varphi G) = (E^\prime, F^\prime, G^\prime) \in \mathrm{Flag}(V)^3$.  

Conversely, for each choice of nonzero complex numbers $x_{abc} \in \Complex - \left\{ 0 \right\}$ assigned to the interior points $(a, b, c) \in \mathrm{int}(\Theta_n)$, there exists a maximum span flag triple $(E, F, G)$ such that $\tau_{abc}(E, F, G) = x_{abc}$ for all $(a, b, c)$.  
\end{theorem}

\begin{proof}
	See \cite[\S9]{FockIHES06}.  The proof uses the concept of snakes, due to Fock and Goncharov.  For a sketch of the proof and some examples, see \cite[\S 2.19]{DouglasThesis20}.  
\end{proof}

	\subsection{Snakes and projective bases}

	\subsubsection{Snakes}
	\label{subsec:notation-and-definitions}

Snakes are combinatorial objects associated to the $ ( n-1 ) $-discrete triangle $ \Theta_{ n-1 } $ (\S \ref{subsec:the-discrete-triangle}).  In contrast to $\Theta_n$, we denote the coordinates of a vertex $\nu \in \Theta_{n-1}$ by $\nu=(\alpha, \beta, \gamma)$ corresponding to solutions $\alpha+\beta+\gamma=n-1$ for $\alpha, \beta, \gamma \in \Zinteger_{\geq 0}$.  

\begin{definition}
A \textit{snake-head} $ \eta $ is a fixed corner vertex of the $ ( n-1 ) $-discrete triangle
\begin{equation*}
	\eta \quad\in \left\{ (n-1,0,0), (0, n-1, 0), (0, 0, n-1) \right\} = \Gamma(\Theta_{ n-1 } ) \quad\subseteq \Theta_{ n-1 }.
\end{equation*}
\end{definition}

\begin{remark}
\label{rem:snake-heads}
In a moment, we will define a snake.  The most general definition involves choosing a snake-head $\eta \in \Gamma(\Theta_{n-1})$.  For simplicity, we define a snake only in the case $\eta=(n-1, 0, 0)$.  The definition for other choices of snake-heads follows by triangular symmetry.  We will usually take $\eta = (n-1,0,0)$ and will alert the reader if otherwise.  
\end{remark}

\begin{definition}
\label{def:snakes}
	A \textit{left $ n $-snake} (for the snake-head $\eta=(n-1,0,0)\in \Gamma(\Theta_{n-1})$), or just \textit{snake}, $\sigma$ is an ordered list $ \sigma = ( \sigma_1, \sigma_2, \dots, \sigma_n ) \in (\Theta_{ n-1 })^n$ of $n$-many vertices $\sigma_k=(\alpha_k, \beta_k, \gamma_k)$ in the discrete triangle $ \Theta_{ n-1 } $, called \textit{snake-vertices}, satisfying
		\begin{equation*}
			\alpha_k = k-1, \quad\quad
			\beta_k \geq \beta_{k+1}, \quad\quad
			\gamma_k \geq \gamma_{k+1}
			\quad\quad  \left(k=1, 2, \dots, n\right).
		\end{equation*}
\end{definition}

See Figure \ref{fig:snakes-theta-n-theta-n-1}.  On the right hand side, we show a snake $\sigma=(\sigma_k)_k$  in the case $ n=5 $ (where we have taken some artistic license to assist the reader in locating the snake's head and tail; in \S \ref{sec:snake-move-algebras}, we will find it useful to split the snake in half down its length, as illustrated in Figure \ref{fig:quantum-snake-sweep-example}).  On the left hand side, we show how the snake-vertices $ \sigma_k \in \Theta_{ n-1 } $ can be pictured as small upward-facing triangles $\Delta$ in the $ n $-discrete triangle~$ \Theta_n $.

\begin{figure}[htb]
\centering
\includegraphics[width=.92\textwidth]{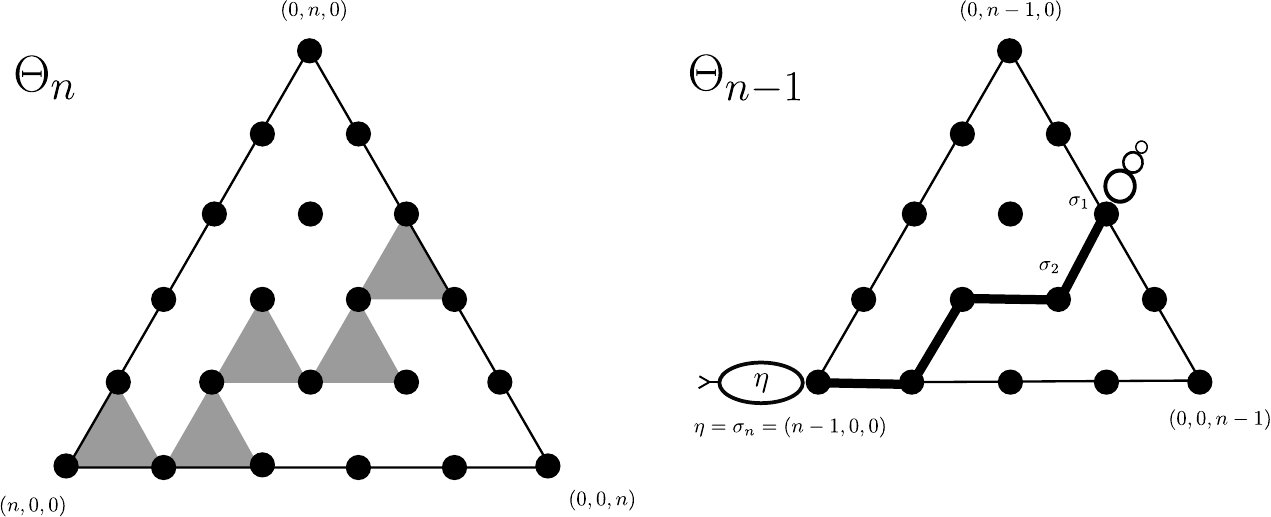}
\caption{\small Snake}     
\label{fig:snakes-theta-n-theta-n-1}
\end{figure}

	\subsubsection{Line decomposition of $V^*$ associated to a generic triple of flags and a snake}
	\label{subsec:line-decomposition-of-Cndual-associated-to-a-triple-of-flags-and-a-snake}

Let $V^*=\left\{\text{linear map }V \to \mathbb{C}\right\}$.  For a subspace $W \subset V$, define $W^\perp = \{ u\in V^*;u(w)=0\text{ for all }w\in W\}$.  A \textit{line} in a vector space $V^\prime$ is a $1$-dimensional subspace.

Fix a maximum span triple $(E, F, G) \in \mathrm{Flag}(V)^3$.  For any vertex $\nu = (\alpha, \beta, \gamma) \in \Theta_{n-1}$,
\begin{equation*}
	\mathrm{dim}\left( ( E^{(\alpha)} \oplus F^{(\beta)} \oplus G^{(\gamma)} )^\perp \right)
	= 1
\end{equation*}
by the Maximum Span Property, since $\alpha+\beta+\gamma = n-1$.  Consequently,  the subspace
\begin{equation*}
	L_{(\alpha, \beta, \gamma)} := \left( E^{(\alpha)} \oplus F^{(\beta)} \oplus G^{(\gamma)} \right)^\perp
	\quad  \subseteq  V^*
\end{equation*}
is a line  for all vertices $( \alpha, \beta, \gamma ) \in \Theta_{n-1}$.

If in addition we are given a snake $\sigma=(\sigma_k)_k$, then we may consider the $n$-many lines 
\begin{equation*}
	L_{\sigma_k} = L_{ (\alpha_k, \beta_k, \gamma_k) } 
	\quad \subseteq V^*
	\quad\quad  \left( k=1, \dots, n \right)
\end{equation*} 
where $\sigma_k = (\alpha_k, \beta_k, \gamma_k) \in \Theta_{n-1}$.  By genericity, we obtain a direct sum line decomposition
\begin{equation*}
	V^* = \bigoplus_{k=1}^n L_{\sigma_k}.
\end{equation*}

	\subsubsection{Projective basis of $V^*$ associated to a generic triple of flags and a snake}
	\label{subsec:projective-basis-of-Cndual-associated-to-a-triple-of-flags-and-a-snake}
	
Given a generic flag triple $(E,F,G)$ and a snake $\sigma$, Fock-Goncharov construct in addition a projective basis $[\mathscr{U}]$ of $V^*$ adapted to the associated line decomposition.  Here, $\mathscr{U}=\left\{u_1, u_2, \dots, u_n\right\}$ is a linear basis of $V^*$ such that $u_k \in L_{\sigma_k}$ for all $k$, and the \textit{projective basis} $[\mathscr{U}]$ is the equivalence class of $\mathscr{U}$ under the relation $\left\{u_1, u_2, \dots, u_n\right\} \sim \left\{\lambda u_1, \lambda u_2, \dots, \lambda u_n\right\}$ for all $\lambda \neq 0$.  

\begin{figure}[htb]
	\centering
	\includegraphics[scale=.50]{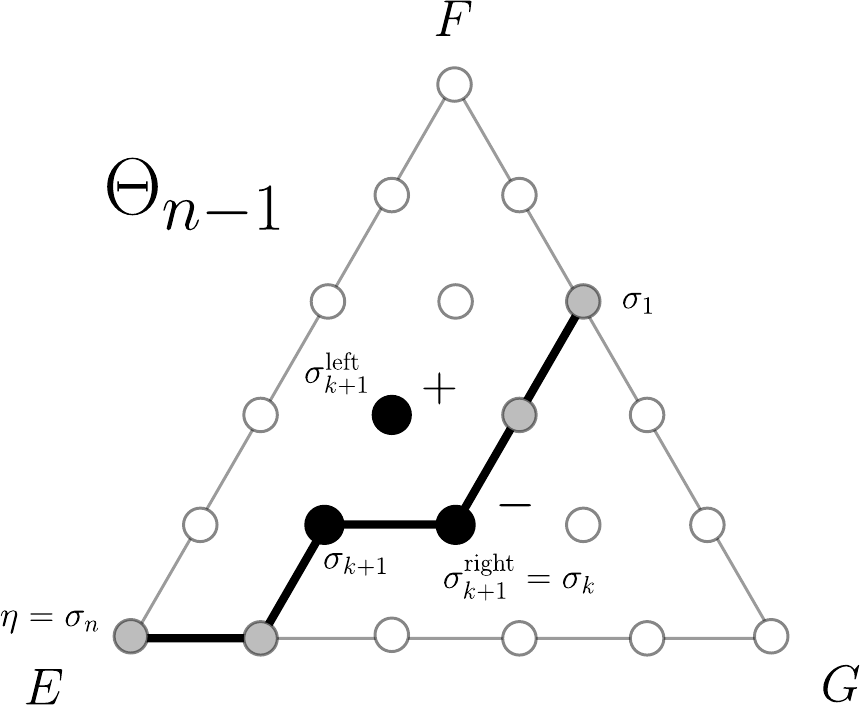}
	\caption{\small Three coplanar lines involved in the definition of a projective basis.  For the meaning of the $+$ and $-$ signs, see Definition \ref{def:defining-the-projective-basis-by-induction}. }
	\label{fig:defining-projective-basis}
\end{figure}

Put $\sigma_k = (\alpha_k, \beta_k, \gamma_k)$.  We begin by choosing a covector $u_n$ in the line $L_{\sigma_n}  \subseteq V^*$, called a \textit{normalization}.  Having defined covectors $u_n, u_{n-1}, \dots, u_{k+1}$, we will define a~covector 
\begin{equation*}
	u_k \quad \in L_{\sigma_k} = \left(E^{(\alpha_k)} \oplus F^{(\beta_k)} \oplus G^{(\gamma_k)}\right)^\perp 
	\quad \subseteq V^*.
\end{equation*}  
By the definition of snakes, we see that given $\sigma_{k+1}$ there are only two possibilities for $\sigma_k$, denoted $\sigma_{k+1}^\mathrm{left}$ and $\sigma_{k+1}^\mathrm{right}$:
\begin{align*}
\label{eq:def-of-left-projective-basis1}
	\sigma_{k+1}^\mathrm{left} &= (\alpha^\mathrm{left}_{k+1}, \beta^\mathrm{left}_{k+1}, \gamma^\mathrm{left}_{k+1}), 
	&  \alpha^\mathrm{left}_{k+1} &= k-1,
	&  \beta^\mathrm{left}_{k+1} &= \beta_{k+1} + 1,
	&  \gamma^\mathrm{left}_{k+1} &= \gamma_{k+1};
\\
	\sigma_{k+1}^\mathrm{right} &= (\alpha^\mathrm{right}_{k+1}, \beta^\mathrm{right}_{k+1}, \gamma^\mathrm{right}_{k+1}), 
	&  \alpha^\mathrm{right}_{k+1} &= k-1,
	&  \beta^\mathrm{right}_{k+1} &= \beta_{k+1},
	&  \gamma^\mathrm{right}_{k+1} &= \gamma_{k+1} + 1.
\end{align*}
See Figure \ref{fig:defining-projective-basis}, where in the example $\sigma_k = \sigma_{k+1}^\mathrm{right}$. Thus, the lines $L_{\sigma^\mathrm{left}_{k+1}}$ and $L_{\sigma^\mathrm{right}_{k+1}}$ can be~written

\begin{equation*}
\begin{split}
	L_{\sigma_{k+1}^\mathrm{left}} 
	&= \left( E^{(k-1)} \oplus F^{( \beta_{k+1} + 1 )}
	\oplus G^{( \gamma_{k+1} )} \right)^\perp 
	\quad \subseteq V^*;
\\	L_{\sigma_{k+1}^\mathrm{right}} 
	&= \left( E^{(k-1)} \oplus F^{( \beta_{k+1} )}
	\oplus G^{( \gamma_{k+1} +1 )} \right)^\perp 
	\quad \subseteq V^*.
\end{split}
\end{equation*}

It follows by the Maximum Span Property that the three lines $L_{\sigma_{k+1}}$, $L_{\sigma_{k+1}^\mathrm{left}}$, $L_{\sigma_{k+1}^\mathrm{right}}$ in $ V^* $ are distinct and coplanar.  Specifically, they lie in the plane
\begin{equation*}
	\left(
	E^{(k-1)} \oplus F^{(\beta_{k+1})} \oplus G^{(\gamma_{k+1})}
	\right)^\perp 
	\quad  \subseteq V^*
\end{equation*}
which is indeed 2-dimensional, since $ (k-1)+ \beta_{k+1} + \gamma_{k+1} = (n-1)-1$, as $\alpha_{k+1}=k$.  Thus, if $u_{k+1}$ is a nonzero covector in the line $L_{\sigma_{k+1}}$, then there exist unique nonzero covectors $u_{k+1}^\mathrm{left}$ and $u_{k+1}^\mathrm{right}$ in the lines $L_{\sigma_{k+1}^\mathrm{left}}$ and $L_{\sigma_{k+1}^\mathrm{right}}$, respectively, such that
\begin{equation*}
	u_{k+1} + u_{k+1}^\mathrm{left} + u_{k+1}^\mathrm{right} = 0 
	\quad \in V^*.
\end{equation*}

\begin{definition}
\label{def:defining-the-projective-basis-by-induction}

Having chosen a normalization $u_n \in L_{\sigma_n} = L_{(n-1,0,0)}$ and having inductively defined $u_{k^\prime} \in L_{\sigma_{k^\prime}}$ for $k^\prime=n,n-1,\dots,k+1$, define $u_k \in L_{\sigma_k}$ by the~formula:
\begin{enumerate}
	\item  if $\sigma_k = \sigma_{k+1}^\mathrm{left}$, put $u_k = + u_{k+1}^\mathrm{left} \in L_{\sigma_{k+1}^\mathrm{left}}$;
	\item  if $\sigma_k = \sigma_{k+1}^\mathrm{right} $, put $u_k = - u_{k+1}^\mathrm{right} \in L_{\sigma_{k+1}^\mathrm{right}}$.
\end{enumerate}
See Figure \ref{fig:defining-projective-basis}, which falls into case (2).  Note if the initial normalization $u_n$ is replaced by $\lambda u_n$ for some scalar $\lambda \neq 0$, then $u_k$ is replaced by $\lambda u_k$ for all $1 \leq k \leq n$.  Thus this process produces a projective basis $[\mathscr{U}] = [\left\{ u_1, u_2, \dots, u_n \right\}]$ of $V^*$, as desired.  We call $\mathscr{U}=\left\{ u_1, u_2, \dots, u_n \right\}$ the \textit{normalized projective basis} for $V^*$ depending on the normalization $u_n \in L_{\sigma_n}$.
\end{definition}

	\subsection{Snake moves}
	\label{subsec:elementary-snake-moves}

	\subsubsection{Elementary matrices}
	\label{sec:theorem-3}

	Let $\mathscr{A}$ be a commutative algebra with $1$, such as $\mathscr{A} = \Complex$.  Let $X^{1/n}, Z^{1/n} \in \mathscr{A}$, and put $X=(X^{1/n})^n$ and $Z=(Z^{1/n})^n$.  Let $\mathrm{M}_n(\mathscr{A})$ (resp. $\mathrm{SL}_n(\mathscr{A})$) denote the ring of $n \times n$ matrices (resp. having determinant equal to 1) over $\mathscr{A}$ (see also \S \ref{sssec:matrix-algebras}).  

For $k = 1, 2, \dots, n-1$ define the \textit{$k$-th left-elementary matrix} $\vec{S}^\mathrm{left}_k (X) \in \mathrm{SL}_n(\mathscr{A})$ by
\begin{equation*}
\label{eq:upper-triangular-matrix}
	\vec{S}^\mathrm{left}_k(X) = X^{-(k-1)/n}
	\left(\begin{smallmatrix}
		X&&&&&&&\\
		&\ddots&&&&&&\\
		&&X&&&&&\\
		&&&1&1&&&\\
		&&&&1&&&\\
		&&&&&1&&\\
		&&&&&&\ddots&\\
		&&&&&&&1
	\end{smallmatrix}\right)
	\quad
	\in  \mathrm{SL}_n(\mathscr{A})
	\quad\quad  \left(X \text{ appears } k-1 \text{ times}\right),
\end{equation*}
and define the \textit{$k$-th right-elementary matrix} $\vec{S}^\mathrm{right}_k (X) \in \mathrm{SL}_n(\mathscr{A})$ by
\begin{equation*}
\label{eq:lower-triang-matrix}
	\vec{S}^\mathrm{right}_k(X) =
	X^{+(k-1)/n}
	\left(\begin{smallmatrix}
		1&&&&&&&\\
		&\ddots&&&&&&\\
		&&1&&&&&\\
		&&&1&&&\\
		&&&1&1&&&\\
		&&&&&X^{-1}&&\\
		&&&&&&\ddots&\\
		&&&&&&&X^{-1}
	\end{smallmatrix}\right)
	\quad	
	 \in \mathrm{SL}_n(\mathscr{A})
	\quad\quad  \left(X \text{ appears } k-1 \text{ times}\right).  
\end{equation*}
Note that $\vec{S}^\mathrm{left}_1(X)$ and $\vec{S}^\mathrm{right}_1(X)$ do not, in fact, involve the variable $X$, and so we will denote these matrices simply by $\vec{S}^\mathrm{left}_1$ and $\vec{S}^\mathrm{right}_1$, respectively.  

For $j = 1, 2, \dots, n-1$ define the \textit{$j$-th edge-elementary matrix} $\vec{S}^\mathrm{edge}_j(Z) \in \mathrm{SL}_n(\mathscr{A})$ by 
\begin{equation*}
\label{eq:shearing-matrix}
	\vec{S}^\mathrm{edge}_j(Z) = Z^{-j/n}
	\left(\begin{smallmatrix}
		Z&&&&&&&\\
		&Z&&&&&&\\
		&&\ddots&&&&&\\
		&&&Z&&&&\\
		&&&&1&&&\\
		&&&&&1&&\\
		&&&&&&\ddots&\\
		&&&&&&&1
	\end{smallmatrix}\right)
	\quad  \in \mathrm{SL}_n(\mathscr{A})
	\quad\quad  \left(Z \text{ appears } j \text{ times}\right).
\end{equation*} 

Lastly, define the \textit{clockwise U-turn matrix} $\vec{U} $ in $ \mathrm{SL}_n(\mathbb{C})$ by 
\begin{equation*}
\label{eq:U-turn-matrix}
	\vec{U}
	=  \left(\begin{smallmatrix}
		&&&&\\
		&&&&(-1)^{n-1}\\
		&&&\reflectbox{$\ddots$}&\\
		&&+1&&\\
		&-1&&&\\
		+1&&&&
	\end{smallmatrix}\right)
	\quad  \in  \mathrm{SL}_n(\mathbb{C}).
\end{equation*}

	\subsubsection{Adjacent snake pairs}
	\label{sssec:adjacent-snake-pairs}

\begin{definition}
\label{def:adjacent-snakes}
	We say that an ordered pair $(\sigma, \sigma^\prime)$ of snakes $\sigma$ and $\sigma^\prime$ forms an \textit{adjacent pair of snakes} if the pair $(\sigma, \sigma^\prime)$ satisfies either of the following conditions:
\begin{enumerate}
	\item  for some $ 2 \leq k \leq n-1 $, 
\begin{enumerate}
	\item  $ \sigma_{j} = \sigma^\prime_{j} $
	\quad\quad
	\quad\quad
	\quad\quad
	\quad\quad
	\quad\quad
	\quad\quad
	 $( 1 \leq j \leq k-1 $, \quad $ k+1 \leq j \leq n)$,
	\item  $\sigma_k = \sigma_{ k+1 }^\mathrm{right} \,\, (=\sigma_{k+1}^{\prime \mathrm{right}}), $ \quad and \quad
  $ \sigma^\prime_k = \sigma_{ k+1 }^\mathrm{left} \,\, (=\sigma_{k+1}^{\prime \mathrm{left}}) $,
\end{enumerate} 
in which case $ ( \sigma, \sigma^\prime ) $ is called an adjacent pair of \textit{diamond-type}, see Figure \ref{fig:diamond-move};
\item  
	\begin{enumerate}
	\item  $ \sigma_{j} = \sigma^\prime_{j}$ 
	\quad\quad
	\quad\quad
	\quad\quad
	\quad\quad
	\quad\quad
	\quad\quad
	$(2 \leq j \leq n)$,
	\item  $\sigma_1 = \sigma_{ 2 }^\mathrm{right}  \,\, (=\sigma_2^{\prime \mathrm{right}})$, \quad and \quad
  $ \sigma^\prime_1 = \sigma_{ 2 }^\mathrm{left} \,\, (=\sigma_2^{\prime \mathrm{left}})$,
\end{enumerate} 
in which case $ ( \sigma, \sigma^\prime ) $ is called an adjacent pair of \textit{tail-type}, see Figure \ref{fig:tail-move}.
\end{enumerate}
\end{definition}

	\subsubsection{Diamond and tail moves}
	\label{sssec:diamond-and-tail-moves}

Let $ ( \sigma, \sigma^\prime ) $ be an adjacent pair of snakes of diamond-type, as shown in Figure \ref{fig:diamond-move}.  

Consider the snake-vertices $ \sigma_{ k+1 }$ $(= \sigma^\prime_{ k+1 }) $, $\sigma_k$, $\sigma^\prime_k$, and $\sigma_{k-1}$ $(= \sigma^\prime_{k-1})$.  One checks~that 
\begin{equation*}
\label{eq:coordinates-of-the-downward-facing-triangle}
	\alpha_k = \alpha^\prime_k = k-1,\quad\quad
	\beta_k^\prime = \beta_{k-1}
	=  \beta_{k+1} + 1, \quad\quad
	\gamma_{k} = \gamma_{k-1}
	=  \gamma_{k+1} + 1. 
\end{equation*}
Taken together, these three coordinates form a vertex
\begin{equation*}
	(a, b, c) \quad = (k-1, \beta_{k+1} + 1, \gamma_{k+1} + 1) \quad \in \mathrm{int}(\Theta_n)
\end{equation*}
in the interior of the $n$-discrete triangle $\Theta_n$ (not $\Theta_{n-1}$), since $(k-1)+(\beta_{k+1}+1)+(\gamma_{k+1}+1)=(\alpha_{k+1}+\beta_{k+1}+\gamma_{k+1})+1=n$.  The coordinates of this internal vertex $(a,b,c)$ can also be thought of as delineating the boundary of a small downward-facing triangle $\nabla$ in the discrete triangle $\Theta_{ n-1 }$, whose three vertices are $\sigma_k, \sigma^\prime_k$, $\sigma_{k-1}$ (Figure \ref{fig:diamond-move}).  Put $X_{abc} = \tau_{abc}(E, F, G) \in \Complex - \left\{ 0 \right\}$, namely $X_{abc}$ is the Fock-Goncharov triangle invariant (\S \ref{subsec:fock-goncharov-triangle-and-edge-invariants}) associated to the generic flag triple $(E, F, G)$ and the internal vertex $(a,b,c) \in \mathrm{int}(\Theta_n)$.  

\begin{figure}[htb]
	\centering
	\includegraphics[scale=.46]{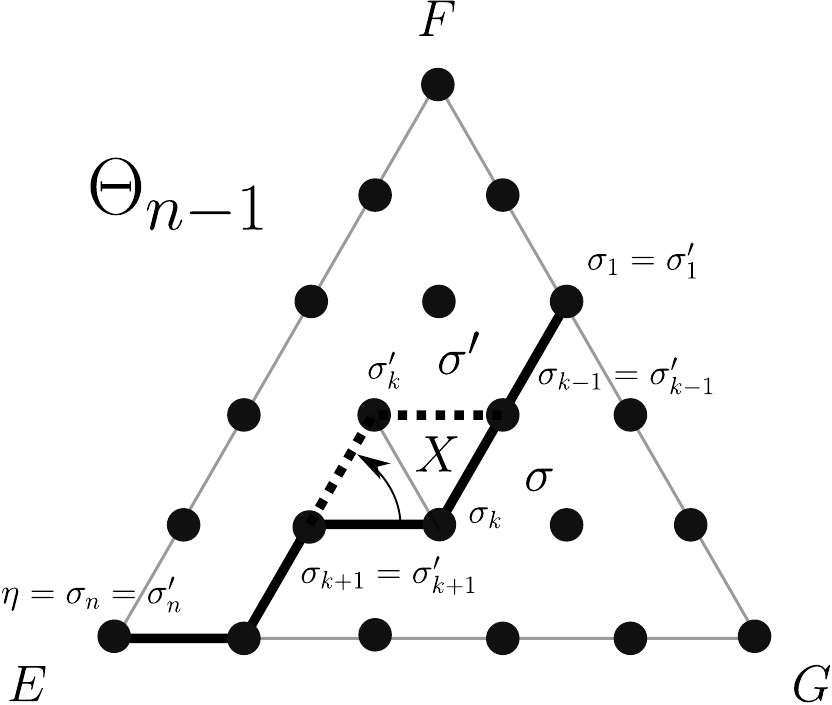}
	\caption{\small Diamond move}
	\label{fig:diamond-move}
\end{figure}

\begin{figure}[htb]
	\centering
	\includegraphics[scale=.46]{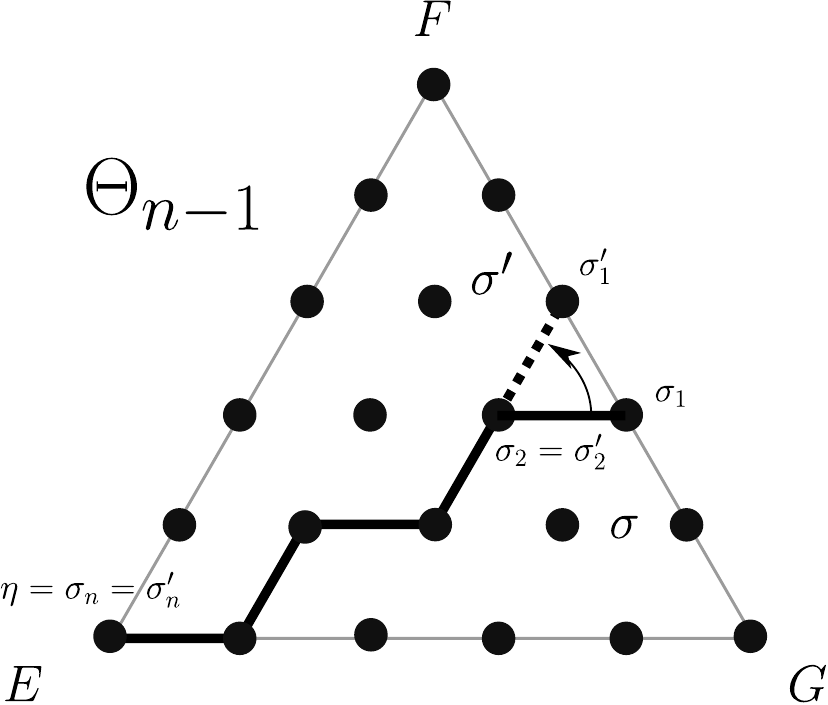}
	\caption{\small Tail move}
	\label{fig:tail-move}
\end{figure}

The proposition below is the main ingredient going into the proof of Theorem \ref{thm:Fock-and-Goncharov's-theorem}.  First, we set our conventions for change of basis matrices for bases of $V^*$.  

Given any basis $ \mathscr{U} = \left\{ u_1, u_2, \dots, u_n \right\} $ of $ V^* $, and given a covector $ u $ in $ V^* $, the \textit{coordinate covector} $ [u]_\mathscr{U} $ of the covector $ u $ with respect to the basis $ \mathscr{U} $ is the unique row matrix $ [u]_\mathscr{U} = \left(\begin{smallmatrix} y_1 & y_2 & \cdots & y_n \end{smallmatrix}\right) $ in $ \mathrm{M}_{1, n}(\Complex) $ such that $ u = \sum_{ i=1 }^n y_i u_i$.  If $ \mathscr{U}^\prime = \left\{ u^\prime_1, u^\prime_2, \dots, u^\prime_n \right\} $ is another basis for $ V^* $, then the \textit{change of basis matrix} $ \vec{B}_{\mathscr{U} \rightarrow \mathscr{U}^\prime} $ going from the basis $ \mathscr{U} $ to the basis $ \mathscr{U}^\prime $ is the unique invertible matrix in $ \mathrm{GL}_n(\Complex) \subset \mathrm{M}_n(\Complex)$ satisfying
\begin{equation*}
	[ u ]_{ \mathscr{U} } \vec{B}_{\mathscr{U} \rightarrow \mathscr{U}^\prime} = [ u ]_{\mathscr{U}^\prime}
	\quad
	\in \mathrm{M}_{1, n}(\Complex)
	\quad\quad\quad
	\left( u \in V^* \right).
\end{equation*}
Change of basis matrices satisfy the property
\begin{equation*}
\label{eq:composition-of-covector-chnage-of-bases}
	\vec{B}_{\mathscr{U} \rightarrow \mathscr{U}^{\prime \prime }}
	=  
	\vec{B}_{\mathscr{U} \rightarrow \mathscr{U}^{\prime }}
	\vec{B}_{\mathscr{U}^\prime \rightarrow \mathscr{U}^{\prime \prime }}
	\quad
	\in  \mathrm{GL}_n(\Complex)
	\quad\quad  \left( \mathscr{U}, \mathscr{U}^\prime, \mathscr{U}^{\prime \prime} \text{ bases for } V^* \right).
\end{equation*}

\begin{proposition}[Fock-Goncharov]
\label{PROP:ELEMENTARY-SNAKE-MOVE-MATRICES}
	Let $ ( E, F, G ) $ be a maximum span flag triple, $ ( \sigma, \sigma^\prime ) $ an adjacent pair of snakes, and $\mathscr{U}$, $\mathscr{U}^\prime$ the corresponding normalized projective bases of $V^*$, satisfying the compatibility condition $u_n = u^\prime_n \in L_{\sigma_n}=L_\eta$.  

If $ ( \sigma, \sigma^\prime ) $ is of diamond-type, then the change of basis matrix $\vec{B}_{\mathscr{U} \to \mathscr{U}^\prime} \in \mathrm{GL}_n(\Complex)$  is
\begin{equation*}  
\label{eq:diamond-move-matrix}
	\vec{B}_{\mathscr{U} \to \mathscr{U}^\prime} = 
	X_{abc}^{+(k-1)/n} \, \vec{S}_k^\mathrm{left}(X_{abc})
	\quad \in \mathrm{GL}_n(\Complex)  \quad\quad  (\text{see } \S \ref{sec:theorem-3}).
\end{equation*}
We say this case expresses a diamond move from the snake $\sigma$ to the adjacent snake $\sigma^\prime$.  

If $ ( \sigma, \sigma^\prime ) $ is of tail-type, then the change of basis matrix $\vec{B}_{\mathscr{U} \to \mathscr{U}^\prime}$ equals
\begin{equation*}
\label{eq:tail-move-matrix}
	\vec{B}_{\mathscr{U} \to \mathscr{U}^\prime} = 
	\vec{S}_1^\mathrm{left}
	\quad  \in  \mathrm{SL}_n(\Complex) \quad\quad  (\text{see } \S \ref{sec:theorem-3}).
\end{equation*}
We say this case expresses a tail move from the snake $\sigma$ to the adjacent snake $\sigma^\prime$.
\end{proposition}    
\begin{proof}
	See \cite[\S 9]{FockIHES06}.  We also provide a proof in \cite[\S 2.18]{DouglasThesis20}. 
\end{proof}

	\subsubsection{Right snakes and right snake moves}
	\label{sec:right-matrices}

Our definition of a (left) snake in \S \ref{subsec:notation-and-definitions} took the snake-head $\eta = \sigma_n$ to be the $n$-th snake-vertex.  There is another possibility, where $\eta = \sigma_1$.

\begin{definition}
	A \textit{right $ n $-snake} $\sigma$ (for the snake-head $\eta=(n-1,0,0)\in \Gamma(\Theta_{n-1})$)  is an ordered list $ \sigma = ( \sigma_1, \sigma_2, \dots, \sigma_n ) \in (\Theta_{ n-1 })^n$ of $n$-many vertices $\sigma_k=(\alpha_k, \beta_k, \gamma_k)$, satisfying
		\begin{equation*}
			\alpha_k = n-k, \quad\quad
			\beta_k \geq \beta_{k-1}, \quad\quad
			\gamma_k \geq \gamma_{k-1}
			\quad\quad  \left(k=1, 2, \dots, n\right).
		\end{equation*}
		Right snakes for other snake-heads $\eta \in \Gamma(\Theta_{n-1})$ are similarly defined by triangular symmetry. 
\end{definition}

To adjust for using right snakes, the definitions of \S \ref{subsec:projective-basis-of-Cndual-associated-to-a-triple-of-flags-and-a-snake}, \ref{sssec:adjacent-snake-pairs}, \ref{sssec:diamond-and-tail-moves} need to be~modified.  

Given $\sigma_{k-1}$, there are two possibilities for $\sigma_k$:
\begin{align*}
\label{eq:def-of-left-projective-basis1}
	\sigma_{k-1}^\mathrm{left} &= (\alpha^\mathrm{left}_{k-1}, \beta^\mathrm{left}_{k-1}, \gamma^\mathrm{left}_{k-1}), 
	&  \alpha^\mathrm{left}_{k-1} &= n-k,
	&  \beta^\mathrm{left}_{k-1} &= \beta_{k-1} + 1,
	&  \gamma^\mathrm{left}_{k-1} &= \gamma_{k-1};
\\
	\sigma_{k-1}^\mathrm{right} &= (\alpha^\mathrm{right}_{k-1}, \beta^\mathrm{right}_{k-1}, \gamma^\mathrm{right}_{k-1}), 
	&  \alpha^\mathrm{right}_{k-1} &= n-k,
	&  \beta^\mathrm{right}_{k-1} &= \beta_{k-1},
	&  \gamma^\mathrm{right}_{k-1} &= \gamma_{k-1} + 1.
\end{align*} 
The algorithm defining the (ordered) projective basis $[ \mathscr{U} ] = [ \left\{ u_1, u_2, \dots, u_n \right\}]$ becomes:
\begin{enumerate}
	\item  if $\sigma_k = \sigma_{k-1}^\mathrm{left}$, put $u_k = - u_{k-1}^\mathrm{left} \in L_{\sigma_{k-1}^\mathrm{left}}$;  
	\item  if $\sigma_k = \sigma_{k-1}^\mathrm{right} $, put $u_k = + u_{k-1}^\mathrm{right} \in L_{\sigma_{k-1}^\mathrm{right}}$.
\end{enumerate}
In particular, the algorithm starts by making a choice of normalization covector $u_1 \in L_{\sigma_1} = L_{(n-1, 0, 0)}$.   Notice that, compared to the setting of left snakes (Definition \ref{def:defining-the-projective-basis-by-induction} and Figure \ref{fig:defining-projective-basis}), the signs defining the projective basis have been swapped.  

	An ordered pair $(\sigma, \sigma^\prime)$ of right snakes forms an adjacent pair if  either:
\begin{enumerate}
	\item  for some $ 2 \leq k \leq n-1 $, 
\begin{enumerate}
	\item  $ \sigma_{j} = \sigma^\prime_{j} $
	\quad\quad
	\quad\quad
	\quad\quad
	\quad\quad
	\quad\quad
	\quad\quad
	 $( 1 \leq j \leq k-1 $, \quad $ k+1 \leq j \leq n)$,
	\item  $\sigma_k = \sigma_{ k-1 }^\mathrm{left} \,\, (=\sigma_{k-1}^{\prime \mathrm{left}}), $ \quad and \quad
  $ \sigma^\prime_k = \sigma_{ k-1 }^\mathrm{right} \,\, (=\sigma_{k-1}^{\prime \mathrm{right}}) $,
\end{enumerate} 
in which case $ ( \sigma, \sigma^\prime ) $ is called an adjacent pair of diamond-type;
\item  
	\begin{enumerate}
	\item  $ \sigma_{j} = \sigma^\prime_{j}$ 
	\quad\quad
	\quad\quad
	\quad\quad
	\quad\quad
	\quad\quad
	\quad\quad
	$(1 \leq j \leq n-1)$,
	\item  $\sigma_n = \sigma_{ n-1 }^\mathrm{left}  \,\, (=\sigma_{n-1}^{\prime \mathrm{left}})$, \quad and \quad
  $ \sigma^\prime_n = \sigma_{ n-1 }^\mathrm{right} \,\, (=\sigma_{n-1}^{\prime \mathrm{right}})$,
\end{enumerate} 
in which case $ ( \sigma, \sigma^\prime ) $ is called an adjacent pair of tail-type.
\end{enumerate}

Given an adjacent pair $(\sigma, \sigma^\prime)$ of right snakes of diamond-type, there is naturally associated a vertex $(a,b,c) \in \Theta_n$ to which is assigned a Fock-Goncharov triangle invariant $X_{abc}$. 

\begin{proposition}[Fock-Goncharov]
\label{prop:right-snake-moves}
	Let $ ( E, F, G ) $ be a maximum span triple, $ ( \sigma, \sigma^\prime ) $ an adjacent pair of right snakes, and $\mathscr{U}$, $\mathscr{U}^\prime$ the corresponding normalized projective bases of $V^*$, satisfying the compatibility condition $u_1 = u^\prime_1 \in L_{\sigma_1}=L_\eta$.  

If $ ( \sigma, \sigma^\prime ) $ is of diamond-type, then the change of basis matrix $\vec{B}_{\mathscr{U} \to \mathscr{U}^\prime} \in \mathrm{GL}_n(\Complex)$ equals
\begin{equation*}  
	\vec{B}_{\mathscr{U} \to \mathscr{U}^\prime} = 
	X_{abc}^{-(k-1)/n} \, \vec{S}_k^\mathrm{right}(X_{abc})
	\quad \in \mathrm{GL}_n(\Complex) \quad\quad  (\text{see } \S \ref{sec:theorem-3}).
\end{equation*}

If $ ( \sigma, \sigma^\prime ) $ is of tail-type, then the change of basis matrix $\vec{B}_{\mathscr{U} \to \mathscr{U}^\prime}$ equals
\begin{equation*}
	\vec{B}_{\mathscr{U} \to \mathscr{U}^\prime} = 
	\vec{S}_1^\mathrm{right}
	\quad  \in  \mathrm{SL}_n(\Complex)  \quad\quad  (\text{see } \S \ref{sec:theorem-3}). 
\end{equation*}
\end{proposition}  

\begin{proof}
	See \cite[\S 9]{FockIHES06}.  Similar to the proof of Proposition \ref{PROP:ELEMENTARY-SNAKE-MOVE-MATRICES}.  
\end{proof}

\begin{remark}
\label{rem:left-and-right-snakes-conventions}
From now on, ``snake'' means ``left snake'', as in Definition \ref{def:snakes}, and we will say explicitly if we are using ``right snakes''.  
\end{remark}

	\subsubsection{Snake moves for edges}
	\label{sec:more-elementary-snake-moves}

\begin{warning}
	In this subsubsection, we will consider snake-heads in the set of corner vertices $\Gamma(\Theta_{n-1})$ other than $(n-1,0,0)$, specifically $\eta$ below; see Remark \ref{rem:snake-heads}.  
\end{warning}

Let $(E, G, F, F^\prime)$ be a maximum span flag quadruple; see \S \ref{subsec:generic-triples-of-flags}.  By \S \ref{subsec:fock-goncharov-triangle-and-edge-invariants}, for each $j=1, \dots, n-1$ we may consider the Fock-Goncharov edge invariant $Z_j = \epsilon_j(E, G, F, F^\prime) \in \Complex - \left\{ 0 \right\}$ associated to the quadruple $(E, G, F, F^\prime)$.  

Consider two copies of the discrete triangle; Figure \ref{fig:shearing-snakes}.  The bottom triangle $\Theta_{ n-1 }(G, F, E)$ has a maximum span flag triple $(G, F, E)$ assigned to the corner vertices $\Gamma(\Theta_{n-1})$, and the top triangle $\Theta_{ n-1 }(E, F^\prime, G)$ has assigned to $\Gamma(\Theta_{n-1})$ a maximum span flag triple~$(E, F^\prime, G)$.  

Define (left) snakes $\sigma$ and $\sigma^\prime$ in $\Theta_{ n-1 }(G, F, E)$ and $\Theta_{ n-1 }(E, F^\prime, G)$, respectively, as~follows:
\begin{align*}
	\sigma_k &= (n-k, 0, k-1)
	  &&\in \Theta_{ n-1 }(G, F, E)
	&&\left(  k=1, \dots, n  \right);
	\\
	\sigma^\prime_k &= (k-1, 0, n-k)
	&&\in \Theta_{ n-1 }(E, F^\prime, G)
	&&\left(  k=1, \dots, n  \right).
\end{align*}

Notice that the line decompositions associated to the snakes $\sigma$ and $\sigma^\prime$ and their respective triples of flags are the same:
\begin{equation*}
	L_{\sigma_k} = L_{\sigma^\prime_k} 
	=  \left(E^{(k-1)} \oplus G^{(n-k)}\right)^\perp
	\quad  \subset V^*
	\quad\quad  \left(  k=1, \dots, n  \right).
\end{equation*}
Let $\mathscr{U}$ and $\mathscr{U}^\prime$ be the associated normalized projective bases, where the normalizations are chosen in a compatible way, that is, such that  $u_n = u^\prime_n$ in $L_{\sigma_n} = L_{\sigma^\prime_n}$.

	\begin{figure}[htb]
	\centering
	\includegraphics[scale=.55]{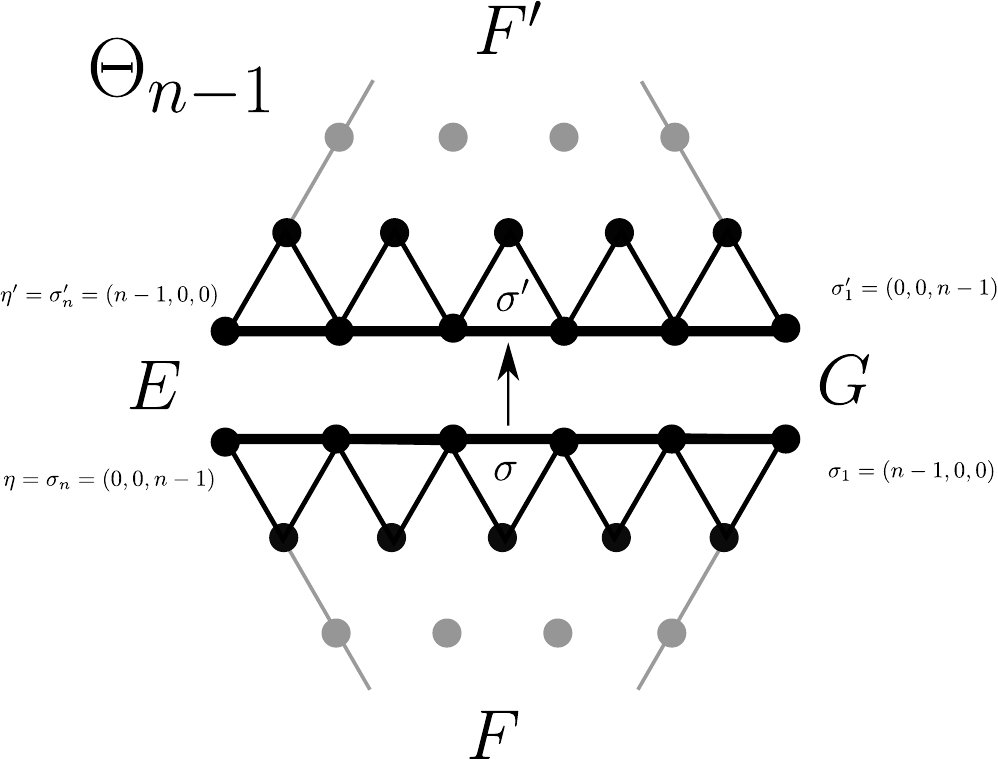}
	\caption{\small Edge move}
	\label{fig:shearing-snakes}
	\end{figure} 

\begin{proposition}[Fock-Goncharov]
\label{prop:edge-move-change-of-basis-matrix}
	The change of basis matrix expressing the snake edge move $\sigma \to \sigma^\prime$ is
	\begin{equation*}
		\vec{B}_{\mathscr{U} \to \mathscr{U}^\prime} =
		\prod_{j=1}^{n-1}  Z_j^{+j/n} \, \vec{S}_j^\mathrm{edge}(Z_j)
		\quad  \in \mathrm{GL}_n(\Complex) \quad\quad  (\text{see } \S \ref{sec:theorem-3}).
	\end{equation*}
\end{proposition}

\begin{proof}
	See \cite[\S 9]{FockIHES06}.  Similar to the proof of Proposition \ref{PROP:ELEMENTARY-SNAKE-MOVE-MATRICES}; see also \cite[\S 2.22]{DouglasThesis20}. 
\end{proof}

Next, define snakes $\sigma$ and $\sigma^\prime$ in a single discrete triangle $\Theta_{n-1}(E, F, G)$ by (see Figure \ref{fig:u-turn})
\begin{align*}
	\sigma_k &= (n-k, 0, k-1)
	  &&\in  \Theta_{ n-1 }
	&&\left(  k=1, \dots, n  \right);
\\
	\sigma^\prime_k &= (k-1, 0, n-k)
	&&\in  \Theta_{ n-1 }
	&&\left(  k=1, \dots, n  \right).
\end{align*}
	
Notice that the lines $L_{\sigma_k} \neq L_{\sigma^\prime_k}$ in $V^*$ are not equal.  In fact, $L_{\sigma_k} = L_{\sigma^\prime_{n-k+1}}$.  Let $\mathscr{U}$ and $\mathscr{U}^\prime$ be the associated normalized projective bases obtained by choosing $u_n = u^\prime_1$ in $L_\eta = L_{\sigma_n} = (G^{(n-1)})^\perp$.

\begin{proposition}
\label{prop:U-turn-change-of-basis-matrix}
The change of basis matrix expressing the snake move $\sigma \to \sigma^\prime$ is
	\begin{equation*}
		 \vec{B}_{\mathscr{U} \to \mathscr{U}^\prime} 
		=  \vec{U}
		\quad  \in  \mathrm{SL}_n(\Complex) \quad\quad  (\text{see } \S \ref{sec:theorem-3}).
	\end{equation*}
\end{proposition}

\begin{proof}
	See \cite[\S 9]{FockIHES06}.  Similar to the proof of Proposition \ref{PROP:ELEMENTARY-SNAKE-MOVE-MATRICES}; see also \cite[\S 2.22]{DouglasThesis20}. 
\end{proof}

\begin{remark}  This last U-turn move will not be needed in this paper, but appears in \cite{DouglasArxiv21}.  
\end{remark}

\begin{figure}[htb]
	\centering
	\includegraphics[width=.55\textwidth]{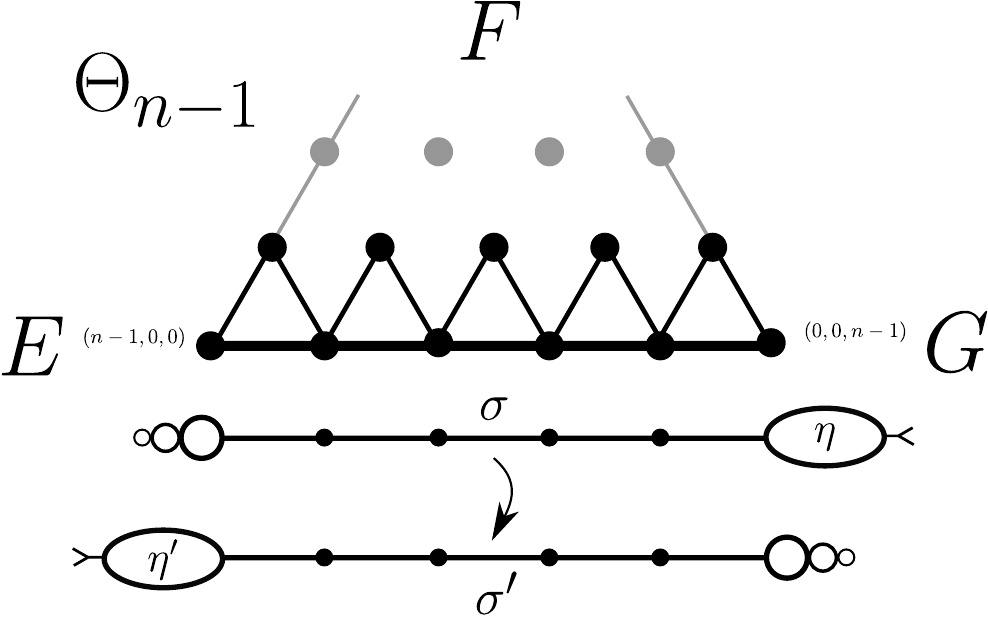}
	\caption{\small Clockwise U-turn Move}
	\label{fig:u-turn}
\end{figure}

	\subsection{Triangle and edge invariants as shears}
	\label{sec:triangle-invariants-as-shears}

This subsection does not involves snakes.  A \textit{line} $L$ (resp. \textit{plane} $P$) in $ V^* $ is a 1-dimensional (resp. 2-dimensional) subspace of $V^*$.  

\begin{definition}
	A \textit{shear} $S$ from a line $L_1$ in $V^*$ to another line $ L_2 $ in $V^*$ is a linear isomorphism $ S : L_1 \to L_2 $. 
\end{definition}

	\subsubsection{Triangle invariants as shears}

Let $ ( E, F, G ) $ be a maximum span flag triple, and consider an internal vertex $(a, b, c) \in \mathrm{int}(\Theta_n)$ in the $n$-discrete triangle.  As in \S \ref{sssec:diamond-and-tail-moves}, the level sets in $\Theta_{n-1}$ defined by the equations $\alpha=a$, $\beta=b$, $\gamma=c$ delineate the boundary of a downward-facing triangle $\nabla$ with vertices $\nu_1^\prime, \nu_2^\prime, \nu_3^\prime$, which is centered in a larger upward-facing triangle $\Delta$ with vertices $\nu_1$, $\nu_2$, $\nu_3$; see Figure \ref{fig:diamond-move-proof}.  There are also three smaller upward-facing triangles $\Delta_1$, $\Delta_2$, $\Delta_3$ defined by their vertices:

\begin{equation*}
	\Delta_1 = \left\{ \nu_1, \nu^\prime_3, \nu^\prime_2 \right\},
	\quad\quad
	\Delta_2 = \left\{ \nu_2, \nu^\prime_1, \nu^\prime_3 \right\},
	\quad\quad
	\Delta_3 = \left\{ \nu_3, \nu^\prime_2, \nu^\prime_1 \right\}.
\end{equation*}

Given one of these small upward-facing triangles, say $\Delta_1$, the property we used to define projective bases in \S \ref{subsec:projective-basis-of-Cndual-associated-to-a-triple-of-flags-and-a-snake} is that the three lines  $L_{\nu_1}$, $L_{\nu^\prime_3}$, $L_{\nu^\prime_2}$ in $V^*$ attached to the vertices of $\Delta_1$ are coplanar.  Consequently, to the triangle $\Delta_1$ there are associated six shears: $S^{\Delta_1}_{\nu_1 \nu^\prime_3} : L_{\nu_1} \to L_{\nu^\prime_3}$ and $S^{\Delta_1}_{\nu^\prime_3 \nu^\prime_2} : L_{\nu^\prime_3} \to L_{\nu^\prime_2}$ and $S^{\Delta_1}_{\nu^\prime_2 \nu_1} : L_{\nu^\prime_2} \to L_{\nu_1}$ and their inverses.  For instance, the shear $S^{\Delta_1}_{\nu_1 \nu^\prime_3}$ sends a point $ p$ in $ L_{\nu_1} $ to the unique point $ p^\prime$ in $L_{\nu^\prime_3}$ such that
\begin{equation*}
	p + p^\prime + p^{\prime\prime} = 0 
	\quad  \in V^*
\end{equation*}
for some (unique) point $p^{\prime\prime} \in L_{\nu^\prime_2}$. And $S^{\Delta_1}_{\nu_1 \nu^\prime_2}(p) = p^{\prime \prime}$.  Similarly for the other triangles $\Delta_2, \Delta_3$.  

Let $X_{abc} = \tau_{abc}(E, F, G)$ be the triangle invariant associated to the vertex $(a,b,c)$.  

\begin{proposition}
\label{prop:triangle-invariants-as-shears}
	Fix a point $p_0$ in the line $L_{ \nu^\prime_1 } $.  
	Let $ p_1 $ be the point in the line $ L_{\nu^\prime_3 } $ resulting from the shear $S^{\Delta_2}_{\nu^\prime_1 \nu^\prime_3}$ associated to the triangle $ \Delta_2 $ applied to the point $ p_0$,  
	let $ p_2 $ be the point in the line $ L_{ \nu^\prime_2 } $ resulting from the shear $S^{\Delta_1}_{\nu^\prime_3 \nu^\prime_2}$ associated to the triangle $ \Delta_1 $ applied to the point $ p_1$, and
	let $ p_3 $ be the point in the line $ L_{ \nu^\prime_1 } $ resulting from the shear $S^{\Delta_3}_{\nu^\prime_2 \nu^\prime_1}$ associated to the triangle $ \Delta_3 $ applied to the point $ p_2$.  It follows that
\begin{equation*}
	p_3 = +X_{abc} \, p_0.
\end{equation*}  

This was the case going counterclockwise around the $(a,b,c)$-downward-facing triangle $\nabla$; see Figure {\upshape\ref{fig:diamond-move-proof}}.  If instead one goes clockwise around $\nabla$, then the total shearing is $+X_{abc}^{-1}$.  
\end{proposition}

\begin{proof}
	See \cite{FockIHES06, GaiottoAnnHenriPoincare14}.  Similar to that of Proposition \ref{PROP:ELEMENTARY-SNAKE-MOVE-MATRICES}; see also \cite[\S 2.21]{DouglasThesis20}.
\end{proof}

	\subsubsection{Edge invariants as shears}
	\label{sec:edge-invariants-as-shears}

Similarly, consider two discrete triangles $\Theta_{n-1}(E, F^\prime, G)$ and $\Theta_{n-1}(G, F, E)$ as in the first half of \S \ref{sec:more-elementary-snake-moves}, the edge invariant $Z_j = \epsilon_j(E, G, F, F^\prime)$ for $j=1,\dots,n-1$, and two small upward-facing (relatively speaking) triangles $\Delta^\prime$ and $\nabla$ in $\Theta_{n-1}(E, F^\prime, G)$ and $\Theta_{n-1}(G, F, E)$, respectively, as shown in Figure \ref{fig:edge-invariants-as-shears}.  

\begin{proposition}
\label{prop:edge-invariants-as-shears}
	Fix a point $p_0$ in the line $L_{\nu^\prime_0}(E, F^\prime, G)$.  
	Let $ p_1 $ be the point in the line $L_{ \nu_1^\prime }(E, F^\prime, G) = L_{\nu_1}(G,F,E) $ resulting from the shear $S^{\Delta^\prime}_{\nu_0^\prime \nu^\prime_1}$ associated to the triangle $ \Delta^\prime $ applied to the point $ p_0$,  
	and let $ p_2 $ be the point in the line $ L_{ \nu_0 }(G, F, E) = L_{\nu^\prime_0}(E, F^\prime, G) $ resulting from the shear $S^{\nabla}_{\nu_1 \nu_0}$ associated to the triangle $ \nabla $  applied to the point $ p_1$.  Then
\begin{equation*}
	p_2 = -Z_j \, p_0.  
\end{equation*}  

This was the case going counterclockwise around the $j$-th diamond; see Figure {\upshape\ref{fig:edge-invariants-as-shears}}.  If instead one goes clockwise around the diamond, then the total shearing is $-Z_j^{-1}$.  
\end{proposition}

\begin{proof}
	See \cite{FockIHES06, GaiottoAnnHenriPoincare14}.  Similar to that of Proposition \ref{PROP:ELEMENTARY-SNAKE-MOVE-MATRICES}; see also \cite[\S 2.23]{DouglasThesis20}.  
\end{proof}

	\begin{figure}[htb]
	\centering
	\includegraphics[width=.65\textwidth]{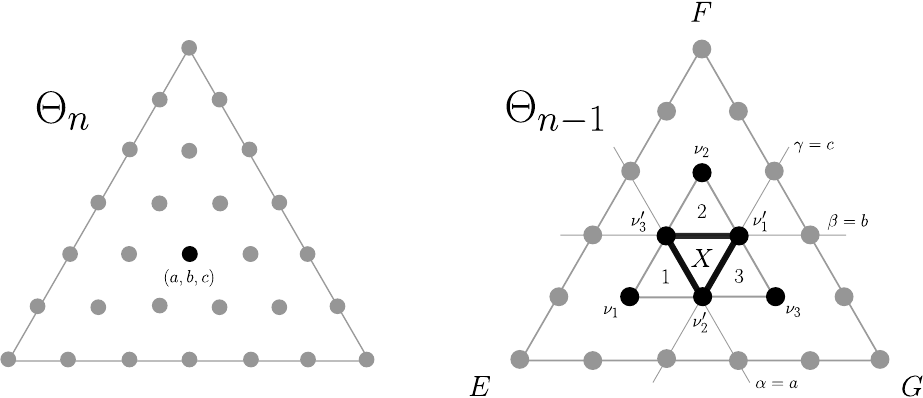}
	\caption{\small Triangle invariants as shears}
	\label{fig:diamond-move-proof}
\end{figure}

\begin{figure}[htb]
	\centering
	\includegraphics[width=.47\textwidth]{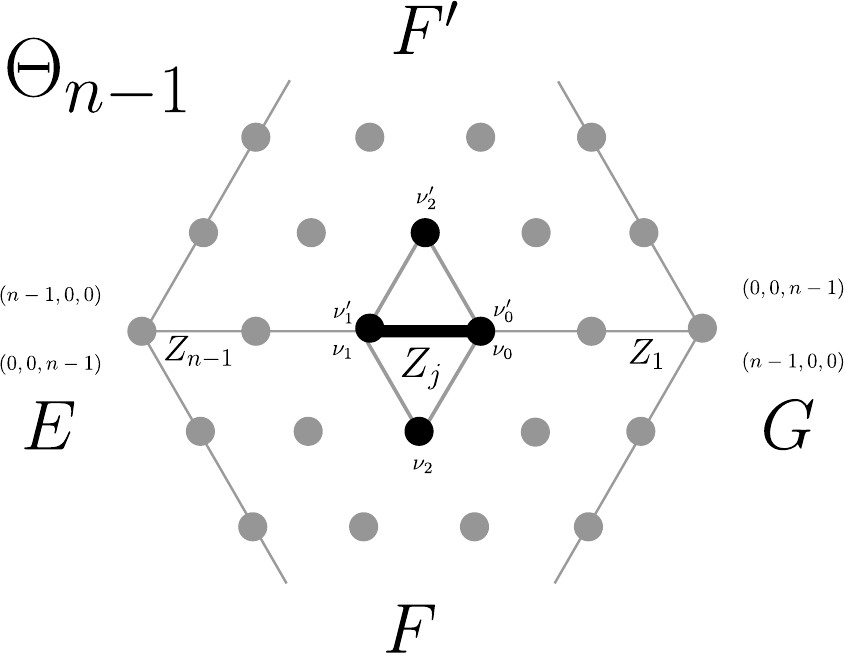}
	\caption{\small Edge invariants as shears}
	\label{fig:edge-invariants-as-shears}
\end{figure}

	\subsection{Classical left, right, and edge matrices}
	\label{sec:local-monodromy-matrices}

\begin{warning}
	In this subsection, we will consider snake-heads in the set of corner vertices $ \left\{ (n-1, 0, 0), (0, n-1, 0), (0, 0, n-1) \right\}$ other than $(n-1,0,0)$; see Remark~\ref{rem:snake-heads}.  
	
	We will also consider both (left) snakes and right snakes; see Remark \ref{rem:left-and-right-snakes-conventions}.
\end{warning}  
	
We begin the process of algebraizing the geometry discussed throughout this first~section.

	\subsubsection{Snake sequences}
	\label{sssec:snake-sequences}
	$ $

\textit{Left setting}:  Define a snake-head $\eta \in \Gamma(\Theta_{n-1})$ and two (left) snakes $\sigma^\mathrm{bot}$, $\sigma^\mathrm{top}$, called the bottom and top snakes, respectively, by
\begin{equation*}
	\eta = (n-1, 0, 0),
\quad	\sigma^\mathrm{bot}_k = (k-1, 0, n-k),
\quad	\sigma^\mathrm{top}_k = (k-1, n-k, 0)
	\quad  \left( k=1, \dots, n \right).  
\end{equation*}

\textit{Right setting}: Define $\eta$ and right snakes $\sigma^\mathrm{bot}$, $\sigma^\mathrm{top}$ by
\begin{equation*}
	\eta = (0, 0, n-1),
\quad	\sigma^\mathrm{bot}_k = (k-1, 0, n-k),
\quad	\sigma^\mathrm{top}_k = (0, k-1, n-k)
	\quad  \left( k=1, \dots, n \right).  
\end{equation*}

In either left or right setting, consider a sequence $\sigma^\mathrm{bot} = \sigma^1, \sigma^2, \cdots,\sigma^{N-1}, \sigma^N = \sigma^\mathrm{top}$ of snakes having the same snake-head $\eta$ as $\sigma^\mathrm{bot}$ and $\sigma^\mathrm{top}$, such that $(\sigma^\ell, \sigma^{\ell+1})$ is an adjacent pair; see Figure \ref{fig:n=5-sequences}. Note that this sequence of snakes is not in general unique.  For the $N$-many projective bases $[\mathscr{U}^\ell] = [\left\{ u^\ell_1, u^\ell_2, \dots, u^\ell_n \right\}]$ associated to the snakes $\sigma^\ell$, choose a common normalization $u_n^\ell :=u_n \in L_\eta$ (resp. $u_1^\ell :=u_1 \in L_\eta$), where the same $u_n$ (resp. $u_1$) is used for all $\ell$, when working in the left (resp. right) setting.  Then, the change of basis matrix $\vec{B}_{\mathscr{U}^\mathrm{bot} \to \mathscr{U}^\mathrm{top}}$ can be decomposed as (see \S \ref{sssec:diamond-and-tail-moves})
\begin{equation*}
\tag{$\ast$}
\label{eq:change-of-basis-left-and-right-matrices}
	\vec{B}_{\mathscr{U}^\mathrm{bot} \to \mathscr{U}^\mathrm{top}}
	=
	\vec{B}_{\mathscr{U}^1 \to \mathscr{U}^2} \vec{B}_{\mathscr{U}^2 \to \mathscr{U}^3} \cdots \vec{B}_{\mathscr{U}^{N-1} \to \mathscr{U}^N}
	\quad  \in \mathrm{GL}_n(\Complex).
\end{equation*}
Here, the matrices $\vec{B}_{\mathscr{U}^\ell \to \mathscr{U}^{\ell+1}}$ are computed as in Proposition \ref{PROP:ELEMENTARY-SNAKE-MOVE-MATRICES} (resp. Proposition \ref{prop:right-snake-moves}) in the left (resp. right) setting, and in particular are completely determined by the Fock-Goncharov triangle invariants $X_{abc} \in \Complex - \left\{ 0 \right\}$ associated to the internal vertices $(a, b, c) \in \mathrm{int}(\Theta_n)$ of the $n$-discrete triangle. 

Note that the matrix $\vec{B}_{\mathscr{U}^\mathrm{bot} \to \mathscr{U}^\mathrm{top}}$ is, by definition, independent of the choice of snake sequence $(\sigma^\ell)_\ell$.  For concreteness, throughout we make a preferred choice of such sequence, depending on whether we are in the left or right setting; see Figure \ref{fig:n=5-sequences}.

\begin{figure}[htb]
     \centering
     \begin{subfigure}[b]{0.48\textwidth}
         \centering
         \includegraphics[width=.75\textwidth]{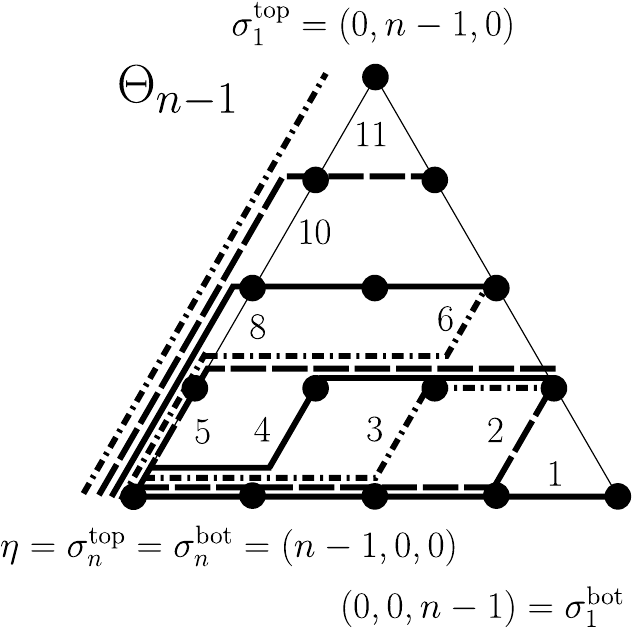}
         \caption{\small Left snake sequence (preferred choice)}
         \label{fig:sequence-of-snakes-n=5-left}
     \end{subfigure}     
\hfill
     \begin{subfigure}[b]{0.48\textwidth}
         \centering
         \includegraphics[width=.75\textwidth]{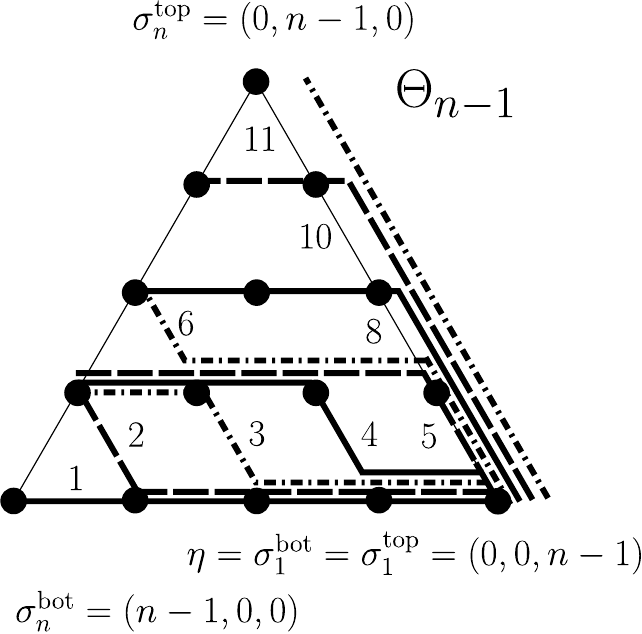}
         \caption{\small Right snake sequence (preferred choice)}
         \label{fig:sequence-of-snakes-n=5-right}
     \end{subfigure}
     	\caption{\small Classical snake sweep ($n=5$)}
        \label{fig:n=5-sequences}
\end{figure}

	\subsubsection{Algebraization}
	\label{sssec:algebraization}

	Let $\mathscr{A}$ be a commutative algebra (\S \ref{sec:theorem-3}).  For $i=1,2,\dots,(n-1)(n-2)/2$, let $X_i^{1/n} \in \mathscr{A}$ and put $X_i = (X_i^{1/n})^n$.   For $j=1,2,\dots,n-1$, let $Z_j^{1/n} \in \mathscr{A}$ and put $Z_j = (Z_j^{1/n})^n$.  Note $(n-1)(n-2)/2$ is the number of elements $(a,b,c) \in \mathrm{int}(\Theta_n)$, which we arbitrarily enumerate  from $i=1,2,\dots,(n-1)(n-2)/2$; see Figures \ref{subfig:left-matrix}, \ref{subfig:right-matrix}.  And note $n-1$ is the number of non-corner vertices of $\Theta_n$ lying on a single edge, which we enumerate from $j=1,2,\dots,n-1$ in the specific way as shown in Figure \ref{subfig:edge-matrix}.  Let $\vec{X}=(X_i)_i$ and $\vec{Z}=(Z_j)_j$ be the corresponding tuples of these elements of $\mathscr{A}$.  

As a notational convention, given a family $\vec{M}_
\ell \in \mathrm{M}_n(\mathscr{A})$ of $n \times n$ matrices, put
\begin{equation*}
\begin{split}
	\prod_{\ell=m}^p \vec{M}_\ell &= \vec{M}_{m} \vec{M}_{m+1} \cdots \vec{M}_p,
	\quad\quad  \prod_{\ell=p+1}^m \vec{M}_\ell = 1
	\quad\quad  \left(  m \leq p  \right),
\\	\coprod_{\ell=p}^m \vec{M}_\ell &= \vec{M}_p \vec{M}_{p-1} \cdots \vec{M}_m,
	\quad\quad  \coprod_{\ell=m-1}^p \vec{M}_\ell = 1
	\quad\quad  \left(  m \leq p  \right).
\end{split}
\end{equation*}

\begin{definition}

 The \textit{left matrix} $\vec{M}^\mathrm{left}(\vec{X})$ in $\mathrm{SL}_n(\mathscr{A})$ is defined by

\begin{equation*}
\label{eq:explicit-quantum-left-matrix}
\vec{M}^\mathrm{left}(\vec{X}) = 
\coprod_{k=n-1}^1 \left( \vec{S}^\mathrm{left}_1 \prod_{\ell=2}^k \vec{S}^\mathrm{left}_\ell\left(X_{(\ell-1)(n-k)(k-\ell+1)}\right) \right)
\quad
\in \mathrm{SL}_n(\mathscr{A})
\end{equation*}
where the matrix $\vec{S}_\ell^\mathrm{left}(X_{abc})$ is the $\ell$-th left-elementary matrix; see \S \ref{sec:theorem-3}. 

Similarly, the \textit{right matrix} $\vec{M}^\mathrm{right}(\vec{X}) $ in $ \mathrm{SL}_n(\mathscr{A})$ is defined by
\begin{equation*}
\label{eq:explicit-quantum-right-matrices}
\vec{M}^\mathrm{right}(\vec{X}) = 
 \coprod_{k=n-1}^1 \left( \vec{S}^\mathrm{right}_1 \prod_{\ell=2}^k \vec{S}^\mathrm{right}_\ell\left(X_{(k-\ell+1)(n-k)(\ell-1)}\right) \right)
\quad   \in \mathrm{SL}_n(\mathscr{A})
\end{equation*}
where the matrix $\vec{S}_\ell^\mathrm{right}(X_{abc})$ is the $\ell$-th right-elementary matrix; see \S \ref{sec:theorem-3}.  

Lastly, the \textit{edge matrix} $\vec{M}^\mathrm{edge}(\vec{Z}) $ in $ \mathrm{SL}_n(\mathscr{A})$ is defined by
\begin{equation*}
\label{eq:edge-matrix}
	\vec{M}^\mathrm{edge}(\vec{Z}) 
	= \prod_{\ell=1}^{n-1} \vec{S}^\mathrm{edge}_\ell(Z_\ell)
	\quad  \in \mathrm{SL}_n(\mathscr{A})
\end{equation*}
where the matrix $\vec{S}_\ell^\mathrm{edge}(Z_\ell)$ is the $\ell$-th  edge-elementary matrix; see \S \ref{sec:theorem-3}.  See Figure~\ref{fig:local-classical-matrices}.
\end{definition}

\begin{figure}[htb]
     \centering
     \begin{subfigure}{0.3\textwidth}
         \centering
         \includegraphics[width=.8\textwidth]{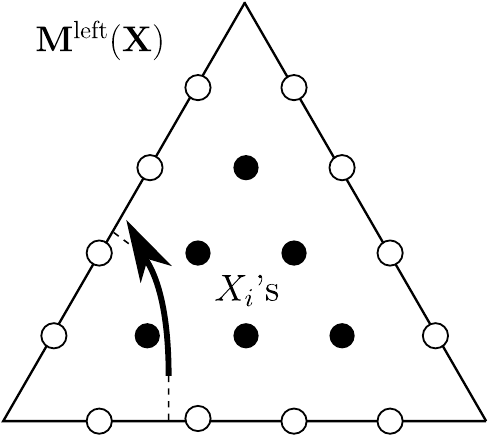}
         \caption{\small Left matrix}
         \label{subfig:left-matrix}
     \end{subfigure}     
\hfill
     \begin{subfigure}{0.3\textwidth}
         \centering
         \includegraphics[width=.8\textwidth]{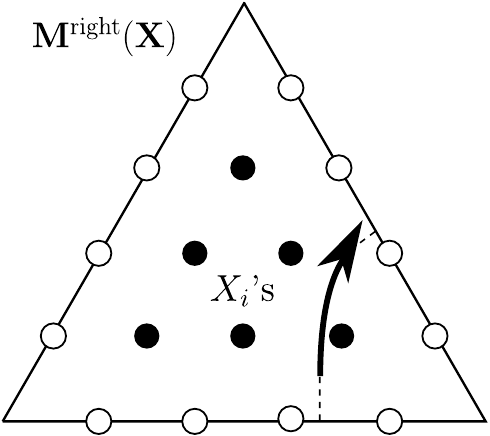}
         \caption{\small Right matrix}
         \label{subfig:right-matrix}
     \end{subfigure}
\hfill
     \begin{subfigure}{0.3\textwidth}
         \centering
         \includegraphics[width=.8\textwidth]{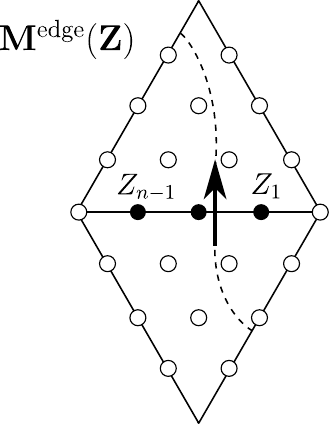}
         \caption{\small Edge matrix}
         \label{subfig:edge-matrix}
     \end{subfigure}
     	\caption{\small Classical matrices (viewed from the $\Theta_n$-perspective)}
     	\label{fig:local-classical-matrices}
\end{figure} 

\begin{remark}
	In the case where $\mathscr{A} = \Complex$ and the $X_i = \tau_{abc}(E, F, G)$ and $Z_j = \epsilon_j(E,G,F,F^\prime)$ in $\Complex - \left\{ 0 \right\}$ are the triangle and edge invariants (as in \S \ref{sssec:diamond-and-tail-moves}, \ref{sec:right-matrices}, \ref{sec:more-elementary-snake-moves}): then, the left and right matrices $\vec{M}^\mathrm{left}(\vec{X})$ and $\vec{M}^\mathrm{right}(\vec{X})$ are the normalized change of basis matrix $\vec{B}_{\mathscr{U}^\mathrm{bot} \to \mathscr{U}^\mathrm{top}}/\mathrm{Det}^{1/n}$ (see Eq. \eqref{eq:change-of-basis-left-and-right-matrices}) in the left and right setting, respectively, normalized to have determinant 1, and decomposed in terms of our preferred snake sequence (Figure \ref{fig:n=5-sequences}); and, the edge matrix $\vec{M}^\mathrm{edge}(\vec{Z})$ is the normalization $\vec{B}_{\mathscr{U} \to \mathscr{U}^\prime}/\mathrm{Det}^{1/n}$ of the change of basis matrix from Proposition \ref{prop:edge-move-change-of-basis-matrix}.  Note these normalizations require choosing $n$-roots of the invariants $X_i$ and $Z_j$.  
\end{remark}

	\section{Quantum matrices}
	\label{sec:quantum-matrices}

	Although we will not use explicitly the geometric results of the previous section, those results motivate the algebraic objects that are the main focus of the present work.  

	Throughout, let $q \in \Complex - \left\{ 0 \right\}$ and $\omega = q^{1/n^2}$ be a $n^2$-root of $q$.  Technically, also choose~$\omega^{1/2}$.

	\subsection{Quantum tori, matrix algebras, and the Weyl quantum ordering}
	\label{ssec:quantum-tori-matrix-algebras-and-weyl-ordering}

	\subsubsection{Quantum tori}
	\label{sssec:quantum-tori}

Let $\vec{P}$ (for ``Poisson'') be an integer $N \times N$ anti-symmetric matrix.  

\begin{definition}
The \emph{quantum torus (with $n$-roots)} $\mathscr{T}^\omega(\vec{P})$ associated to $\vec{P}$ is the quotient of the free algebra $\Complex\{ X_1^{1/n}, X_1^{-1/n}, \dots, X_{N}^{1/n}, X_{N}^{-1/n}\}$ in the indeterminates $X_i^{\pm 1/n}$ by the two-sided ideal generated by the relations
\begin{equation*}	
\label{q-commutation-relations}
	X_i^{1/n} X_j^{1/n} 
	= \omega^{\vec{P}_{ij} } X_j^{1/n} X_i^{1/n},\quad\quad\quad
	X_i^{1/n} X_i^{-{1/n}} = X_i^{-{1/n}} X_i^{1/n} = 1. 
\end{equation*}
Put $X_i^{\pm 1} = (X_i^{\pm 1/n})^n$.  We refer to the $X_i^{\pm 1/n}$ as \textit{generators}, and the $X_i$ as \textit{quantum coordinates}, or just \textit{coordinates}.  Define the subset of fractions
\begin{equation*}
	\Zinteger/n = \left\{ {m/n}; \quad m \in \Zinteger\right\} \quad \subset \Qrational.
\end{equation*}  
\end{definition}
Written in terms of the coordinates $X_i$ and the fractions $r \in \Zinteger/n$, we have the relations
\begin{equation*}
\label{eq:easier-to-read-q-commutation-relations}
	X_i^{r_i} X_j^{r_j} = q^{\vec{P}_{ij} r_i r_j} X_j^{r_j} X_i^{r_i}
	\quad\quad  (r_i, r_j \in \Zinteger/n).
\end{equation*}

	\subsubsection{Matrix algebras}
	\label{sssec:matrix-algebras}

\begin{definition}
\label{def:matrix-algebra}
	The \textit{matrix algebra} $\mathrm{M}_{n}(\mathscr{T})$ with coefficients in a possibly non-commutative algebra $\mathscr{T}$ is the vector space of $n \times n$ matrices, equipped with the usual multiplicative structure.  Namely, the product $\vec{M} \vec{N}$ of two matrices $\vec{M}$ and $\vec{N}$ is defined entrywise by 
	\begin{equation*}
	(\vec{M} \vec{N})_{ij} = \sum_{k=1}^{n} \vec{M}_{ik} \vec{N}_{kj} \quad \in \mathscr{T}
	\quad\quad  \left(  1 \leq i, j \leq n  \right).
\end{equation*}  
\end{definition}
Here, we use the usual convention that the entry $\vec{M}_{ij}$ of a matrix $\vec{M}$ is the entry in the $i$-th row and $j$-th column.  Note that the order of $\vec{M}_{ik}$ and $\vec{N}_{kj}$ in the above equation matters since these elements might not commute in $\mathscr{T}$.

	\subsubsection{Weyl quantum ordering}
	\label{sssec:weyl-ordering}

If $\mathscr{T}$ is a quantum torus, then there is a linear map
\begin{equation*}
\label{eq:weyl-ordering-eq-1}
	[-] \colon \Complex\{ X_1^{1/n}, X_1^{-1/n}, \dots, X_{N}^{1/n}, X_N^{-1/n}\} \to \mathscr{T}
\end{equation*}
from the free algebra to $\mathscr{T}$, called the \emph{Weyl quantum ordering}, defined by the property that a word
	$X_{i_1}^{r_1} X_{i_2}^{r_2} \cdots X_{i_k}^{r_k}$ for $r_a \in \Zinteger/n$ (note $i_a$ may equal $i_b$ if $a \neq b$)
 is mapped to
\begin{equation*}
\label{eq:weyl-quantum-ordering-def}
	[X_{i_1}^{r_1} X_{i_2}^{r_2} \cdots X_{i_k}^{r_k}] = \left( q^{-\frac{1}{2} \sum_{1 \leq a < b \leq k} \vec{P}_{i_a i_b} r_a r_b} \right) X_{i_1}^{r_1} X_{i_2}^{r_2} \cdots X_{i_k}^{r_k},
\end{equation*}
where on the right hand side we implicitly mean the equivalence class in $\mathscr{T}$.  Also, the empty word is mapped to $1$.  Note the Weyl ordering $[-]$ depends on the choice of $\omega^{1/2}$; see the beginning of \S \ref{sec:quantum-matrices}.

  The Weyl ordering is specially designed to satisfy the symmetry
\begin{equation*}
\label{eq:symmetric-property-of-the-Weyl-ordering}
	\left[X_{i_1}^{r_1} X_{i_2}^{r_2} \cdots X_{i_k}^{r_k}\right] 
	= \left[X_{i_{\sigma(1)}}^{r_{\sigma(1)}} X_{i_{\sigma(2)}}^{r_{\sigma(2)}} \cdots X_{i_{\sigma(k)}}^{r_{\sigma(k)}}\right]
\end{equation*}
for every permutation $\sigma$ of $\left\{1, \dots, k\right\}$; see \cite{BonahonGT11}.  Also, $[X_i^{1/n} X_i^{-1/n}]=1$.   Consequently, there is induced a linear map
\begin{equation*}
	[-] \colon \Complex[ X_1^{\pm 1/n}, \dots, X_{N}^{\pm 1/n}] \to \mathscr{T}
\end{equation*}
from the commutative Laurent polynomial algebra to $\mathscr{T}$.  This determines a linear map of matrix algebras
\begin{equation*}
	[-] \colon \mathrm{M}_{n}( \Complex[ X_1^{\pm 1/n}, \dots, X_{N}^{\pm 1/n}]) \to \mathrm{M}_{n}(\mathscr{T}), 
	\quad\quad 	[\vec{M}]_{ij} = [\vec{M}_{ij}] \quad \in \mathscr{T}.
\end{equation*}

	\subsection{Fock-Goncharov quantum torus for a triangle}	
	\label{sec:Fock-Goncharov-algebra-for-a-triangle}

Let $\Gamma(\Theta_n)$ denote the set of corner vertices $\Gamma(\Theta_n)=\left\{ (n, 0, 0), (0, n, 0), (0, 0, n) \right\}$ of the discrete triangle $\Theta_n$; see \S \ref{subsec:the-discrete-triangle}.  

Define a function 
\begin{equation*}
	\vec{P} : \left(\Theta_n - \Gamma(\Theta_n)\right)
	\times 
	\left(\Theta_n - \Gamma(\Theta_n)\right)
	 \to \left\{-2, -1, 0, 1, 2\right\}
\end{equation*}
using the \textit{quiver} with vertex set $\Theta_n - \Gamma(\Theta_n)$ illustrated in Figure \ref{fig:Fg-quiver}.  The function $\vec{P}$ is defined by sending the ordered tuple $(\nu_1, \nu_2)$ of vertices of $\Theta_n - \Gamma(\Theta_n)$ to $2$ (resp. $-2$) if there is a solid arrow pointing from $\nu_1$ to $\nu_2$ (resp. $\nu_2$ to $\nu_1$), to $1$ (resp. $-1$) if there is a dotted arrow pointing from $\nu_1$ to $\nu_2$ (resp. $\nu_2$ to $\nu_1$), and to $0$ if there is no arrow connecting $\nu_1$ and $\nu_2$.  Note that all of the small downward-facing triangles are oriented clockwise, and all of the small upward-facing triangles are oriented counterclockwise.  By labeling the vertices of $\Theta_n - \Gamma(\Theta_n)$ by their coordinates $(a, b, c)$ we may think of the function $\vec{P}$ as a $N \times N$ anti-symmetric matrix $\vec{P} = (\vec{P}_{abc, a^\prime b^\prime c^\prime})$ called the \textit{Poisson matrix} associated to the quiver.  Here, $N = 3(n-1) + (n-1)(n-2)/2$; compare~\S \ref{sssec:algebraization}.  

\begin{definition}
Define the \emph{Fock-Goncharov quantum torus} 
\begin{equation*}
	\mathscr{T}_n^\omega =  \mathbb{C}[X_1^{\pm 1/n}, X_2^{\pm 1/n}, \dots, X_{N}^{\pm 1/n}]^\omega
\end{equation*} 
associated to the discrete $n$-triangle $\Theta_n$ to be the quantum torus $\mathscr{T}^\omega(\vec{P})$ defined by the $N \times N$ Poisson matrix $\vec{P}$, with generators $X_i^{\pm 1/n} = X_{abc}^{\pm 1/n}$ for all $(a, b, c) \in \Theta_n - \Gamma(\Theta_n)$.  Note that when $q=\omega=1$ this recovers the classical Laurent polynomial algebra  $\mathscr{T}_n^1 = \mathbb{C}[X_1^{\pm 1/n}, X_2^{\pm 1/n}, \dots, X_{N}^{\pm 1/n}]$.  

As a notational convention, for $j = 1, 2, \dots, n-1$ we write $Z_{j}^{\pm 1/n}$ (resp. $Z_{j}^{\prime \pm 1/n}$ and $Z_j^{\prime\prime \pm 1/n}$) in place of $X_{j0(n-j)}^{\pm 1/n}$ (resp. $X_{j(n-j)0}^{\pm 1/n}$ and $X_{0j(n-j)}^{\pm 1/n}$); see Figure   \ref{fig:left-and-right-quantum-matrices}.  So, \textit{triangle-coordinates} will be denoted $X_i = X_{abc}$ for $(a, b, c) \in \mathrm{int}(\Theta_n)$ while \textit{edge-coordinates} will be denoted $Z_j, Z^\prime_j, Z^{\prime\prime}_j$.  
\end{definition}

	\subsection{Quantum left and right matrices}
	\label{ssec:quantum-left-and-right-matrices}

	\subsubsection{Weyl quantum ordering for the Fock-Goncharov quantum torus}
	\label{sssec:weyl-ordering-continued}

Let $\mathscr{T} =  \mathscr{T}_n^\omega$ be the Fock-Goncharov quantum torus (\S $\ref{sec:Fock-Goncharov-algebra-for-a-triangle}$).  Then the Weyl ordering $[-]$ of \S \ref{sssec:weyl-ordering} gives a map
\begin{equation*}
	[-] : 
	\mathrm{M}_n(\mathscr{T}_n^1) 
	 \longrightarrow \mathrm{M}_n(\mathscr{T}_n^\omega)
\end{equation*}
where we have used the identification $\mathscr{T}_n^1 = \mathbb{C}[X_1^{\pm 1/n}, X_2^{\pm 1/n}, \dots, X_{N}^{\pm 1/n}]$ discussed in~\S \ref{sec:Fock-Goncharov-algebra-for-a-triangle}.

	\subsubsection{Quantum left and right matrices}
	\label{sssec:quantum-left-and-right-matrices}

\begin{figure}[htb]
	\centering
	\includegraphics[scale=.88]{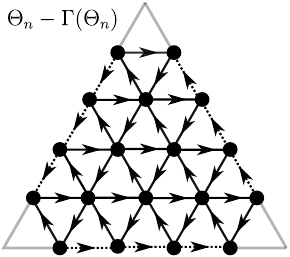}
	\caption{\small Quiver defining the Fock-Goncharov quantum torus}
	\label{fig:Fg-quiver}
\end{figure}	

\begin{figure}[htb]
	\centering
	\includegraphics[scale=.60]{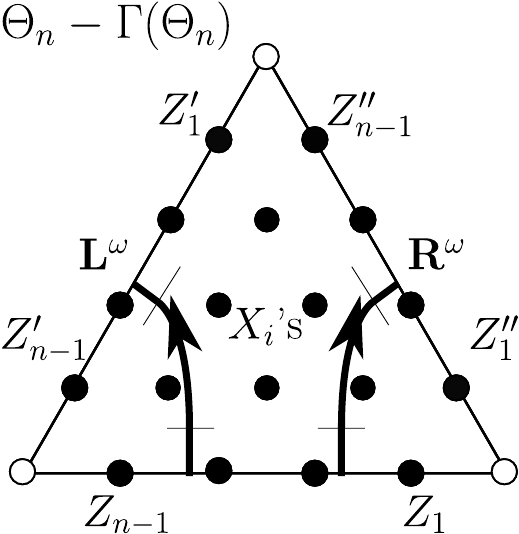}
	\caption{\small Quantum left and right matrices (compare Figure \ref{fig:local-classical-matrices}) }
	\label{fig:left-and-right-quantum-matrices}
\end{figure}	

For a commutative algebra $\mathscr{A}$, in \S \ref{sssec:algebraization} we defined the classical matrices $\vec{M}^\mathrm{left}(\vec{X})$, $\vec{M}^\mathrm{right}(\vec{X})$, and $\vec{M}^\mathrm{edge}(\vec{Z})$ in $\mathrm{SL}_n(\mathscr{A})$.  When $\mathscr{A} = \mathbb{C}[X_1^{\pm 1/n}, \dots, X_{N}^{\pm 1/n}] = \mathscr{T}_n^1$, we now use these matrices to define the primary objects of study.  

\begin{definition}
\label{def:left-and-right-quantum-matrices}
Put vectors $\vec{X}=(X_i)$, $\vec{Z}=(Z_j)$, $\vec{Z}^\prime=(Z_j^\prime)$, $\vec{Z}^{\prime\prime}=(Z_j^{\prime\prime})$ as in Figure \ref{fig:left-and-right-quantum-matrices}.  We define the \textit{quantum left matrix} $\vec{L}^\omega$ in $\mathrm{M}_n(\mathscr{T}_n^\omega)$ by the formula
\begin{equation*}
\label{eq:quantum-left-matrix}
	 \vec{L}^\omega 
	 = \vec{L}^\omega(\textbf{Z}, \textbf{X}, \vec{Z}^\prime)
	=  \left[  \vec{M}^\mathrm{edge}(\vec{Z}) \vec{M}^\mathrm{left}(\vec{X}) \vec{M}^\mathrm{edge}(\vec{Z}^\prime) \right]
	\quad  \in  \mathrm{M}_n\left(\mathscr{T}_n^\omega\right)
\end{equation*}
where we have applied the Weyl quantum ordering $[-]$ discussed in \S \ref{sssec:weyl-ordering-continued} to the product  $\vec{M}^\mathrm{edge}(\vec{Z}) \vec{M}^\mathrm{left}(\vec{X}) \vec{M}^\mathrm{edge}(\vec{Z}^\prime)$ of classical matrices in $\mathrm{M}_n(\mathscr{T}_n^1) $.  In other words, this  means that we apply the Weyl ordering to each entry of the classical matrix.  

Similarly, as in Figure \ref{fig:left-and-right-quantum-matrices}, we define the \textit{quantum right matrix} $\vec{R}^\omega$ in $\mathrm{M}_n(\mathscr{T}_n^\omega)$ by
\begin{equation*}
\label{eq:quantum-right-matrix}
	 \vec{R}^\omega 
	 = \vec{R}^\omega(\textbf{Z}, \textbf{X}, \vec{Z}^{\prime\prime})
	=  \left[  \vec{M}^\mathrm{edge}(\vec{Z}) \vec{M}^\mathrm{right}(\vec{X}) \vec{M}^\mathrm{edge}(\vec{Z}^{\prime\prime}) \right]
	\quad  \in  \mathrm{M}_n\left(\mathscr{T}_n^\omega\right).  
\end{equation*}
\end{definition}

	\subsection{Main result}

	\subsubsection{Quantum \texorpdfstring{$\mathrm{SL}_n$}{SLn} and its points}
	\label{sec:quantum-SLn}

Let $\mathscr{T}$ be a, possibly non-commutative, algebra.  

\begin{definition}
We say that a $2 \times 2$ matrix $\vec{M} = \left(\begin{smallmatrix} a & b \\ c & d \end{smallmatrix}\right)$ in $\mathrm{M}_2(\mathscr{T})$ is a \emph{$\mathscr{T}$-point of the quantum matrix algebra $\mathrm{M}_2^q$}, denoted $\vec{M} \in \mathrm{M}_2^q(\mathscr{T}) \subset \mathrm{M}_2(\mathscr{T})$, if 
\begin{equation*}
\tag{$\ast\ast$}
\label{eq:quantum-group-2by2-relations}
	ba = q ab, 
	\quad  dc = qcd, 
	\quad  ca = qac, 
	\quad  db = qbd, 
	\quad  bc = cb,
	\quad  da - ad = (q - q^{-1}) bc
	\quad  \in  \mathscr{T}.  
\end{equation*}

We say that a matrix $\vec{M} \in \mathrm{M}_2(\mathscr{T})$ is a \emph{$\mathscr{T}$-point of the quantum special linear group $\mathrm{SL}_2^q$}, denoted $\vec{M} \in \mathrm{SL}_2^q(\mathscr{T}) \subset \mathrm{M}_2^q(\mathscr{T}) \subset \mathrm{M}_2(\mathscr{T})$, if 
	$\vec{M} \in \mathrm{M}_2^q(\mathscr{T})$
and the \textit{quantum determinant}
\begin{equation*}
	\mathrm{Det}^q(\vec{M}) \quad =  ad - q^{-1}bc \quad = 1
	\quad  \in \mathscr{T}.
\end{equation*}
\end{definition}

These notions are also defined for $n \times n$ matrices, as follows.  

\begin{definition}
\label{def:pointsofmnq}
A matrix $\vec{M} \in \mathrm{M}_n(\mathscr{T})$ is a \emph{$\mathscr{T}$-point of the quantum matrix algebra $\mathrm{M}_n^q$}, denoted $\vec{M} \in \mathrm{M}_n^q(\mathscr{T}) \subset \mathrm{M}_n(\mathscr{T})$, if every $2 \times 2$ submatrix of $\vec{M}$ is a $\mathscr{T}$-point of $\mathrm{M}_2^q$.  That is,
\begin{gather*}		\label{defining-relations-of-quantum-coord-ring}
	\textbf{M}_{im} \textbf{M}_{ik} = q \textbf{M}_{ik} \textbf{M}_{im}, \quad\quad
	\textbf{M}_{jm} \textbf{M}_{im} = q \textbf{M}_{im} \textbf{M}_{jm},	\\		\notag
	\textbf{M}_{im} \textbf{M}_{jk} = \textbf{M}_{jk} \textbf{M}_{im}, \quad\quad
	\textbf{M}_{jm} \textbf{M}_{ik} - \textbf{M}_{ik} \textbf{M}_{jm} = (q - q^{-1}) \textbf{M}_{im} \textbf{M}_{jk},
\end{gather*}	
for all $i < j$ and $k < m$, where $1 \leq i, j, k, m \leq n$.  

The \textit{quantum determinant} $\mathrm{Det}^q(\vec{M}) \in \mathscr{T}$ of a matrix $\vec{M} \in \mathrm{M}_n(\mathscr{T})$ is
\begin{equation*}
	\mathrm{Det}^q(\vec{M})=\sum_{\sigma\in\mathfrak{S}_n} (-q^{-1})^{\ell(\sigma)} \vec{M}_{1\sigma(1)} \vec{M}_{2\sigma(2)} \cdots \vec{M}_{n\sigma(n)}  
\end{equation*}
where the length $\ell(\sigma)$ of the permutation $\sigma$ is the minimum number of factors appearing in a decomposition of $\sigma$ as a product of adjacent transpositions $(i, i+1)$; see, for example, \cite[Chapter I.2]{Brown02}.

A matrix $\vec{M} \in \mathrm{M}_n(\mathscr{T})$ is a \emph{$\mathscr{T}$-point of the quantum special linear group $\mathrm{SL}_n^q$}, denoted $\vec{M} \in \mathrm{SL}_n^q(\mathscr{T}) \subset \mathrm{M}_n^q(\mathscr{T}) \subset \mathrm{M}_n(\mathscr{T})$, if both $\vec{M} \in \mathrm{M}_n^q(\mathscr{T})$ and  $\mathrm{Det}^q(\vec{M}) = 1$.  
\end{definition}

\begin{remark}\label{rem:remarksaboutquantummatrices}  $ $
\begin{enumerate}
\item\label{item:simplifiedqdet}  
It follows from the definitions that if a $\mathscr{T}$-point $\vec{M} \in \mathrm{M}_n^q(\mathscr{T}) \subset \mathrm{M}_n(\mathscr{T})$ is a triangular matrix, then the diagonal entries $\vec{M}_{ii} \in \mathscr{T}$ commute, and $\mathrm{Det}^q(\vec{M}) = \prod_i \vec{M}_{ii} \in \mathscr{T}$.  
\item	 The subsets $\mathrm{M}_n^q(\mathscr{T}) \subset \mathrm{M}_n(\mathscr{T})$ and $\mathrm{SL}_n^q(\mathscr{T}) \subset \mathrm{M}_n^q(\mathscr{T})$ are generally not closed under matrix multiplication (see, however, the sketch of proof below for a relaxed~property).  
\item  More abstractly, the \textit{quantum special linear group} $\mathrm{SL}_n^q$ is the non-commutative algebra defined as the quotient of the free algebra on generators $\vec{m}_{ij}$ ($1\leq i,j\leq n$) subject to the four relations appearing in Definition \ref{def:pointsofmnq} (with $\vec{M}_{ij}$ replaced by $\vec{m}_{ij}$) plus the relation $\mathrm{Det}^q(\vec{m})=1$; see, for example, \cite[Chapter I.2]{Brown02}.  Note then that a $\mathscr{T}$-point $\vec{M}$ of $\mathrm{SL}_n^q$ is equivalent to an algebra homomorphism $\varphi(\vec{M}):\mathrm{SL}_n^q\to\mathscr{T}$ defined by the property that $\varphi(\vec{M})(\vec{m}_{ij})=\vec{M}_{ij}$ for all $1 \leq i,j \leq n$.
\end{enumerate}
\end{remark}

	\subsubsection{Main result}
	\label{ssec:main-result}

Take $\mathscr{T} = \mathscr{T}_n^\omega$ to be the Fock-Goncharov quantum torus for the discrete $n$-triangle $\Theta_n$; see \S \ref{sec:Fock-Goncharov-algebra-for-a-triangle}.  Let $\vec{L}^\omega$ and $\vec{R}^\omega$ in $\mathrm{M}_n(\mathscr{T}_n^\omega)$ be the quantum left and right matrices, respectively, as defined in Definition \ref{def:left-and-right-quantum-matrices}.   

\begin{theorem} 
\label{thm:first-theorem}
	The quantum left and right matrices
	\begin{equation*}
		\vec{L}^\omega=\vec{L}^\omega(\vec{Z}, \vec{X}, \vec{Z}^\prime), \quad \vec{R}^\omega=\vec{R}^\omega(\vec{Z}, \vec{X}, \vec{Z}^{\prime\prime})	 
		\quad  \in  \mathrm{M}_n(\mathscr{T}_n^\omega)
	\end{equation*}
	 are $\mathscr{T}_n^\omega$-points of the quantum special linear group $\mathrm{SL}_n^q$.  That is, $\vec{L}^\omega, \vec{R}^\omega \in \mathrm{SL}_n^q(\mathscr{T}_n^\omega) \subset~\mathrm{M}_n(\mathscr{T}_n^\omega)$.
\end{theorem}	

The proof, provided in \S \ref{sec:snake-move-algebras}, uses a quantum version of Fock-Goncharov snakes (\S \ref{sec:fock-goncharov-snakes}).  

\begin{proof}[Sketch of proof (see {\upshape\S \ref{sec:snake-move-algebras}} for more details)]
	In the case $n=2$, this is an enjoyable calculation.  When $n \geq 3$, the argument hinges on the following well-known fact (see, for example, \cite[Proposition IV.3.4 and Section IV.10]{Kassel95}):  If $\mathscr{T}$ is an algebra with subalgebras $\mathscr{T}^\prime, \mathscr{T}^{\prime\prime} \subset \mathscr{T}$ that commute in the sense that $a' a'' = a'' a'$ for all $a' \in \mathscr{T}^\prime$ and $a'' \in \mathscr{T}^{\prime\prime}$, and if $\vec{M}' \in \mathrm{M}_n(\mathscr{T}^\prime) \subset \mathrm{M}_n(\mathscr{T})$ and $\vec{M}'' \in \mathrm{M}_n(\mathscr{T}^{\prime\prime}) \subset \mathrm{M}_n(\mathscr{T})$ are $\mathscr{T}$-points of $\mathrm{SL}_n^q$, then the matrix product (\S\ref{def:matrix-algebra}) $\vec{M}' \vec{M}'' \in \mathrm{M}_n(\mathscr{T}^\prime \mathscr{T}^{\prime\prime}) \subset \mathrm{M}_n(\mathscr{T})$ is also a $\mathscr{T}$-point of~$\mathrm{SL}_n^q$.  

Put $\vec{M}_\mathrm{FG} := \vec{L}^\omega$, the quantum left matrix, say.  The proof is the same for the quantum right matrix.  See Definition \ref{def:left-and-right-quantum-matrices}.  The strategy is to see $\vec{M}_\mathrm{FG} \in \mathrm{M}_n(\mathscr{T}_n^\omega)$ as the product of simpler matrices, over mutually-commuting subalgebras, that are themselves points of $\mathrm{SL}_n^q$.  

More precisely, for a fixed sequence of adjacent snakes $\sigma^\mathrm{bot} = \sigma^1, \sigma^2, \dots, \sigma^N = \sigma^\mathrm{top}$ moving left across the triangle from the bottom edge to the top-left edge, we will define for each $i=1, \dots, N-1$ an auxiliary algebra $\mathscr{S}^\omega_{j_i}$ called a \textit{snake-move algebra}, $j_i \in \{ 1, \dots, n-1 \}$, corresponding to the adjacent snake pair $(\sigma^i, \sigma^{i+1})$.  As a technical step, there is a distinguished  subalgebra $\mathscr{T}_L \subset \mathscr{T}_n^\omega$ satisfying $\vec{M}_\mathrm{FG} \in \mathrm{M}_n(\mathscr{T}_L) \subset \mathrm{M}_n(\mathscr{T}_n^\omega)$.  We construct an algebra embedding $\mathscr{T}_L \hookrightarrow \bigotimes_i \mathscr{S}^\omega_{j_i}$.  Through this embedding, we may view $\vec{M}_\mathrm{FG} \in \mathrm{M}_n(\mathscr{T}_L) \subset \mathrm{M}_n(\bigotimes_i \mathscr{S}^\omega_{j_i})$.  

Following, we construct (Proposition \ref{lem:def-of-quantum-elementary-matrix-left}), for each $i$, a matrix $\vec{M}_{j_i} \in \mathrm{M}_n(\mathscr{S}^\omega_{j_i}) \subset \mathrm{M}_n(\bigotimes_i \mathscr{S}^\omega_{j_i})$ such that $\vec{M}_{j_i}$ is a $\mathscr{S}^\omega_{j_i}$-point of $\mathrm{SL}_n^q$, in other words $\vec{M}_{j_i} \in \mathrm{SL}_n^q(\mathscr{S}^\omega_{j_i}) \subset \mathrm{SL}_n^q(\bigotimes_i \mathscr{S}^\omega_{{j_i}})$.  Since by definition the subalgebras $\mathscr{S}^\omega_{j_i}, \mathscr{S}^\omega_{j_{i^\prime}} \subset \bigotimes_i \mathscr{S}^\omega_{j_i}$ commute if $i \neq i'$, as they constitute different tensor factors of $\bigotimes_i \mathscr{S}^\omega_{j_i}$, it follows from the essential fact mentioned above that $\vec{M} := \vec{M}_{j_1} \vec{M}_{j_2} \cdots \vec{M}_{j_{N-1}} \in \mathrm{M}_n(\bigotimes_i \mathscr{S}^\omega_{j_i})$ is a $(\bigotimes_i \mathscr{S}^\omega_{j_i})$-point of $\mathrm{SL}_n^q$, in other words $\vec{M} \in \mathrm{SL}_n^q(\bigotimes_i \mathscr{S}^\omega_{j_i})$.  

Since this matrix product $\vec{M}$, as well as the quantum left matrix $\vec{M}_\mathrm{FG}$, are being viewed as elements of $\mathrm{M}_n(\bigotimes_i \mathscr{S}^\omega_{j_i})$, it makes sense to ask whether $\vec{M}_\mathrm{FG} \overset{?}{=} \vec{M} \in \mathrm{M}_n(\bigotimes_i \mathscr{S}^\omega_{j_i})$.  We show this is true, implying that $\vec{M}_\mathrm{FG} \in \mathrm{SL}_n^q(\bigotimes_i \mathscr{S}^\omega_{j_i})$.  Since $\vec{M}_\mathrm{FG} \in \mathrm{M}_n(\mathscr{T}_L) \subset \mathrm{M}_n(\bigotimes_i \mathscr{S}^\omega_{j_i})$, we conclude that $\vec{M}_\mathrm{FG}$ is in $\mathrm{SL}_n^q(\mathscr{T}_L) \subset \mathrm{SL}_n^q(\mathscr{T}_n^\omega)$.
\end{proof}

	\subsection{Example}
	\label{sec:concrete-formulas}

Consider the case $n=4$; see Figure \ref{fig:n=4-left-and-right-matrices-example}.  On the right hand side we show the quiver defining the commutation relations in the quantum torus $\mathscr{T}_4^\omega$, recalling Figure \ref{fig:Fg-quiver}, but viewed in $\Theta_{n-1}$.  Note that there is a one-to-one correspondence between points $(a,b,c) \in \mathrm{int}(\Theta_n)$ and small downward-facing triangles inside $\Theta_{n-1}$, as shown Figure \ref{fig:n=4-left-and-right-matrices-example}.  In particular, to each downward-facing triangle there is associated a triangle-coordinate~$X_i$. 

 Some sample commutation relations in $\mathscr{T}_4^\omega$ are:
\begin{equation*}
	X_3 Z^{\prime\prime}_2 = q^2 X_3 Z^{\prime\prime}_2,
	\quad\quad
	X_3 X_1 = q^{-2} X_1 X_3,
	\quad\quad
	Z_3 Z_2 = q Z_2 Z_3,
	\quad\quad
	Z_3 Z^\prime_3 = q^2 Z^\prime_3 Z_3.
\end{equation*}

\begin{figure}[htb]
     \centering
     \begin{subfigure}[b]{0.465\textwidth}
         \centering
         \includegraphics[width=.55\textwidth]{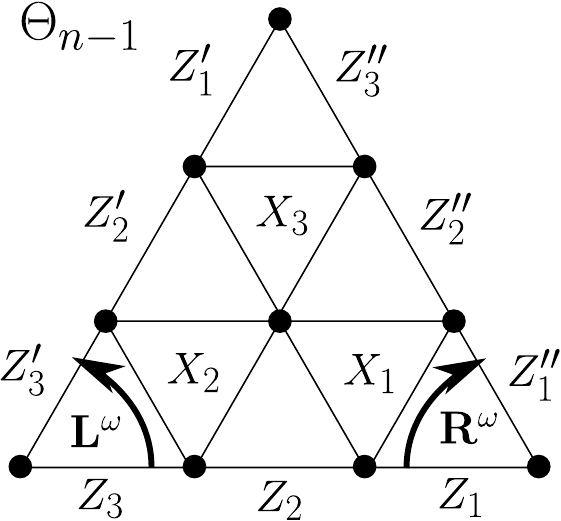}
         \caption{\small Left and right matrices}
         \label{fig:n=4-left-and-right-matrices-example-subfigure}
     \end{subfigure}     
\hfill
     \begin{subfigure}[b]{0.465\textwidth}
         \centering
         \includegraphics[width=.6\textwidth]{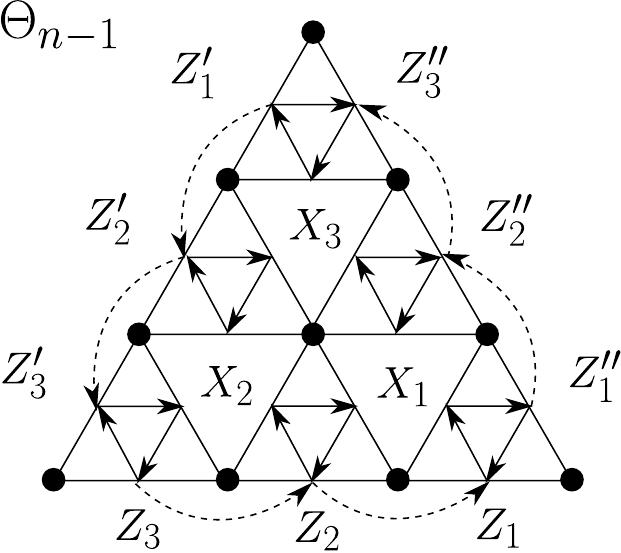}
         \caption{\small Quiver}
         \label{fig:n=4-quiver}
     \end{subfigure}
     	\caption{\small Quantum matrices and quantum torus ($n=4$)}
        \label{fig:n=4-left-and-right-matrices-example}
\end{figure}

Then, the quantum left and right matrices $\vec{L}^\omega$ and $\vec{R}^\omega$ are computed as
\begin{equation*}
\begin{split}
	\vec{L}^\omega = 
	&\Bigg[
Z_1^{-\frac{1}{4}} Z_2^{-\frac{2}{4}} Z_3^{-\frac{3}{4}}	
\left(\begin{smallmatrix}
Z_1 Z_2 Z_3&&&\\
&Z_2 Z_3&&\\
&&Z_3&\\
&&&1
\end{smallmatrix}\right)
\left(\begin{smallmatrix}
1&1&&\\
&1&&\\
&&1&\\
&&&1
\end{smallmatrix}\right)
X_1^{-\frac{1}{4}}
\left(\begin{smallmatrix}
X_1&&&\\
&1&1&\\
&&1&\\
&&&1
\end{smallmatrix}\right)
X_2^{-\frac{2}{4}}
\left(\begin{smallmatrix}
X_2&&&\\
&X_2&&\\
&&1&1\\
&&&1
\end{smallmatrix}\right)
\\
&\left(\begin{smallmatrix}
1&1&&\\
&1&&\\
&&1&\\
&&&1
\end{smallmatrix}\right)
X_3^{-\frac{1}{4}}
\left(\begin{smallmatrix}
X_3&&&\\
&1&1&\\
&&1&\\
&&&1
\end{smallmatrix}\right)
\left(\begin{smallmatrix}
1&1&&\\
&1&&\\
&&1&\\
&&&1
\end{smallmatrix}\right)
Z_1^{\prime-\frac{1}{4}} Z_2^{\prime-\frac{2}{4}} Z_3^{\prime-\frac{3}{4}}	
\left(\begin{smallmatrix}
Z^\prime_1 Z^\prime_2 Z^\prime_3&&&\\
&Z^\prime_2 Z^\prime_3&&\\
&&Z^\prime_3&\\
&&&1
\end{smallmatrix}\right)
	\Bigg];
\end{split}
\end{equation*}
and
\begin{equation*}
\label{eq:example-4by4-right-matrix}
\begin{split}
	\vec{R}^\omega = 
	&\Bigg[
Z_1^{-\frac{1}{4}} Z_2^{-\frac{2}{4}} Z_3^{-\frac{3}{4}}	
\left(\begin{smallmatrix}
Z_1 Z_2 Z_3&&&\\
&Z_2 Z_3&&\\
&&Z_3&\\
&&&1
\end{smallmatrix}\right)
\left(\begin{smallmatrix}
1&&&\\
&1&&\\
&&1&\\
&&1&1
\end{smallmatrix}\right)
X_2^{+\frac{1}{4}}
\left(\begin{smallmatrix}
1&&&\\
&1&&\\
&1&1&\\
&&&X_2^{-1}
\end{smallmatrix}\right)
X_1^{+\frac{2}{4}}
\left(\begin{smallmatrix}
1&&&\\
1&1&&\\
&&X_1^{-1}&\\
&&&X_1^{-1}
\end{smallmatrix}\right)
\\
&\left(\begin{smallmatrix}
1&&&\\
&1&&\\
&&1&\\
&&1&1
\end{smallmatrix}\right)
X_3^{+\frac{1}{4}}
\left(\begin{smallmatrix}
1&&&\\
&1&&\\
&1&1&\\
&&&X_3^{-1}
\end{smallmatrix}\right)
\left(\begin{smallmatrix}
1&&&\\
&1&&\\
&&1&\\
&&1&1
\end{smallmatrix}\right)
Z_1^{\prime\prime-\frac{1}{4}} Z_2^{\prime\prime-\frac{2}{4}} Z_3^{\prime\prime-\frac{3}{4}}	
\left(\begin{smallmatrix}
Z^{\prime\prime}_1 Z^{\prime\prime}_2 Z_3^{\prime\prime}&&&\\
&Z^{\prime\prime}_2 Z^{\prime\prime}_3&&\\
&&Z^{\prime\prime}_3&\\
&&&1
\end{smallmatrix}\right)
	\Bigg].
\end{split}
\end{equation*}

Theorem \ref{thm:first-theorem} says that these two matrices are elements of $\mathrm{SL}_4^q(\mathscr{T}_4^\omega)$.  For instance, the entries $a,b,c,d$ of the $2 \times 2$ submatrix (arranged as a $4 \times 1$ matrix) of $\vec{L}^\omega$
\begin{equation*}
\label{eq:quantum-left-matrix-2x2-sub-matrix}
\begin{split}
	\left(\begin{smallmatrix}
		a\\b\\
		c\\d
	\end{smallmatrix}\right)
	=
	\left(\begin{smallmatrix}
		\vec{L}^\omega_{13}\\\vec{L}^\omega_{14}\\
		\vec{L}^\omega_{23}\\\vec{L}^\omega_{24}
	\end{smallmatrix}\right)
	=
	\left(\begin{smallmatrix}
		[Z_3^\frac{1}{4} Z_2^\frac{2}{4} Z_1^\frac{3}{4} 
		Z_3^{\prime\frac{1}{4}} Z_2^{\prime-\frac{2}{4}} Z_1^{\prime-\frac{1}{4}} 
		X_1^{-\frac{1}{4}} X_2^{-\frac{2}{4}} X_3^{-\frac{1}{4}}]
	+	
		[Z_3^\frac{1}{4} Z_2^\frac{2}{4} Z_1^\frac{3}{4} 
		Z_3^{\prime\frac{1}{4}} Z_2^{\prime-\frac{2}{4}} Z_1^{\prime-\frac{1}{4}} 
		X_1^{-\frac{1}{4}} X_2^\frac{2}{4} X_3^{-\frac{1}{4}}]+
	\\+	
		[Z_3^\frac{1}{4} Z_2^\frac{2}{4} Z_1^\frac{3}{4} 
		Z_3^{\prime\frac{1}{4}} Z_2^{\prime-\frac{2}{4}} Z_1^{\prime-\frac{1}{4}} 
		X_1^\frac{3}{4} X_2^\frac{2}{4} X_3^{-\frac{1}{4}}]
\\
		[Z_3^\frac{1}{4} Z_2^\frac{2}{4} Z_1^\frac{3}{4} 
		Z_3^{\prime-\frac{3}{4}} Z_2^{\prime-\frac{2}{4}} Z_1^{\prime-\frac{1}{4}} 
		X_1^{-\frac{1}{4}} X_2^{-\frac{2}{4}} X_3^{-\frac{1}{4}}]
\\	
		[Z_3^\frac{1}{4} Z_2^\frac{2}{4} Z_1^{-\frac{1}{4}} 
		Z_3^{\prime\frac{1}{4}} Z_2^{\prime-\frac{2}{4}} Z_1^{\prime-\frac{1}{4}} 
		X_1^{-\frac{1}{4}} X_2^{-\frac{2}{4}} X_3^{-\frac{1}{4}}]	
	+	
		[Z_3^\frac{1}{4} Z_2^\frac{2}{4} Z_1^{-\frac{1}{4}} 
		Z_3^{\prime\frac{1}{4}} Z_2^{\prime-\frac{2}{4}} Z_1^{\prime-\frac{1}{4}} 
		X_1^{-\frac{1}{4}} X_2^\frac{2}{4} X_3^{-\frac{1}{4}}]
\\
		[Z_3^\frac{1}{4} Z_2^\frac{2}{4} Z_1^{-\frac{1}{4}} 
		Z_3^{\prime-\frac{3}{4}} Z_2^{\prime-\frac{2}{4}} Z_1^{\prime-\frac{1}{4}} 
		X_1^{-\frac{1}{4}} X_2^{-\frac{2}{4}} X_3^{-\frac{1}{4}}]	
	\end{smallmatrix}\right)
\end{split}
\end{equation*}
satisfy Equation \eqref{eq:quantum-group-2by2-relations}.  For a computer demonstration of this, see \cite[Appendix B]{DouglasArxiv21}.  We also verify in that appendix that Equation \eqref{eq:quantum-group-2by2-relations} is satisfied by the entries $a,b,c,d$ of the $2 \times 2$ submatrix (arranged as a $4 \times 1$ matrix) of $\vec{R}^\omega$
\begin{equation*}
\label{eq:quantum-right-matrix-2x2-sub-matrix}
\begin{split}
	\left(\begin{smallmatrix}
		a\\b\\
		c\\d
	\end{smallmatrix}\right)
	\!=\!
	\left(\begin{smallmatrix}
		\vec{R}^\omega_{31}\\\vec{R}^\omega_{32}\\
		\vec{R}^\omega_{41}\\\vec{R}^\omega_{42}
	\end{smallmatrix}\right)
	\!=\!
	\left(\begin{smallmatrix}
		[Z_3^\frac{1}{4} Z_2^{-\frac{1}{2}}  Z_1^{-\frac{1}{4}}  X_2^\frac{1}{4}  X_1^{\frac{1}{2}}  X_3^{\frac{1}{4}}  Z_3^{\prime\prime \frac{1}{4}} 
		Z_2^{\prime\prime \frac{1}{2}}   Z_1^{\prime\prime \frac{3}{4}}
			]
	\\
		[Z_3^\frac{1}{4} Z_2^{-\frac{1}{2}}  Z_1^{-\frac{1}{4}}  X_2^\frac{1}{4}  X_1^{-\frac{1}{2}}  X_3^{\frac{1}{4}}  Z_3^{\prime\prime \frac{1}{4}}  
		Z_2^{\prime\prime \frac{1}{2}}  Z_1^{\prime\prime -\frac{1}{4}}
			]	
	+	
		[Z_3^\frac{1}{4} Z_2^{-\frac{1}{2}}  Z_1^{-\frac{1}{4}}  X_2^\frac{1}{4}  X_1^{\frac{1}{2}}  X_3^{\frac{1}{4}}  Z_3^{\prime\prime \frac{1}{4}}  
		Z_2^{\prime\prime \frac{1}{2}}  Z_1^{\prime\prime -\frac{1}{4}}
			]
	\\
		[Z_3^{-\frac{3}{4}} Z_2^{-\frac{1}{2}}  Z_1^{-\frac{1}{4}}  X_2^\frac{1}{4}  X_1^{\frac{1}{2}}  X_3^{\frac{1}{4}}  Z_3^{\prime\prime \frac{1}{4}}  
		Z_2^{\prime\prime \frac{1}{2}}  Z_1^{\prime\prime \frac{3}{4}}
			]
	\\		
			[Z_3^{-\frac{3}{4}} Z_2^{-\frac{1}{2}}  Z_1^{-\frac{1}{4}}  X_2^{-\frac{3}{4}}  X_1^{-\frac{1}{2}}  X_3^{\frac{1}{4}}  Z_3^{\prime\prime \frac{1}{4}}  
		Z_2^{\prime\prime \frac{1}{2}}  Z_1^{\prime\prime -\frac{1}{4}}
				]
		+	
			[Z_3^{-\frac{3}{4}} Z_2^{-\frac{1}{2}}  Z_1^{-\frac{1}{4}}  X_2^\frac{1}{4}  X_1^{-\frac{1}{2}}  X_3^{\frac{1}{4}}  Z_3^{\prime\prime \frac{1}{4}}  
		Z_2^{\prime\prime \frac{1}{2}}  Z_1^{\prime\prime -\frac{1}{4}}
				]+
		\\+	
			[Z_3^{-\frac{3}{4}} Z_2^{-\frac{1}{2}}  Z_1^{-\frac{1}{4}}  X_2^\frac{1}{4}  X_1^{\frac{1}{2}}  X_3^{\frac{1}{4}}  Z_3^{\prime\prime \frac{1}{4}}  
		Z_2^{\prime\prime \frac{1}{2}}  Z_1^{\prime\prime -\frac{1}{4}}
				]
	\end{smallmatrix}\right).
\end{split}
\end{equation*}

\begin{remark}
In order for these matrices to satisfy the relations required just to be in $\mathrm{M}_n^q(\mathscr{T}_n^\omega)$ (let alone $\mathrm{SL}_n^q(\mathscr{T}_n^\omega)$), they have to be normalized by dividing out their determinants.  For example, the above matrix $\vec{L}^\omega$ for $n=4$ would not satisfy the $q$-commutation relations required to be a point of $\mathrm{M}_4^q(\mathscr{T}_4^\omega)$ if we had not included the normalizing term $Z_1^{-\frac{1}{4}} Z_2^{-\frac{2}{4}} Z_3^{-\frac{3}{4}} 
X_1^{-\frac{1}{4}} X_2^{-\frac{2}{4}}X_3^{-\frac{1}{4}}  Z_1^{\prime-\frac{1}{4}} Z_2^{\prime-\frac{2}{4}} Z_3^{\prime-\frac{3}{4}}	$, as there would be a $1$ in the bottom corner.  
\end{remark}

	\section{Quantum snakes: proof of Theorem \ref{thm:first-theorem}}
	\label{sec:snake-move-algebras}

Above, we gave a sketch of the proof.  We now fill in the details.  Our emphasis will be on the left matrix $\vec{L}^\omega$.  The proof for the right matrix $\vec{R}^\omega$ is similar, as we will discuss in \S \ref{sec:setup-for-the-quantum-right-matrix}.  

Fix a sequence $\sigma^\mathrm{bot} = \sigma^1, \sigma^2, \dots, \sigma^N = \sigma^\mathrm{top}$ of adjacent snakes, as in the left setting; see \S \ref{sssec:snake-sequences}. The proof is valid for any choice of snake sequence, but our demonstrations in figures and examples will be for our preferred snake sequence; see Figure \ref{fig:n=5-sequences}.  Note that the example quantum matrices in \S \ref{sec:concrete-formulas} were presented using this preferred snake~sequence.

	\subsection{Snake-move quantum tori}

\begin{definition}
For $j=1, \dots, n-1$, the \textit{$j$-th snake-move quantum torus} $\mathscr{S}^\omega_j = \mathscr{T}(\vec{P}_j)$ is the quantum torus with Poisson matrix $\vec{P}_j$ defined by the quiver shown in Figure \ref{fig:snake-move-algebra-left-diamond} when $j=2,\dots, n-1$, and in Figure \ref{fig:snake-move-algebra-left-tail} when $j=1$.  As usual, there is one generator per edge of the quiver, solid arrows carry a weight $2$, and dotted arrows carry a weight $1$; compare \S \ref{sec:Fock-Goncharov-algebra-for-a-triangle}.  
\end{definition}

\begin{figure}[htb]
	\centering
	\includegraphics[width=.66\textwidth]{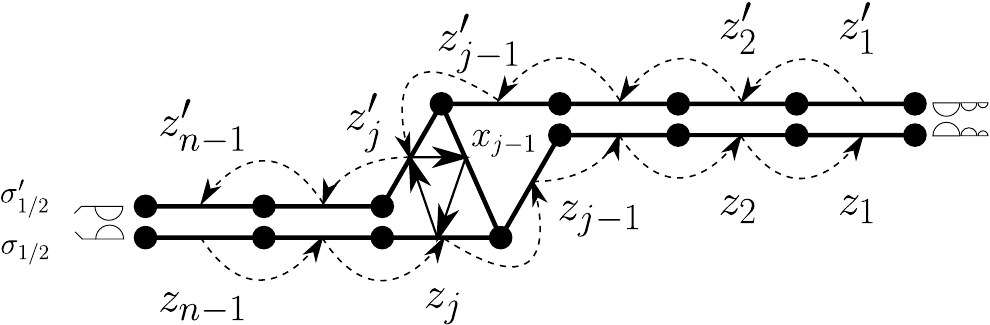}
	\caption{\small Diamond snake-move algebra ($j = 2, \dots, n-1$)}
	\label{fig:snake-move-algebra-left-diamond}
\end{figure}

\begin{figure}[htb]
	\centering
	\includegraphics[width=.66\textwidth]{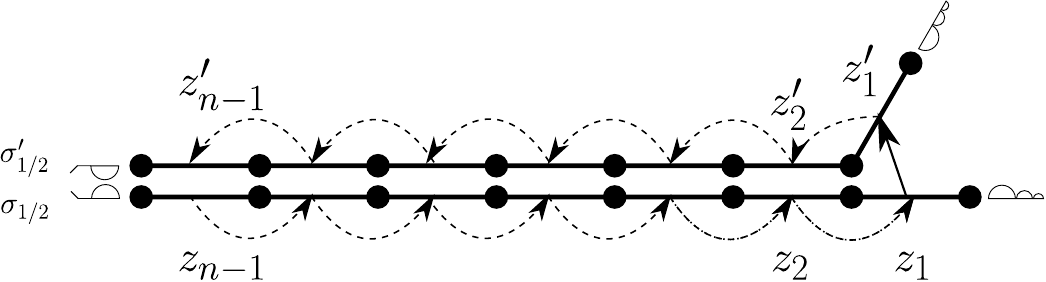}
	\caption{\small Tail snake-move algebra ($j = 1$)}
	\label{fig:snake-move-algebra-left-tail}
\end{figure}

\begin{conceptualremark}
\label{rem:snake-sweep-remark}
	The intention of this remark is to provide some guiding intuition for the upcoming constructions, and, strictly speaking, it is not required for the mathematical progression of the article.  
	
	The quiver of Figure \ref{fig:snake-move-algebra-left-tail} for the tail-move quantum torus is divided into a bottom and top side.  Similarly, the quiver of Figure \ref{fig:snake-move-algebra-left-diamond} for a diamond-move quantum torus has a bottom and top side, connected by a diagonal (where the variable $x_{j-1}$ is located).  As illustrated in the figures, we think of the bottom side (with un-primed generators $z_j$) as the top ``snake-half'' $\sigma_{1/2}$ of a snake $\sigma$ that has been ``split in half down its length''.  Similarly, we think of the top side (with primed generators $z_j^\prime$) as the bottom snake-half $\sigma^\prime_{1/2}$ of a split snake $\sigma^\prime$.  Compare Figure \ref{fig:snakes-theta-n-theta-n-1}, which illustrates a classical snake ``before~splitting''.  

This snake splitting can be seen more clearly in the \textit{quantum snake sweep} (see \S \ref{sec:embedding-FG-subalgebra} and  Figure \ref{fig:quantum-snake-sweep-example}  below) determined by the sequence of adjacent snakes $\sigma^\mathrm{bot} = \sigma^1, \sigma^2, \dots, \sigma^N = \sigma^\mathrm{top}$, where each snake $\sigma^i$ is split in half, so that each  snake-half forms a side in one of two adjacent snake-move quantum tori.  In the figure, the other halves (colored gray) of the bottom-most and top-most quantum snakes  can be thought of as either living in other triangles or not existing at all.  Prior to splitting a snake $\sigma$ in half, the snake consists of $n-1$ ``vertebrae'' connecting the $n$ snake-vertices $\sigma_k \in \Theta_{n-1}$.  Upon splitting the snake, the $j$-th vertebra splits into two generators $z_j$ and $z_j^\prime$ living in adjacent snake-move quantum tori.  
\end{conceptualremark}

	\subsection{Quantum snake-move matrices}

We turn to the key observation for the proof.  

\begin{proposition}
\label{lem:def-of-quantum-elementary-matrix-left}
	For $j=1,\dots,n-1$, the $j$-th quantum snake-move matrix
	\begin{equation*} 
	\label{eq:definition-of-quantum-elementary-matrices}
		\vec{M}_j 
		:=
		\left[
		\left( \prod_{k=1}^{n-1} \vec{S}^\mathrm{edge}_k(z_{k}) \right) 
		\vec{S}^\mathrm{left}_j(x_{j-1})
		\left( \prod_{k=1}^{n-1} \vec{S}^\mathrm{edge}_k(z'_{k}) \right)
		\right]
		\quad  
		\in  \mathrm{M}_n(\mathscr{S}^\omega_j)
	\end{equation*}
	is a $\mathscr{S}^\omega_j$-point of the quantum special linear group $\mathrm{SL}_n^q$.  That is, $\vec{M}_j  \in \mathrm{SL}_n^q(\mathscr{S}^\omega_j) \subset~\mathrm{M}_n(\mathscr{S}^\omega_j)$.  
\end{proposition}

Note the use of the Weyl quantum ordering; see \S \ref{sssec:weyl-ordering}.  Here, the matrices $\vec{S}_k^\mathrm{edge}(z)$ and $\vec{S}_j^\mathrm{left}(x)$ for $z, x$ in the commutative algebra $\mathscr{S}_j^1$ are defined as in \S \ref{sec:theorem-3}; see also \S \ref{sssec:weyl-ordering-continued}-\ref{sssec:quantum-left-and-right-matrices}.  Note when $j=1$, the matrix $\vec{S}^\mathrm{left}_1(x_{0}) = \vec{S}^\mathrm{left}_1$ is well-defined, despite $x_0$ not being defined.

	Proposition \ref{lem:def-of-quantum-elementary-matrix-left} is a direct calculation.  See Appendix \ref{sec:proofofsnakemoveslnq} for the proof.

For example, in the case $n=4$, $j=3$, the lemma says that the matrix
\begin{equation*}
	\vec{M}_3
	=
	\left[
z_1^{-\frac{1}{4}} z_2^{-\frac{2}{4}} z_3^{-\frac{3}{4}}	
\left(\begin{smallmatrix}
z_1 z_2 z_3&&&\\
&z_2 z_3&&\\
&&z_3&\\
&&&1
\end{smallmatrix}\right)
x_2^{-\frac{2}{4}}
\left(\begin{smallmatrix}
x_2&&&\\
&x_2&&\\
&&1&1\\
&&&1
\end{smallmatrix}\right)
z_1^{\prime-\frac{1}{4}} z_2^{\prime-\frac{2}{4}} z_3^{\prime-\frac{3}{4}}	
\left(\begin{smallmatrix}
z^\prime_1 z^\prime_2 z^\prime_3&&&\\
&z^\prime_2 z^\prime_3&&\\
&&z^\prime_3&\\
&&&1
\end{smallmatrix}\right)
	\right]  
\end{equation*}
is in $\mathrm{SL}_4^q(\mathscr{S}^\omega_3)$.

	\subsection{Technical step:  embedding a distinguished subalgebra \texorpdfstring{$\mathscr{T}_L$}{TLeft} of \texorpdfstring{$\mathscr{T}_n^\omega$}{Tnw} into a tensor product \texorpdfstring{$\bigotimes_{i=1}^{N-1} 
	\mathscr{S}^\omega_{j_i}$}{} of snake-move quantum tori}
	\label{sec:embedding-FG-subalgebra}

\begin{figure}[htb]
	\centering
	\includegraphics[scale=.57]{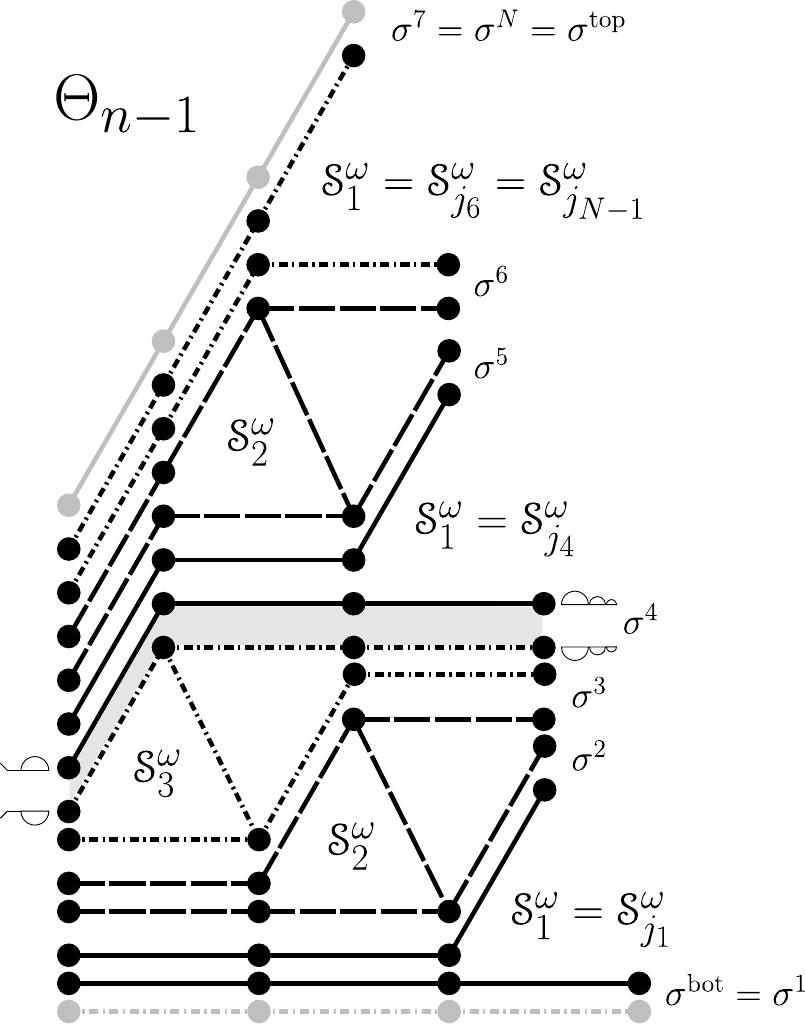}
	\caption{\small Quantum snake sweep ($n=4$); compare Figure \ref{fig:sequence-of-snakes-n=5-left}}
	\label{fig:quantum-snake-sweep-example}
\end{figure}

\begin{figure}[htb]
	\centering
	\includegraphics[width=.81\textwidth]{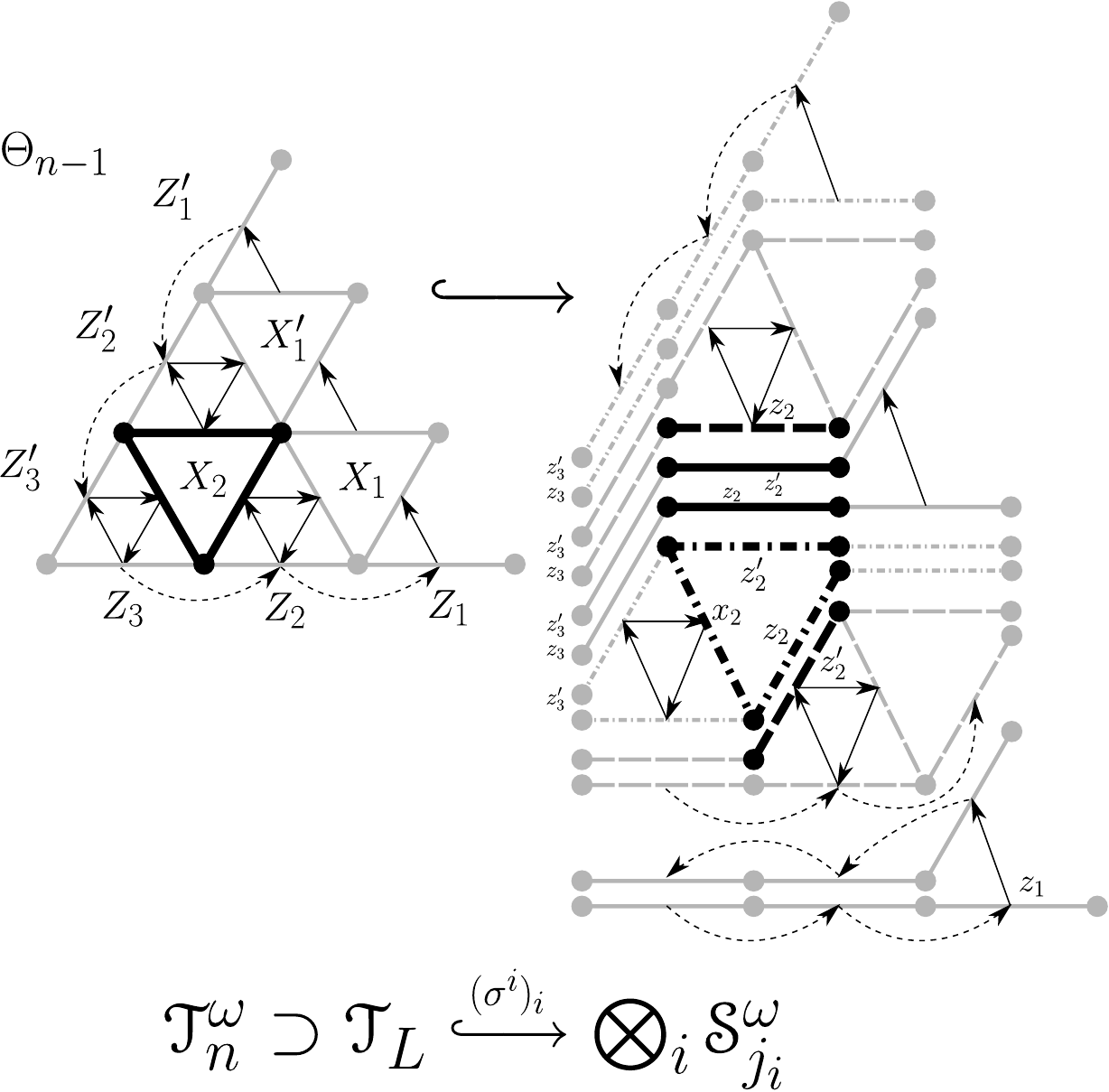}
	\caption{\small Embedding $\mathscr{T}_L$ in the tensor product of snake-move quantum tori}
	\label{fig:embedding-into-snake-move-algebra}
\end{figure}
		
For the snake-sequence $(\sigma^i)_{i=1,\dots,N}$, to each pair $(\sigma^i, \sigma^{i+1})$ of adjacent snakes we associate a snake-move quantum torus $\mathscr{S}^\omega_{j_i}$, recalling   Figure \ref{fig:quantum-snake-sweep-example} (see also Remark \ref{rem:snake-sweep-remark}).  Here $j_i$ corresponds to what was called $k$ in Definition \ref{def:adjacent-snakes}.  Recall the Fock-Goncharov quantum torus $\mathscr{T}_n^\omega$ (for example, Figure \ref{fig:n=4-left-and-right-matrices-example}).  

We now take a technical step.  Using the notation of Figures \ref{fig:left-and-right-quantum-matrices} and \ref{fig:n=4-left-and-right-matrices-example}, define $\mathscr{T}_L \subset \mathscr{T}_n^\omega$ (``$L$'' for ``Left'') to be the subalgebra generated by all the generators (and their inverses) of $\mathscr{T}_n^\omega$ except for $Z_1^{\prime\prime \pm 1/n}, \dots, Z_{n-1}^{\prime\prime \pm 1/n}$.  We claim that the snake-sequence $(\sigma^i)_i$  induces an embedding
\begin{equation*}
\label{eq:subalgebra-embedding}	
\mathscr{T}_L
	\overset{(\sigma^i)_i}{\longhookrightarrow}
	\bigotimes_{i=1}^{N-1} 
	\mathscr{S}^\omega_{j_i}
\end{equation*}
of algebras, realizing $\mathscr{T}_L \subset \mathscr{T}_n^\omega$ as a subalgebra of the tensor product of the  snake-move quantum tori $\mathscr{S}^\omega_{j_i}$ associated to the adjacent snake pairs $(\sigma^i, \sigma^{i+1})$.  Here, recall in general that the algebra structure for a  tensor product $A \otimes B$ of algebras $A$ and $B$ is defined by $(a \otimes b)\cdot (a^\prime \otimes b^\prime)=(a \cdot a^\prime) \otimes (b \cdot b^\prime)$ for all $a,a^\prime \in A$ and $b,b^\prime \in B$, extended linearly.

A more formal definition of the embedding will be given in \S \ref{ssec:formaldefofembedding} below.  We first explain the embedding through an example, in the setting $n=4$; see Figure \ref{fig:embedding-into-snake-move-algebra} (compare Figure \ref{fig:quantum-snake-sweep-example}).  

In this setting, the coordinate $X_2$ for instance (emphasized in Figure \ref{fig:embedding-into-snake-move-algebra}), is mapped to
\begin{equation*}
\label{eq:example-image-point-of-embedding}
	X_2 \mapsto 
	1 \otimes 
	z^\prime_2 \otimes
	z_2 x_2 z^\prime_2 \otimes
	z_2 z^\prime_2 \otimes
	z_2 \otimes
	1
	\quad\quad\quad \in \mathscr{S}^\omega_1 \otimes 
	\mathscr{S}^\omega_2 \otimes
	\mathscr{S}^\omega_3 \otimes
	\mathscr{S}^\omega_1 \otimes
	\mathscr{S}^\omega_2 \otimes
	\mathscr{S}^\omega_1.
\end{equation*} 
Similarly, the other coordinates $Z_1, Z^\prime_3, Z_2, Z_3, X_1, X_1^\prime, Z_2^\prime, Z_1^\prime$ are mapped to
\begin{align*}
	Z_1 &\mapsto z_1\otimes1\otimes1\otimes1\otimes1\otimes1,
&
	Z^\prime_3 &\mapsto 1 \otimes 1 \otimes z^\prime_3 \otimes z_3 z^\prime_3 \otimes z_3 z^\prime_3 \otimes z_3 z^\prime_3,
\\	Z_2 &\mapsto z_2 z_2^\prime \otimes z_2 \otimes 1 \otimes 1 \otimes 1 \otimes 1,
&
	Z_3 &\mapsto z_3 z_3^\prime \otimes z_3 z_3^\prime \otimes z_3 \otimes 1 \otimes 1 \otimes 1,
\\	X_1 &\mapsto z_1^\prime \otimes z_1 x_1 z_1^\prime \otimes z_1 z_1^\prime \otimes z_1 \otimes 1 \otimes 1,
&
	X_1^\prime &\mapsto 1 \otimes 1 \otimes 1 \otimes z_1^\prime \otimes z_1 x_1 z_1^\prime \otimes z_1,
\\	Z_2^\prime &\mapsto 1 \otimes 1 \otimes 1 \otimes 1 \otimes z_2^\prime \otimes z_2 z_2^\prime,
&
	Z_1^\prime &\mapsto 1 \otimes 1 \otimes 1 \otimes 1 \otimes 1 \otimes z_1^\prime.
\end{align*}

Note that the monomials (for instance, $z_2 x_2 z^\prime_2$ or $z_2 z_2^\prime$) appearing in the $i$-th tensor factor of the image of a generator $X$ or $Z$ of the subalgebra $\mathscr{T}_L$ under this mapping consist of mutually commuting generators $x$'s and/or $z$'s of the $i$-th snake-move quantum torus $\mathscr{S}_{j_i}^\omega$, so the order in which they are written is irrelevant. It is clear from Figure \ref{fig:embedding-into-snake-move-algebra} that these images satisfy the relations of $\mathscr{T}_L$.  In particular, the ``interior'' dotted arrows  lying at each interface between two snake-move quantum tori cancel each other out; note that, in Figure \ref{fig:embedding-into-snake-move-algebra}, we have omitted drawing some of these dotted arrows.  We gather that the mapping is well-defined and is an algebra homomorphism.  Injectivity follows from the property that every generator (that is, quiver edge) appearing on the right side of Figure \ref{fig:embedding-into-snake-move-algebra} corresponds to a unique generator on the left side.  Lastly, we technically should have defined the map on the formal $n$-roots of the coordinates of $\mathscr{T}_L$.  This is done in the obvious way, for instance,
\begin{equation*}
	X_2^{1/4} \mapsto 
	1 \otimes 
	z^{\prime 1/4}_2 \otimes
	z_2^{1/4} x_2^{1/4} z^{\prime 1/4}_2 \otimes
	z_2^{1/4} z^{\prime 1/4}_2 \otimes
	z_2^{1/4} \otimes
	1
	\quad \in \mathscr{S}^\omega_1 \otimes 
	\mathscr{S}^\omega_2 \otimes
	\mathscr{S}^\omega_3 \otimes
	\mathscr{S}^\omega_1 \otimes
	\mathscr{S}^\omega_2 \otimes
	\mathscr{S}^\omega_1.
\end{equation*}

	\subsubsection{Formal definition of the embedding}
	\label{ssec:formaldefofembedding}
	
	A \textit{segment} $\overline{\mu \nu}=\overline{\nu \mu}$ of the discrete triangle $\Theta_{n-1}$ is a line connecting neighboring vertices $\mu, \nu$.  Segments of the form $\overline{(\alpha,\beta,\gamma),(\alpha-1,\beta,\gamma+1)}$ are called \textit{horizontal}; segments of the form $\overline{(\alpha,\beta,\gamma),(\alpha-1,\beta+1,\gamma)}$ are called \textit{acute}; and, segments of the form $\overline{(\alpha,\beta,\gamma),(\alpha,\beta+1,\gamma-1)}$ are called \textit{obtuse}.  Let $\mathrm{Seg}_L$ denote the set of segments minus the obtuse segments with zero first coordinate.  For example, in the case $n=4$ the set $\mathrm{Seg}_L$ has $15$ elements; see the left hand side of Figure \ref{fig:embedding-into-snake-move-algebra}.  Let $\mathrm{Coord}_L \subset \mathscr{T}_L$ denote the set of coordinates, that is,
\begin{equation*}
	\mathrm{Coord}_L
	=
	\left\{X_{abc};\,\,(a,b,c)\in\mathrm{int}(\Theta_n)\right\}
	\cup
	\left\{Z_j;\,\, j=1,2,\dots,n-1\right\}
	\cup
	\left\{Z^\prime_j;\,\, j=1,2,\dots,n-1\right\}.
\end{equation*}
Note that the coordinates $X_{abc}$ correspond to the small downward facing triangles in $\Theta_{n-1}$, each of which is a union of a horizontal, acute, and obtuse segment in $\mathrm{Seg}_L$; the coordinates $Z_j$ correspond to horizontal segments with zero second coordinate; the coordinates $Z^\prime_j$ correspond to acute segments with zero third coordinate; and, each segment corresponds in this way to a unique coordinate.   In particular, there is a canonical surjective function $\pi^\prime : \mathrm{Seg}_L \longrightarrow \mathrm{Coord}_L \subset \mathscr{T}_L$.  See   the left hand side of Figure \ref{fig:embedding-into-snake-move-algebra} for the case $n=4$, where for example the three bold segments constitute the preimage $\pi^{\prime -1}(X_2)$.

Given a snake sequence $(\sigma^i)_{i=1,2,\dots,N}$, for each $1 \leq i \leq N-1$ let $\mathrm{Coord}^i \subset \mathscr{S}_{j_i}^\omega$ denote the set of coordinates in the snake-move quantum torus $\mathscr{S}^\omega_{j_i}$.  There are associated functions  $\varphi^i :  \mathrm{Coord}^i  \longrightarrow \mathrm{Coord}_L$, in general neither surjective nor injective, defined as follows.  To each coordinate $z_k \in \mathscr{S}_{j_i}^\omega$ for $k=1,2,\dots,n-1$ there is associated a segment $\mathrm{seg}(z_k) \in \mathrm{Seg}_L$ of the snake $\sigma^i$, namely $\mathrm{seg}(z_k)=\overline{\sigma_{k+1}^i \sigma_k^i }$; to each coordinate $z^\prime_k \in \mathscr{S}_{j_i}^\omega$ for $k=1,2,\dots,n-1$ there is associated a segment $\mathrm{seg}(z^\prime_k) \in \mathrm{Seg}_L$ of the snake $\sigma^{i+1}$, namely $\mathrm{seg}(z^\prime_k)=\overline{\sigma_{k+1}^{i+1} \sigma_k^{i+1} }$; and, to the coordinate $x_{j_i-1} \in \mathscr{S}_{j_i}^\omega$ there is associated an obtuse segment $\mathrm{seg}(x_{j_i-1}) \in \mathrm{Seg}_L$, which is not a segment of a snake, namely $\mathrm{seg}(x_{j_i-1})=\overline{\sigma^i_{j_i} \sigma^{i+1}_{j_i}}$.  Compare Figure \ref{fig:quantum-snake-sweep-example}.  Then the function $\varphi^i$ is defined by $\varphi^i(x)=\pi^\prime(\mathrm{seg}(x))$ for all $x \in \mathrm{Coord}^i$.  

For example in the case $n=4$, as illustrated in Figure \ref{fig:embedding-into-snake-move-algebra}, we have for instance $\varphi^1(z_1)=Z_1$, $\varphi^2(z_2^\prime)=X_2$, $\varphi^3(x_2)=X_2$, $\varphi^4(z_2^\prime)=X_2$, $\varphi^5(z_2)=X_2$, and $\varphi^6(z_3)=Z_3^\prime$.  

Finally, define the desired embedding   on generators $X^{1/n}$ of $\mathscr{T}_L$, for $X \in \mathrm{Coord}_L$, so that the image of $X^{1/n}$ in the tensor product $\bigotimes_i \mathscr{S}_{j_i}^\omega$ is the pure tensor defined by the property that its $i$-th factor is $\prod_{x \in (\varphi^i)^{-1}(X)} x^{1/n} \in \mathscr{S}^\omega_{j_i}$.  Note this is well-defined, since the coordinates $x \in (\varphi^i)^{-1}(X)$ in each preimage commute by design.  Note, by definition, if $(\varphi^i)^{-1}(X)$ is empty, then the product defining the $i$-th factor is $1$.  

In \S \ref{ssec:finishingtheproof}, we will make use of the surjective function $\pi : \bigcup_{i=1}^{N-1} \mathrm{Coord}^i \longrightarrow \mathrm{Coord}_L$ defined by $\pi(x)=\varphi^i(x)$ for $x \in \mathrm{Coord}^i$.

	\subsection{Finishing the proof}
	\label{ssec:finishingtheproof}

Comparing to the sketch of proof given in \S \ref{ssec:main-result}, we gather:
\begin{itemize}
	\item  $\vec{M}_\mathrm{FG} := \vec{L}^\omega
	\quad  \in  \mathrm{M}_n(\mathscr{T}_L) 
	\quad  \subset \mathrm{M}_n\left(\bigotimes_{i=1}^{N-1} 
	\mathscr{S}^\omega_{j_i}\right)$;
	\item  $\vec{M} := \vec{M}_{j_1} \vec{M}_{j_2} \cdots \vec{M}_{j_{N-1}}
	\quad  \in  \mathrm{SL}_n^q\left(\bigotimes_{i=1}^{N-1} 
	\mathscr{S}^\omega_{j_i}\right)
	\quad  \subset \mathrm{M}_n\left(\bigotimes_{i=1}^{N-1}	\mathscr{S}^\omega_{j_i}\right)$.  
\end{itemize}
To finish the proof, it remains to show
\begin{equation*}
\tag{$\ast\!\ast\!\ast$}
\label{eq:last-equation-theorem-1}
	\vec{M}_\mathrm{FG} \overset{?}{=} \vec{M}
	\quad  \in \mathrm{M}_n\left(\bigotimes_{i=1}^{N-1} 
	\mathscr{S}^\omega_{j_i}\right).
\end{equation*}

The strategy is to commute the many variables (as in the right hand side of Figure \ref{fig:embedding-into-snake-move-algebra}) appearing on the right hand side $\vec{M} = \prod_i \vec{M}_{j_i}$ (defined via Proposition \ref{lem:def-of-quantum-elementary-matrix-left}) of Equation \eqref{eq:last-equation-theorem-1}, until $\vec{M}$ has been put into the form of the left hand side $\vec{M}_\mathrm{FG}$ (defined via Definition \ref{def:left-and-right-quantum-matrices} followed by applying the embedding $\mathscr{T}_L \hookrightarrow \bigotimes_i \mathscr{S}^\omega_{j_i}$ of \S \ref{sec:embedding-FG-subalgebra}).  This is accomplished by applying the following two facts.

\begin{lemma}
\label{lem:technical-lemma} $ $
	\begin{enumerate}
	\item  If $\widetilde{\vec{M}}_1, \widetilde{\vec{M}}_2, \dots, \widetilde{\vec{M}}_{N-1}$ are $n \times n$ matrices with coefficients in $(q=\omega=\omega^{1/2}=1)$-specializations $\mathscr{T}^1_i$ of general quantum tori $\mathscr{T}^\omega_1, \mathscr{T}^\omega_2, \dots, \mathscr{T}^\omega_{N-1}$, viewed as factors in \hbox{$\mathscr{T}^\omega_1 \otimes \mathscr{T}^\omega_2 \otimes \cdots \otimes \mathscr{T}^\omega_{N-1}$}, then
\begin{equation*}
	\left[\widetilde{\vec{M}}_1\right]\left[\widetilde{\vec{M}}_2\right]\cdots\left[\widetilde{\vec{M}}_{N-1}\right] 
	=  \left[\widetilde{\vec{M}}_1 \widetilde{\vec{M}}_2 \cdots \widetilde{\vec{M}}_{N-1}\right]
	\quad  \in  \mathrm{M}_n\left(\mathscr{T}^\omega_1 \otimes \mathscr{T}^\omega_2 \otimes \cdots \otimes \mathscr{T}^\omega_{N-1}\right).
\end{equation*}
Here, we are viewing the tensor product $\mathscr{T}^\omega_1 \otimes \mathscr{T}^\omega_2 \otimes \cdots \otimes \mathscr{T}^\omega_{N-1}$ as a quantum torus in the obvious way, as demonstrated in the proof below.

\item  For commuting variables $z$ and $x$, the matrices $\vec{S}^\mathrm{edge}_k(z)$ and $\vec{S}^\mathrm{left}_{j}(x)$,  as in {\upshape\S  \ref{sec:theorem-3}},~satisfy
\begin{equation*}
	\vec{S}^\mathrm{edge}_k(z) \vec{S}^\mathrm{left}_{j}(x)
	=  \vec{S}^\mathrm{left}_{j}(x) \vec{S}^\mathrm{edge}_k(z)
	\quad  \text{ if and only if }
	\quad
	k \neq j.
\end{equation*}
\end{enumerate}
\end{lemma}

\begin{proof}
The proof of part (1) is straightforward.  To simplify the notation, we demonstrate the calculation for two matrices $\widetilde{\vec{M}}$ and $\widetilde{\vec{N}}$ with coefficients in classical tori $\mathscr{T}$ and $\mathscr{U}$ with coordinates $\{X_i\}_{i=1,2,\dots,m}$ and $\{Y_j\}_{j=1,2,\dots,p}$ and quivers $\epsilon$ and $\zeta$, respectively, where $\mathscr{T}$ and $\mathscr{U}$ are viewed in $\mathscr{T} \otimes \mathscr{U}$.  The proof for finitely many matrices is analogous.  

By linearity, it suffices to assume $\widetilde{\vec{M}}_{ij} \in \mathscr{T}$ and $\widetilde{\vec{N}}_{k\ell} \in \mathscr{U}$ are monomials, that is, $\widetilde{\vec{M}}_{ij}=X_1^{a^{ij}_1} X_2^{a^{ij}_2} \cdots X_m^{a^{ij}_m}$ and $\widetilde{\vec{N}}_{k\ell}=Y_1^{b^{k\ell}_1}Y_2^{b^{k\ell}_2} \cdots Y_p^{b^{k\ell}_p}$.  Recall that, by definition, different tensor factors commute under multiplication in $\mathscr{T} \otimes \mathscr{U}$.  We have for all $1 \leq i,j \leq n$
\begin{align*}
	&\left[ \widetilde{\vec{M}} \widetilde{\vec{N}} \right]_{ij} 
	= \left[ (\widetilde{\vec{M}} \widetilde{\vec{N}})_{ij} \right]
	= \sum_k \left[ \widetilde{\vec{M}}_{ik} \widetilde{\vec{N}}_{kj} \right]
	= \sum_k \left[ X_1^{a^{ik}_1} X_2^{a^{ik}_2}\cdots X_m^{a^{ik}_m} Y_1^{b^{kj}_1}Y_2^{b^{kj}_2}  \cdots Y_p^{b^{kj}_p} \right]
\\	
	&= \sum_k 
	q^{-\frac{1}{2}\left(
	\sum_{1 \leq \alpha < \beta \leq m} (\epsilon \otimes \zeta)_{\alpha \beta} a_\alpha^{ik} a_\beta^{ik} 
	+  \sum_{1 \leq \alpha \leq m; 1 \leq \beta \leq p} (\epsilon \otimes \zeta)_{\alpha(m + \beta)} a_\alpha^{ik} b_\beta^{kj} 
	+  \sum_{1 \leq \alpha < \beta \leq p} (\epsilon \otimes \zeta)_{(m+\alpha)(m+\beta)} b_\alpha^{kj} b_\beta^{kj}
	\right)}\ast
\\	&\ast X_1^{a^{ik}_1} X_2^{a^{ik}_2}\cdots X_m^{a^{ik}_m} Y_1^{b^{kj}_1}Y_2^{b^{kj}_2}  \cdots Y_p^{b^{kj}_p}
\\	&= \sum_k 
	q^{-\frac{1}{2}\left(
	\sum_{1 \leq \alpha < \beta \leq m} \epsilon_{\alpha \beta} a_\alpha^{ik} a_\beta^{ik} 
	+  \sum_{1 \leq \alpha \leq m; 1 \leq \beta \leq p} 0 a_\alpha^{ik} b_\beta^{kj} 
	+  \sum_{1 \leq \alpha < \beta \leq p} \zeta_{\alpha \beta} b_\alpha^{kj} b_\beta^{kj}
	\right)} X_1^{a^{ik}_1} X_2^{a^{ik}_2}\cdots X_m^{a^{ik}_m} Y_1^{b^{kj}_1}Y_2^{b^{kj}_2}  \cdots Y_p^{b^{kj}_p}
\\	&= \sum_k \left[ X_1^{a^{ik}_1} X_2^{a^{ik}_2}\cdots X_m^{a^{ik}_m} \right] \left[ Y_1^{b^{kj}_1}Y_2^{b^{kj}_2}  \cdots Y_p^{b^{kj}_p} \right]
	= \sum_k \left[ \widetilde{\vec{M}}_{ik} \right] \left[ \widetilde{\vec{N}}_{kj} \right]
	=\left(\left[ \widetilde{\vec{M}} \right] \left[ \widetilde{\vec{N}} \right]\right)_{ij}.
\end{align*}

The proof of part (2) is by inspection.
\end{proof}

\begin{proof}[Proof of Theorem {\upshape\ref{thm:first-theorem}}]
	By part (1) of Lemma \ref{lem:technical-lemma}, it suffices to establish Equation \eqref{eq:last-equation-theorem-1} when $q=\omega=\omega^{1/2}=1$, in which case we do not need to worry about the Weyl quantum ordering.  

It is helpful to introduce a simplifying notation.  For coordinates $z_k^{(i)}, x^{(i)}_{j}, z^{\prime(i)}_k \in \mathscr{S}_{j_i}^1$, put
\begin{equation*}
	\vec{Z}_k^{(i)} := \vec{S}_k^\mathrm{edge}(z_k^{(i)}), 
	\quad\quad
	\vec{X}_{j}^{(i)} := \vec{S}^\mathrm{left}_{j+1}(x_{j }^{(i)}),
	\quad\quad
	\vec{Z}_k^{\prime(i)} := \vec{S}_k^\mathrm{edge}(z^{\prime(i)}_k)
	\quad\quad
	\in \mathrm{M}_n(\mathscr{S}^1_{j_i}).  
\end{equation*}
In this new notation, the matrices $\vec{M}_{j_i} \in \mathrm{M}_n(\mathscr{S}^1_{j_i})$ of Proposition \ref{lem:def-of-quantum-elementary-matrix-left} can be expressed by
\begin{equation*}
	\vec{M}_{j_i} = \left( \prod_{k=1}^{n-1} \vec{Z}_k^{(i)} \right) \vec{X}^{(i)}_{j_i - 1} \left( \prod_{k=1}^{n-1} \vec{Z}_k^{\prime(i)} \right)
	\quad\quad  
\in \mathrm{M}_n(\mathscr{S}^1_{j_i})
\end{equation*}
and part (2) of Lemma \ref{lem:technical-lemma} now reads, for any $i_1, i_2 \in \{ 1, 2, \dots, N-1 \}$,
\begin{equation*}
\tag{$\dagger$}
\label{eq:reformulation-of-part2-of-lemma}
	\vec{Z}^{(i_1)}_k \vec{X}_j^{(i_2)} = \vec{X}_j^{(i_2)} \vec{Z}^{(i_1)}_k  
\in \mathrm{M}_n\left(\bigotimes_{i=1}^{N-1} 
	\mathscr{S}^1_{j_i}\right)
\quad
  \text{ if and only if } \quad  k \neq j+1
\quad  (\text{similarly for } \vec{Z} \to \vec{Z}^\prime).  
\end{equation*}

\textit{Example: n=2.}  In this case, $N=2$, we have $\mathscr{S}_{j_1}^1 = \mathscr{S}_1^1 \cong \mathscr{T}_L \subset \mathscr{T}_n^1$, and the embedding $\mathscr{T}_L \overset{\sim}{\hookrightarrow} \mathscr{S}_1^1$ is the identity, $Z_1 \mapsto z_1^{(1)}$, $Z_1^\prime \mapsto z_1^{\prime(1)}$.  Equation \eqref{eq:last-equation-theorem-1} is also trivial, reading
\begin{equation*}
\begin{split}
	\vec{M} = \vec{M}_1 = \vec{Z}_1^{(1)} \vec{X}_0^{(1)} \vec{Z}_1^{\prime(1)}
&=	z_1^{(1)\frac{-1}{2}} \left(\begin{smallmatrix} z_1^{(1)}&0\\0&1 \end{smallmatrix} \right)
	\left( \begin{smallmatrix} 1&1\\0&1 \end{smallmatrix} \right)
	z_1^{\prime(1)\frac{-1}{2}}  \left(\begin{smallmatrix} z^{\prime(1)}_1&0\\0&1 \end{smallmatrix} \right)
\\ &=	  Z_1^{-\frac{1}{2}} \left(\begin{smallmatrix} Z_1&0\\0&1 \end{smallmatrix} \right)
	\left( \begin{smallmatrix} 1&1\\0&1 \end{smallmatrix} \right)
	Z_1^{\prime -\frac{1}{2}} \left(\begin{smallmatrix} Z^\prime_1&0\\0&1 \end{smallmatrix} \right) =
\vec{S}^\mathrm{edge}_1(Z_1) \vec{S}^\mathrm{left}_1 \vec{S}^\mathrm{edge}_1(Z_1^\prime)
= \vec{M}_\mathrm{FG}.  
\end{split}
\end{equation*}

\textit{Example: n=3.}  Here $N = 4$, the subalgebra $\mathscr{T}_L$ has coordinates $Z_1, Z_2, X_1, Z_1^\prime, Z_2^\prime$, and the embedding $\mathscr{T}_L \hookrightarrow \mathscr{S}^1_1 \otimes \mathscr{S}^1_2 \otimes \mathscr{S}^1_1$  is defined by  (compare the $n=4$ case, Figure \ref{fig:embedding-into-snake-move-algebra})
\begin{equation*}
	Z_1 \mapsto z_1^{(1)}, 
\quad  Z_2 \mapsto z_2^{(1)} z_2^{\prime(1)} z_2^{(2)}, 
\quad  X_1 \mapsto z_1^{\prime(1)} z_1^{(2)} x_1^{(2)} z_1^{\prime(2)} z_1^{(3)}, 
\quad Z_1^\prime \mapsto z_1^{\prime(3)}, 
\quad  Z_2^\prime \mapsto z_2^{\prime(2)} z_2^{(3)} z_2^{\prime(3)}
\end{equation*}
where we have suppressed the tensor products.  Note in this case there is a unique snake-sequence $(\sigma^i)_{i=1, \dots, 4}$ so there is only one associated embedding of $\mathscr{T}_L$.  Equation \eqref{eq:last-equation-theorem-1} reads
\begin{equation*}
\begin{split}
	\vec{M} &= \vec{M}_1 \vec{M}_2 \vec{M}_1 = 
	\underline{\vec{Z}_1}^{(1)} \underline{\vec{Z}_2}^{(1)} \underline{\vec{X}_0}^{(1)} \vec{Z}_1^{\prime(1)} \vec{Z}_2^{\prime(1)} \cdot 
	\vec{Z}_1^{(2)} \vec{Z}_2^{(2)} \underline{\vec{X}_1}^{(2)} \vec{Z}_1^{\prime(2)} \vec{Z}_2^{\prime(2)} \cdot 
	\vec{Z}_1^{(3)} \vec{Z}_2^{(3)} \underline{\vec{X}_0}^{(3)} \underline{\vec{Z}_1^\prime}^{(3)} \underline{\vec{Z}_2^\prime}^{(3)} \\
	&=  
\underline{\vec{Z}_1}^{(1)}  \cdot
\underline{\vec{Z}_2}^{(1)}  \vec{Z}_2^{\prime(1)}  \vec{Z}_2^{(2)} \cdot
\underline{\vec{X}_0}^{(1)} \cdot
 \vec{Z}_1^{\prime(1)} \vec{Z}_1^{(2)} \underline{\vec{X}_1}^{(2)}
                       \vec{Z}_1^{\prime(2)} \vec{Z}_1^{(3)} \cdot
\underline{\vec{X}_0}^{(3)} \cdot
\underline{\vec{Z}_1^\prime}^{(3)}  \cdot
  \vec{Z}_2^{\prime(2)}  \vec{Z}_2^{(3)} 
\underline{\vec{Z}_2^\prime}^{(3)} \\
	&= \vec{S}^\mathrm{edge}_1(Z_1) \vec{S}^\mathrm{edge}_2(Z_2) \vec{S}_1^\mathrm{left} \vec{S}^\mathrm{left}_2(X_1) \vec{S}^\mathrm{left}_1 \vec{S}^\mathrm{edge}_1(Z^\prime_1) \vec{S}^\mathrm{edge}_2(Z^\prime_2) = \vec{M}_\mathrm{FG}
\end{split}
\end{equation*}
where for the third equality we have used the reformulation \eqref{eq:reformulation-of-part2-of-lemma} of part (2) of Lemma \ref{lem:technical-lemma} to commute the matrices.  Note that the ordering of terms in any of the seven groupings in the fourth expression is immaterial.  The fourth equality uses the embedding $\mathscr{T}_L \hookrightarrow \mathscr{S}^1_1 \otimes \mathscr{S}^1_2 \otimes \mathscr{S}^1_1$.  

\textit{General case.}  As we saw in the examples, the expression $\vec{M} = \prod_{i=1}^{N-1} \vec{M}_{j_i}$ is a product of distinct terms of the form $\vec{Z}_k^{(i)}$, $\vec{X}_j^{(i)}$, or $\vec{Z}_k^{\prime(i)}$.  Let $A$ be the set of terms, that is, \hbox{$A=\cup_{i=1,2,\dots,N-1} \{ \vec{Z}_k^{(i)}, \vec{X}^{(i)}_{j_i - 1}, \vec{Z}^{\prime (i)}_k; \quad k=1,2,\dots n-1 \}$}.
Besides terms of the form $\vec{X}_0^{(i)}$, there is one term in $A$ for each coordinate $z_k^{(i)}, x_j^{(i)}, z^{\prime(i)}_k$ of $\bigotimes_{i=1}^{N-1} \mathscr{S}^1_{j_i}$.  We show that there is an algorithm that commutes these terms into the correct groupings, as in the above examples.  

There is a distinguished subset $A_L \subset A$, precisely defined in the next paragraph.  In the example $n=2$, $A_L = A$, and in the example $n=3$, the terms in $A_L$ are underlined above.  All the $\vec{X}_0^{(i)}$ terms are in $A_L$.  Besides the $\vec{X}_0^{(i)}$ terms, there is one term in $A_L$ for each coordinate $Z_k, X_j, Z^\prime_k$ of $\mathscr{T}_L$; see  Figure \ref{fig:embedding-into-snake-move-algebra}.  As another example, for $n=4$ and our usual preferred snake sequence $(\sigma^i)_i$, then $A_L = \{ \vec{Z}_1^{(1)}, \vec{Z}_2^{(1)}, \vec{Z}_3^{(1)}, \vec{X}_0^{(1)}, \vec{X}_1^{(2)}, \vec{X}_2^{(3)}, \vec{X}_0^{(4)}, \vec{X}_1^{(5)}, \vec{X}_0^{(6)}, \vec{Z}_1^{\prime(6)}, \vec{Z}_2^{\prime(6)}, \vec{Z}_3^{\prime(6)} \}$; see Figures~\ref{fig:quantum-snake-sweep-example},~\ref{fig:embedding-into-snake-move-algebra}.  

More precisely, the general definition of $A_L \subset A$, valid for any snake sequence $(\sigma^i)_{i=1,2,\dots,N}$, is as follows.  First, $\vec{Z}^{(1)}_k=\vec{S}_k^\mathrm{edge}(z_k^{(1)})$, $\vec{Z}^{\prime(N-1)}_k=\vec{S}_k^\mathrm{edge}(z^{\prime(N-1)}_k)$, and $\vec{X}^{(i)}_0=\vec{S}_1^\mathrm{left}$ are in $A_L$ for all $k=1,2,\dots,n-1$ and for all $1 \leq i \leq N-1$ such that $j_i - 1=0$.  And $\vec{X}^{(i)}_{j_i - 1} = \vec{S}^\mathrm{left}_{j_i}(x_{j_i -1 }^{(i)})$ is in $A_L$ for all $1 \leq i \leq N-1$ such that $j_i > 1$.

Recall that the injectivity of the embedding $\mathscr{T}_L \hookrightarrow \bigotimes_{i=1}^{N-1} \mathscr{S}^1_{j_i}$ followed immediately from the property that every coordinate $z_k^{(i)}, x_j^{(i)}, z^{\prime(i)}_k$ of $\bigotimes_{i=1}^{N-1} \mathscr{S}^1_{j_i}$ corresponds to a unique coordinate $Z_k, X_j, Z^\prime_k$ of $\mathscr{T}_L$; see Figure \ref{fig:embedding-into-snake-move-algebra}.  This property thus defines a retraction $r : A \twoheadrightarrow A_L$, namely a surjective function restricting to the identity on $A_L \subset A$ (by definition, $\vec{X}_0^{(i)} \mapsto \vec{X}_0^{(i)}$).  See the next paragraph for a precise definition.  The retraction $r$ can be visualized as collapsing the right side of Figure \ref{fig:embedding-into-snake-move-algebra} to obtain the left side.  

More precisely, in the notation of \S \ref{ssec:formaldefofembedding}, there is a bijection 
$
	f : A -A_0 \longrightarrow \bigcup_{i=1}^{N-1} \mathrm{Coord}^i
$
defined by $f(\vec{Z}^{(i)}_k)=z_k^{(i)}$, $f(\vec{Z}^{\prime(i)}_k)=z^{\prime(i)}_k$, and $f(\vec{X}^{(i)}_{j_i - 1})=x_{j_i -1 }^{(i)}$.  Here, we have put $A_0 = \{ \vec{X}^{(i)}_0; \quad 1 \leq i \leq N-1\text{ such that } j_i - 1 = 0 \}$.  Observe, by definition of $A_L$, that the restricted composition $g$ defined by
$
g = \pi \circ (f|_{A_L - A_0}) : A_L - A_0 \longrightarrow \mathrm{Coord}_L
$
is a bijection, where  $\pi : \bigcup_{i=1}^{N-1} \mathrm{Coord}^i \to \mathrm{Coord}_L$ is defined at the end of \S \ref{ssec:formaldefofembedding}.  The retraction $r : A \to A_L$ is defined on  $A-A_0$ by $r = g^{-1} \circ \pi \circ f$, and as the identity on $A_0 \subset A_L$.  

The desired algorithm grouping the terms in $A$, where there is one grouping per term in $A_L$, is defined by selecting an ungrouped term $a \in A$ and commuting it left or right until it is adjacent to $r(a) \in A_L$.  Here, the terms are viewed in the expression for $\vec{M}$.  This commutation is possible by part (2) of Lemma \ref{lem:technical-lemma}, that is, \eqref{eq:reformulation-of-part2-of-lemma}.  

More precisely, in the expression for $\vec{M}$ at step $s$ of the algorithm, for each $a_0 \in A_L$ let $\ell(a_0,s)$ denote the length of the longest chain of adjacent terms $a \in r^{-1}(a_0)$ such that this chain contains $a_0$.  For instance, in the $n=3$ example above, for $a_0 = \vec{X}_1^{(2)}$, initially the length of the chain containing $a_0$ is $2$, while at the end of the algorithm this length is $5=|r^{-1}(a_0)|$.  Assuming for the moment that the algorithm is well-defined, that is, that the commutation is possible, we see that $\ell(a_0, s) \leq \ell(a_0, s+1)$ for all $a_0 \in A_L$ and for all steps $s$, and moreover that at least one of these inequalities is strict at each step.  It follows that the algorithm terminates, at which point the length $\ell(a_0,s_\mathrm{term})$ of the chain containing $a_0$ is $|r^{-1}(a_0)|$ for all $a_0 \in A_L$.  Thus, in the expression for $\vec{M}$ at the end of the algorithm, replacing each string $\prod_{a \in r^{-1}(a_0)} a$ with $\vec{S}^\mathrm{edge}_k(Z_k)$, $\vec{S}^\mathrm{left}_{j+1}(X_j)$, or $\vec{S}^\mathrm{edge}_k(Z^\prime_k)$, depending on $a_0$, completes the proof.  It only remains to show that the commutations at each step of the algorithm are possible.  Only the diagonal matrices $\vec{Z}^{(i)}_k$, or $\vec{Z}^{\prime (i)}_k$, are ``moving'' during the commutation, and these matrices commute with each other.  So, by \eqref{eq:reformulation-of-part2-of-lemma}, we just need to argue that, upon commuting $a=\vec{Z}^{(i)}_k$, say,  until it is adjacent to $a_0=r(a)$, we do not need to commute $\vec{Z}^{(i)}_k$ past any $\vec{X}^{(i^{\prime})}_{k-1}$.  For concreteness, assume $a_0$ is of the form $\vec{X}^{(i^{\prime\prime})}_k$ with $i^{\prime\prime} \leq i$.  The argument is analogous in the cases where $a_0$ is of the form $\vec{X}^{(i^{\prime\prime})}_k$ with $i^{\prime\prime} > i$, or $\vec{Z}_k^{(1)}$, or $\vec{Z}_k^{\prime(N-1)}$.  The claim is clear when $i^{\prime\prime}=i$, so assume $i^{\prime\prime} < i$, so, in particular, $\vec{Z}^{(i)}_k$ is being commuted to the left until it is just to the right of $\vec{X}^{(i^{\prime\prime})}_k$.  Note such a $\vec{Z}^{(i)}_k$ appears as a horizontal edge lying over the top of the small downward facing triangle corresponding to $\vec{X}^{(i^{\prime\prime})}_k$; compare Figure \ref{fig:embedding-into-snake-move-algebra}.  In the notation of \S \ref{ssec:formaldefofembedding}, the horizontal segment $\mathrm{seg}(f(\vec{Z}^{(i)}_k)) \in \mathrm{Seg}_L$ in the discrete triangle $\Theta_{n-1}$ is of the form $\overline{(k, \beta, n-1-k-\beta)(k-1, \beta, n-k-\beta)}$.  Thus, the key observation is that if some snake-move matrix $\vec{M}_{j_{i^{\prime}}}=\vec{M}_k$ contributes $\vec{X}^{(i^{\prime})}_{k-1}$ to $\vec{M}$, then either $(1)$ the bottom snake $\sigma^{i^\prime}$ of the $i^\prime$-snake-move is later in the snake sequence than the bottom snake $\sigma^i$ of the $i$-snake-move, in particular $i \leq i^\prime$; or, (2) the top snake $\sigma^{i^\prime+1}$ of the $i^\prime$-snake-move is earlier in the snake sequence than the bottom snake $\sigma^{i^{\prime\prime}}$ of the $i^{\prime\prime}$-snake-move, in particular $i^\prime + 1 \leq i^{\prime\prime}$.  In the former case, $\vec{X}^{(i^{\prime})}_{k-1}$ lies to the right of $\vec{Z}^{(i)}_k$, and in the latter case $\vec{X}^{(i^{\prime})}_{k-1}$ lies to the left of $\vec{X}^{(i^{\prime\prime})}_k$.  
\end{proof}

	\subsection{Setup for the quantum right matrix}
	\label{sec:setup-for-the-quantum-right-matrix}

We end with a few words about the proof for the quantum right matrix $\vec{M}_\mathrm{FG} = \vec{R}^\omega$, which essentially goes the same as for the left matrix.

\begin{figure}[htb]
	\centering
	\includegraphics[width=.66\textwidth]{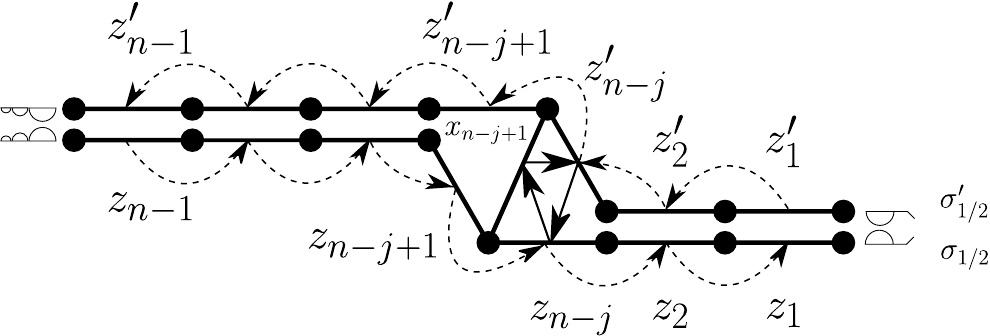}
	\caption{\small Right diamond snake-move algebra ($j = 2, \dots, n-1$)}
	\label{fig:snake-move-algebra-right-diamond}
\end{figure}

\begin{figure}[htb]
	\centering
	\includegraphics[width=.66\textwidth]{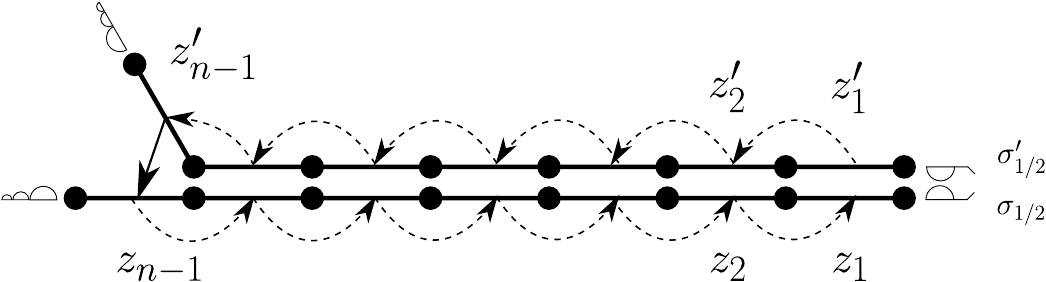}
	\caption{\small Right tail snake-move algebra ($j = 1$)}
	\label{fig:snake-move-algebra-right-tail}
\end{figure}

	\textit{(i)}  The right version of the $j$-th snake algebra $\mathscr{S}^\omega_j$ for $j=1,2,\dots,n-1$ is given by replacing the quivers of Figures \ref{fig:snake-move-algebra-left-diamond} and \ref{fig:snake-move-algebra-left-tail} by the quivers shown in Figures \ref{fig:snake-move-algebra-right-diamond} and \ref{fig:snake-move-algebra-right-tail}.  

	\textit{(ii)}  The $j$-th quantum snake-move matrix $\vec{M}_j$ of Proposition \ref{lem:def-of-quantum-elementary-matrix-left} is replaced by
\begin{equation*}
		\vec{M}_j 
		:=
		\left[
		\left( \prod_{k=1}^{n-1} \vec{S}^\mathrm{edge}_k(z_{k}) \right) 
		\vec{S}^\mathrm{right}_j(x_{n-j+1})
		\left( \prod_{k=1}^{n-1} \vec{S}^\mathrm{edge}_k(z'_{k}) \right)
		\right]
		\quad  
		\in  \mathrm{M}_n(\mathscr{S}^\omega_j).
\end{equation*}
Note, when $j=1$, the matrix $\vec{S}^\mathrm{right}_1(x_{n}) = \vec{S}^\mathrm{right}_1$ is well-defined, despite $x_n$ not being~defined.  
 
	\textit{(iii)}  The subalgebra $\mathscr{T}_R \subset \mathscr{T}_n^\omega$ is generated by all but the $Z^{\prime \pm 1/n}_j\text{'s}$; see Figures \ref{fig:left-and-right-quantum-matrices} and~\ref{fig:n=4-left-and-right-matrices-example}.

\appendix

\section{Proof of Proposition \ref{lem:def-of-quantum-elementary-matrix-left}}
\label{sec:proofofsnakemoveslnq}

\begin{lemma}\label{lem:prodofweylequals1}
If $ZW=q^\epsilon WZ$ in some quantum torus $\mathscr{T}$, and if $\sum_{i=1}^m r_i = 0$, then
\begin{equation*}
\prod_{i=1}^m [Z^{r_i} W^{r_i}]=1 \quad \in \mathscr{T}.
\end{equation*}
\end{lemma}
\begin{proof}
Using $(\sum_i r_i)^2/2=\sum_i r_i^2/2 + \sum_{i<j} r_i r_j$, we compute
\begin{gather*}
\prod_i [Z^{r_i} W^{r_i}]=q^{-\epsilon\sum_i r_i^2/2} Z^{r_1} W^{r_1} Z^{r_2} W^{r_2} \cdots Z^{r_m} W^{r_m}
\\  = q^{-\epsilon\sum_i r_i^2/2} q^{-\epsilon \sum_{i<j} r_i r_j} Z^{\sum_i r_i} W^{\sum_i r_i} = q^{-\epsilon \left( \sum_i r_i \right)^2/2} \cdot Z^0 \cdot W^0 = 1.  \qedhere
\end{gather*}
\end{proof}

\begin{proof}[Proof of Proposition {\upshape\ref{lem:def-of-quantum-elementary-matrix-left}}]
As a shorthand, put $L_{i\ell} := \left(\vec{S}^\mathrm{left}_j(x_{j-1})\right)_{i\ell}$ and $\widetilde{E}_{ii} := \prod_{k=1}^{n-1} \left(\vec{S}^\mathrm{edge}_k(z_k)\right)_{ii}$ and $\widetilde{E}^\prime_{ii} := \prod_{k=1}^{n-1} \left(\vec{S}^\mathrm{edge}_k(z^\prime_k)\right)_{ii}$.  By Definition \ref{def:pointsofmnq}, and by the structure of the matrix $\vec{M}_j$, the following three relations are needed to establish that $\vec{M}_j$ is in $\mathrm{M}_n^q(\mathscr{S}_j^\omega)$:
\begin{enumerate}
\item\label{item:appen1firstequation}  
$\left[  \widetilde{E}_{jj} L_{j(j+1)} \widetilde{E}^\prime_{(j+1)(j+1)}  \right]  
\left[  \widetilde{E}_{jj} L_{jj} \widetilde{E}^\prime_{jj}  \right]
=
q
\left[  \widetilde{E}_{jj} L_{jj} \widetilde{E}^\prime_{jj}  \right]
\left[  \widetilde{E}_{jj} L_{j(j+1)} \widetilde{E}^\prime_{(j+1)(j+1)}  \right]$;
\item\label{item:appen1secondequation}  
$\left[  \widetilde{E}_{(j+1)(j+1)} L_{(j+1)(j+1)} \widetilde{E}^\prime_{(j+1)(j+1)} \right]
\left[  \widetilde{E}_{jj} L_{j(j+1)} \widetilde{E}^\prime_{(j+1)(j+1)}  \right]$\\
$=
q
\left[  \widetilde{E}_{jj} L_{j(j+1)} \widetilde{E}^\prime_{(j+1)(j+1)}  \right] 
\left[  \widetilde{E}_{(j+1)(j+1)} L_{(j+1)(j+1)} \widetilde{E}^\prime_{(j+1)(j+1)} \right]$;
\item\label{item:appen1thirdequation}
$\left[  \widetilde{E}_{ii} L_{ii} \widetilde{E}^\prime_{ii} \right]
\left[  \widetilde{E}_{kk} L_{kk} \widetilde{E}^\prime_{kk} \right]
=
\left[  \widetilde{E}_{kk} L_{kk} \widetilde{E}^\prime_{kk} \right]
\left[  \widetilde{E}_{ii} L_{ii} \widetilde{E}^\prime_{ii} \right]$ for $i < k$.
\end{enumerate}

We begin with Equation \eqref{item:appen1firstequation}.  Note $L_{j(j+1)}=L_{jj}=x_{j-1}^{(1-j)/n}$, and
$\left[  \widetilde{E}_{jj} L_{j(j+1)} \widetilde{E}^\prime_{(j+1)(j+1)}  \right]  
=
\left[  \widetilde{E}_{jj} L_{jj} \widetilde{E}^\prime_{jj} z_j^{\prime -1}  \right]$.  So it suffices to show that commuting $z_j^{\prime -1}$ from left to right across $\widetilde{E}_{jj} L_{jj} \widetilde{E}^\prime_{jj}$ contributes a factor $q$, equivalently, $z_j^\prime$ contributes $q^{-1}$.  Indeed, in $\widetilde{E}_{jj} L_{jj} \widetilde{E}^\prime_{jj}$, we see $z_j^\prime$ only interacts with $x_{j-1}^{(1-j)/n}$ with weight $q^2$; with $\left(\vec{S}^\mathrm{edge}_j(z_j)\right)_{jj}=z_j^{(n-j)/n}$ with weight $q^{-2}$; with $\left(\vec{S}^\mathrm{edge}_{j+1}(z^\prime_{j+1})\right)_{jj}=z_{j+1}^{\prime (n-j-1)/n}$ with weight $q$; and, with $\left(\vec{S}^\mathrm{edge}_{j-1}(z^\prime_{j-1})\right)_{jj}=z_{j-1}^{\prime (1-j)/n}$ with weight $q^{-1}$.  The total exponent of $q$ that $z^\prime_j$ contributes is therefore
$(2*(1-j)-2*(n-j)+1*(n-j-1)-1*(1-j))/n=-1$.

Next we check Equation \eqref{item:appen1secondequation}.  Note $L_{(j+1)(j+1)}=L_{j(j+1)}=x_{j-1}^{(1-j)/n}$, and
$\left[  \widetilde{E}_{jj} L_{j(j+1)} \widetilde{E}^\prime_{(j+1)(j+1)}  \right]
=
\left[  z_j \widetilde{E}_{(j+1)(j+1)} L_{(j+1)(j+1)} \widetilde{E}^\prime_{(j+1)(j+1)} \right]$.  So it suffices to show that commuting $z_j$ from right to left across $\widetilde{E}_{(j+1)(j+1)} L_{(j+1)(j+1)} \widetilde{E}^\prime_{(j+1)(j+1)}$ contributes a factor $q$.  Indeed, in $\widetilde{E}_{(j+1)(j+1)} L_{(j+1)(j+1)} \widetilde{E}^\prime_{(j+1)(j+1)} $, we see $z_j$ only interacts with $x_{j-1}^{(1-j)/n}$ with weight $q^2$ (because it's moving from right to left); with $\left(\vec{S}^\mathrm{edge}_j(z_j^\prime)\right)_{(j+1)(j+1)}=z_j^{\prime -j/n}$ with weight $q^{-2}$; with $\left(\vec{S}^\mathrm{edge}_{j+1}(z_{j+1})\right)_{(j+1)(j+1)}=z_{j+1}^{ (n-j-1)/n}$ with weight $q$; and, with $\left(\vec{S}^\mathrm{edge}_{j-1}(z_{j-1})\right)_{(j+1)(j+1)}=z_{j-1}^{ (1-j)/n}$ with weight $q^{-1}$.  The total exponent of $q$ that $z_j$ contributes is therefore 
$(2*(1-j)-2*(-j)+1*(n-j-1)-1*(1-j))/n=+1$.

Lastly we verify Equation \eqref{item:appen1thirdequation}.  Note that the terms in $\left[  \widetilde{E}_{ii} L_{ii} \widetilde{E}^\prime_{ii} \right]$ appear in the forms $x_{j-1}^{\alpha^i}$ or $z_\ell^{\beta^i_\ell} z_\ell^{\prime\beta^i_\ell}$ for $\ell=1,2,\dots,n-1$.  
We see from the quivers in Figures \ref{fig:snake-move-algebra-left-diamond} and \ref{fig:snake-move-algebra-left-tail} that terms of this form mutually commute.  So 
\begin{equation*}
\left[  \widetilde{E}_{ii} L_{ii} \widetilde{E}^\prime_{ii} \right] = \left[ x_{j-1}^{\alpha^i} \right] \prod_\ell \left[ z_\ell^{\beta^i_\ell} z_\ell^{\prime\beta^i_\ell} \right]
\end{equation*} 
where the right hand side is independent of the ordering of the terms.    Similarly for $\left[  \widetilde{E}_{kk} L_{kk} \widetilde{E}^\prime_{kk} \right]$.  It follows that $\left[  \widetilde{E}_{ii} L_{ii} \widetilde{E}^\prime_{ii} \right]$ commutes with $\left[  \widetilde{E}_{kk} L_{kk} \widetilde{E}^\prime_{kk} \right]$ for all $i,k$.

It remains to check that the quantum determinant of $\vec{M}_j$ is equal to $1 \in \mathscr{S}_j^\omega$.  Since $\vec{M}_j$ is in $\mathrm{M}_n^q(\mathscr{S}_j^\omega)$ and is triangular, by Remark \ref{rem:remarksaboutquantummatrices}(\ref{item:simplifiedqdet}) we have $\mathrm{Det}^q(\vec{M}_j)=\prod_i (\vec{M}_j)_{ii}$.  As the only $\ell$ such that $z_\ell$ does not commute with $z^\prime_\ell$ is $\ell=j$, the above equation~becomes
\begin{equation*}
(\vec{M}_j)_{ii}=\left[  \widetilde{E}_{ii} L_{ii} \widetilde{E}^\prime_{ii} \right] = \left[ z_j^{\beta^i_j} z_j^{\prime\beta^i_j} \right] x_{j-1}^{\alpha^i} \prod_{\ell\neq j} \left(z_\ell z^\prime_\ell \right)^{\beta^i_\ell}.
\end{equation*} 
Note $\sum_i \alpha^i = 0$ and $\sum_i \beta^i_\ell = 0$ for all $\ell=1,2,\dots,n-1$ by construction of $\vec{M}_j$ (this is where the normalizing factors come in; compare the example below Proposition \ref{lem:def-of-quantum-elementary-matrix-left}).  It follows that (where the last equality is by Lemma \ref{lem:prodofweylequals1})
\begin{equation*}
\mathrm{Det}^q(\vec{M}_j)
=\left( \prod_i \left[ z_j^{\beta^i_j} z_j^{\prime\beta^i_j} \right]\right) \left(\prod_i x_{j-1}^{\alpha^i} \prod_{\ell\neq j} \left(z_\ell z^\prime_\ell \right)^{\beta^i_\ell}\right)
=\left(\prod_i \left[ z_j^{\beta^i_j} z_j^{\prime\beta^i_j} \right]\right)\cdot 1=1. \qedhere
\end{equation*}
\end{proof}

\bibliographystyle{alpha}
\bibliography{references.bib}

\end{document}